\documentclass[10pt, reqno]{amsart}
\usepackage{amsfonts,amssymb}
\usepackage{graphicx}

\newtheorem{theorem}{Theorem}[section]
\newtheorem{lemma}[theorem]{Lemma}
\newtheorem{corollary}[theorem]{Corollary}
\newtheorem{proposition}[theorem]{Proposition}
\newtheorem{assumption}[theorem]{Assumption}

\theoremstyle{definition}
\newtheorem{definition}[theorem]{Definition}
\newtheorem{example}[theorem]{Example}
\newtheorem{remark}[theorem]{Remark}

\numberwithin{equation}{section}

\newcommand{\R}{{\mathbb R}}
\newcommand{\C}{{\mathbb C}}

\renewcommand{\Re}{\operatorname{Re}}

\newcommand{\sign}{\operatorname{sign}}

\newcommand{\figref}[1]{Figure~\ref{#1}}
\renewcommand{\epsilon}{\varepsilon}
\newcommand{\supp}{\operatorname{supp}}
\newcommand{\Emb}{\operatorname{Emb}}
\newcommand{\e}{\operatorname{e}}
\newcommand{\with}{\,:\,}

\newcommand{\dd}{{\mathrm{d}}}
\newcommand{\cL}{{\mathcal{L}}}

\renewcommand{\labelenumi}{\roman{enumi})}
\begin{document}

\title[Maximal regularity for operators with mixed boundary conditions]
{Maximal parabolic regularity for divergence operators including mixed boundary conditions}

\author{Robert Haller-Dintelmann}
\address{Technische Universit\"at Darmstadt, Fachbereich
Mathematik, Schlossgartenstr. 7, D-64298 Darmstadt, Germany}
\email{haller@mathematik.tu-darmstadt.de}

\author{Joachim Rehberg}
\address{Weierstrass Institute for Applied Analysis and Stochastics,
 Mohrenstr. 39, D-10117 Berlin, Germany}
\email{rehberg@wias-berlin.de}

\date{}

\begin{abstract}
We show that elliptic second order operators $A$ of divergence type fulfill
maximal parabolic regularity on distribution spaces, even if the underlying
domain is highly non-smooth, the coefficients of $A$ are discontinuous and $A$
is complemented with mixed boundary conditions. Applications to quasilinear
parabolic equations with non-smooth data are presented.
\end{abstract}

\maketitle
\section{Introduction}
It is known that divergence operators fulfill maximal parabolic regularity on
$L^p$ spaces -- even if the underlying domain is non-smooth, the coefficients
are discontinuous and the boundary conditions are mixed, see \cite{arendt} and
also \cite{hie/reh}. This provides a powerful tool for the treatment of linear
and nonlinear parabolic equations in $L^p$ spaces, see \cite{pruess, cle/li,
mazelresc, hie/reh}. The only disadvantage of this concept is that the
appearing Neumann conditions have to be homogeneous and that distributional
right hand sides (e.g. surface densities) are not admissible. Confronted with
these phenomena, it seems an adequate alternative to consider the equations in
distribution spaces, what we will do in this paper. Pursuing this idea, one
has, of course, to prove that the occurring elliptic operators satisfy
parabolic regularity on those spaces in an appropriate sense.

In fact, we show that, under very mild conditions on the domain $\Omega$, the
Dirichlet boundary part $\partial \Omega \setminus \Gamma$ and the coeffcient
function, elliptic divergence operators with real
$L^\infty$-coefficients satisfy maximal parabolic regularity
on a huge variety of spaces, among which are Sobolev, Besov and
Lizorkin-Triebel spaces, provided that the differentiability index is between
$0$ and $-1$ (cf. Theorem~\ref{t-randterm}). We consider this as the first
main result of this work, also interesting in itself. Up to now, the only
existing results for mixed boundary conditions in distribution spaces (apart
from the Hilbert space situation) are, to our knowledge, that of Gr\"oger
\cite{groegerpar} and the recent one of Griepentrog \cite{griepar}. Concerning
the Dirichlet case, compare \cite{byun} and references therein.

Having this first result at hand, the second aim of this work is the treatment
of quasilinear parabolic equations of the formal type
\begin{equation} \label{e-quasiformal}
  \left\{ \begin{aligned}
        \bigl (\mathcal F(u) \bigr )'- \nabla \cdot \mathcal G(u) \mu
                \nabla u &= \mathcal R(t,u), \\
        u(T_0) &= u_0,
        \end{aligned}\right.
\end{equation}
combined with mixed, nonlinear boundary conditions:
\begin{equation} \label{e-boundcond77}
  \nu \cdot \mathcal G(u) \mu \nabla u + b(u) = g \text{ on } \Gamma \qquad
        \text{and} \qquad u = 0 \text{ on } \partial \Omega \setminus \Gamma.
\end{equation}
Let us point out some ideas, which will give a certain guideline for the
paper: Our analysis is based on a regularity result for the square root
$(- \nabla \cdot \mu \nabla )^{1/2}$ on $L^p$ spaces. It has already been
remarked in the introduction of \cite{ausch/tcha01} that estimates between
$\|(-\nabla \cdot \mu \nabla)^{1/2} f \|_p$ and $\|\nabla f \|_p$ should
provide powerful tools for the treatment of elliptic and parabolic problems
involving divergence form operators. It seems, however, that this idea has not
yet been developed to its full strength, cf. \cite[Ch.~5]{e/r/s}.

Originally, our strategy for proving maximal parabolic regularity for
divergence operators on $H^{-1,q}_\Gamma$ was to show an analog of the
central result of \cite{ausch/tcha01}, this time in case of mixed boundary
conditions, namely that
\begin{equation} \label {e-topwurziso}
  \bigl( -\nabla \cdot \mu \nabla + 1 \bigr)^{-1/2}: L^q \to H^{1,q}_\Gamma 
\end{equation}
provides a topological isomorphism for suitable $q$. This would give the
possibility of carrying over the maximal parabolic regularity, known for $L^q$,
 to the dual of $H^{1,q'}_\Gamma$, because, roughly spoken, $( -\nabla \cdot \mu \nabla + 1
)^{-1/2}$ commutes with the corresponding parabolic solution operator.
Unfortunately, we were only able to prove the continuity of
\eqref{e-topwurziso} within the range $q \in [2,\infty[$, due to a result of
Duong and $\rm M^c$Intosh \cite{DM99}, but did not succeed in proving the
continuity of the inverse in general. Let us explicitely mention that the
proof of the isomorphism property of \eqref{e-topwurziso} would be a great
achievement. In particular, this would allow here to avoid the localization
procedure we had to introduce in Section~\ref{sec-Consequences} in order to
prove maximal parabolic regularity, and to generalize our results to higher
dimensions. The isomorphism property is known for the Hilbert space case $L^2$
(see \cite{mcintosh}) in case of mixed boundary conditions and even complex
coefficients, but the proof fundamentally rests on the Hilbert space structure,
so that we do not see a possibility of directly generalizing this to the $L^p$
case.

It turns out, however, that \eqref{e-topwurziso} provides a topological
isomorphism, if $\Omega \cup \Gamma$ is the image under a volume-preserving and
bi-Lipschitz mapping of one of Gr\"oger's model sets \cite{groeger89},
describing the geometric configuration in neighborhoods of boundary points of
$\Omega$. Thus, in these cases one may carry over the maximal parabolic
regularity from $L^q$ to $H^{-1,q}_\Gamma.$ Knowing this, we localize the
linear parabolic problem, use the 'local' maximal parabolic information and
interpret this again in the global context at the end. Interpolation with the
$L^p$ result then yields maximal parabolic regularity on the corresponding
interpolation spaces.

Let us explicitely mention that the concept of Gr\"oger's regular sets, where
the domain itself is a Lipschitz domain, seems adequate to us, because it
covers many realistic geometries that fail to be domains with Lipschitz
boundary. The price one has to pay is that the problem of optimal elliptic
regularity becomes much more delicate and, additionally, trace theorems for
this situation are scarcely to be found in the literature.

\medskip

The strategy for proving that \eqref{e-quasiformal}, \eqref{e-boundcond77}
admit a unique local solution is as follows. We reformulate
\eqref{e-quasiformal} into a usual quasilinear equation, where the time
derivative directly affects the unknown function. Assuming additionally that
the elliptic operator $-\nabla \cdot \mu \nabla+1: H^{1,q}_\Gamma \to
H^{-1,q}_\Gamma$ provides a topological isomorphism for a $q $ larger than the
space dimension $d$, the existence and uniqueness results for abstract
quasilinear equations of Pr\"uss (see \cite{pruess}, see also \cite{cle/li})
apply to the resulting quasilinear parabolic equation. The detailed discussion
how to assure all requirements of \cite{pruess}, including the adequate choice
of the Banach space, is presented in Section~6. The crucial point is that the
linear elliptic operator which corresponds to the initial value satisfies
maximal parabolic regularity, which has been proved before. Let us further
emphasize that the presented setting allows for coefficient functions that
really jump at hetero interfaces of the material and permits mixed boundary
conditions, as well as domains which do not possess a Lipschitz boundary, see
Section~\ref{sec-Examples}. It is well known that this is required when
modelling real world problems, see e.g. \cite{sommer2, cars} for problems from
thermodynamics or \cite{franz, bere} concerning biological models. Last but
not least, heterostructures are the determining features of many fundamental
effects in semiconductors, see for instance \cite{selberherr84, band, kupon}.

One further advantage is that nonlinear, nonlocal boundary conditions are
admissible in our concept, despite the fact that the data is highly
non-smooth, compare \cite{amannonlinearb}. The calculus of maximal parabolic
$L^s(\left] T_0, T \right[;X)$ regularity is preferable to the concept of
H\"older continuity in time, because it allows for reaction terms $\mathcal R$
which discontinously depend on time. This is important in many examples (see
\cite{{ung/troe}, {heink/troe}, krej}), in particular in the control theory of
parabolic equations. Alternatively, the reader should think e.g. of a
manufacturing process for semiconductors, where light is switched on/off at a
sharp time point and, of course, parameters in the chemical process then
change abruptly. It is remarkable that, nevertheless, the solution is
H\"older continuous simultaneously in space and time, see
Corollary~\ref{c-hoel} below.

We finish these considerations by looking at the special case of semilinear problems. It turns out that here satisfactory results may be achieved even
without the additional continuity condition on $- \nabla \cdot \mu \nabla + 1$
mentioned above, see Corollary~\ref{t-semil}.

\medskip

In Section~\ref{sec-Examples} we give examples for geometries, Dirichlet
boundary parts and coefficients in three dimensions for which our additional
supposition, the isomorphy $-\nabla \cdot \mu \nabla + 1 : H^{1,q}_\Gamma \to
H^{-1,q}_\Gamma$ really holds for a $q > d$. In Subsection~\ref{subsec-balks}
we take a closer look at the special geometry of two crossing beams, which
provides a geometrically easy example of a domain $\Omega$ that does not have
a Lipschitz boundary and thus cannot be treated by former theories, but which
is covered by our results.

Finally, some concluding remarks are given in Section~\ref{sec-Remarks}.
%

%
%
\section{Notation and general assumptions}
%
%
%
%
Throughout this article the following assumptions are valid.
\begin{itemize}
\item $\Omega \subseteq \R^d$ is a bounded Lipschitz domain and $\Gamma$ is an
        open subset of $\partial \Omega$.
\item The coefficient function $\mu$ is a Lebesgue measurable, bounded
        function on $\Omega$ taking its values in the set of real, symmetric,
        positive definite $d \times d$ matrices, satisfying the usual
        ellipticity condition.
\end{itemize}
\begin{remark} \label{r-defnlip}
Concerning the notions 'Lipschitz domain' and 'domain with Lipschitz boundary'
(synonymous: strongly Lipschitz domain) we follow the terminology of Grisvard
\cite{grisvard85}, see also \cite{mazsob}.
\end{remark}
For $\varsigma \in \left] 0, 1\right]$ and $1 < q < \infty$ we define
$H^{\varsigma,q}_\Gamma (\Omega)$ as the closure of
\begin{equation} \label{e-defC}
   C^\infty_\Gamma(\Omega) := \{ \psi|_\Omega : \psi \in
        C^\infty(\R^d), \; \supp(\psi) \cap (\partial \Omega \setminus \Gamma )
        = \emptyset \}
\end{equation}
in the Sobolev space $H^{\varsigma, q}(\Omega)$. Of course, if $\Gamma =
\emptyset$, then $H_\Gamma^{\varsigma, q}(\Omega) =
H_0^{\varsigma, q}(\Omega)$ and if $\Gamma = \partial \Omega$, then
$H_\Gamma^{\varsigma, q}(\Omega) = H^{\varsigma, q}(\Omega)$. This last point
follows from the fact that $\Omega$, as a Lipschitz domain, admits a
continuous extension operator from $H^{1,q}(\Omega)$ into $H^{1,q}(\R^d)$, see
\cite[Thm.~3.10]{giusti}. Thus, the set $C^\infty(\Omega) := \{
\psi|_\Omega : \psi \in C^\infty(\R^d) \}$ is dense in $H^{1,q}(\Omega)$.
Concerning the dual of $H^{\varsigma,q}_\Gamma(\Omega)$, we have to
distinguish between the space of linear and the space of anti-linear forms on this space. We define $H^{-\varsigma,q}_\Gamma(\Omega)$ as the space of
continuous, linear forms on $H^{\varsigma,q'}_\Gamma(\Omega)$ and $\breve H^{-\varsigma,q}_\Gamma(\Omega)$
as the space of
anti-linear forms on $H^{\varsigma,q'}_\Gamma(\Omega)$ if $1/q + 1/q' =1$.
Note that $L^p$ spaces may be viewed as part of $\breve H^{-\varsigma,q}_\Gamma$ for
suitable $\varsigma, q$ via the identification of an element $f\in L^p$ with
the anti-linear form $H^{\varsigma,q'}_\Gamma \ni \psi \mapsto \int_\Omega f
\overline{\psi} \, \dd \mathrm x$.

If misunderstandings are not to be expected, we drop the $\Omega$ in the
notation of spaces, i.e. function spaces without an explicitely given domain
are to be understood as function spaces on $\Omega$.

By $K$ we denote the open unit cube in $\R^d$, by $K_-$ the lower half cube
$K \cap \{\mathrm x \with x_d < 0 \}$, by $\Sigma = {K} \cap \{
\mathrm x  \with x_d = 0 \}$ the upper plate of $K_-$ and by $\Sigma_0$
the left half of $\Sigma$, i.e. $\Sigma_0 = \Sigma \cap \{\mathrm x \with
x_{d-1} < 0 \}$.

As in the preceding paragraph, we will throughout the paper use $\mathrm{x},
\mathrm{y},\dots$ for vectors in $\R^d$, whereas the components of
$\mathrm{x}$ will be denoted by italics $x_1, x_2, \dots, x_d$ or in three
dimensions also by $x, y, z$.

If $B$ is a closed operator on a Banach space $X$, then we denote by $dom_X(B)$
the domain of this operator. $\mathcal L(X,Y)$ denotes the space of linear,
continuous operators from $X$ into $Y$; if $X = Y$, then we abbreviate
$\mathcal L(X)$. Furthermore, we will write $\langle \cdot, \cdot \rangle_{X'}$
for the dual pairing of elements of $X$ and the space $X'$ of anti-linear
forms on $X$.

Finally, the letter $c$ denotes a generic constant, not always of the same
value.
%
%
%
%
\section{Preliminaries}
%
%
%
%
In this section we will properly define the elliptic divergence operator and
afterwards collect properties of the $L^p$ realizations of this operator which
will be needed in the subsequent chapters. First of all we establish the
following extension property for function spaces on Lipschitz domains, which
will be used in the sequel.
%
\begin{proposition} \label{p-extend}
There is a continuous extension operator $\mathrm{Ext} : L^1(\Omega) \to
L^1(\R^d)$, whose restriction to any space $H^{1,q}(\Omega)$ ($q \in
\left] 1, \infty \right[$) maps this space continuously into $H^{1,q}(\R^d)$.
Moreover, $\mathrm{Ext}$ maps $L^p(\Omega)$ continuously into $L^p(\R^d)$ for
$p \in \left] 1, \infty \right]$.
\end{proposition}
%
\begin{proof}
The assertion is proved for the spaces $H^{1,q}$ in \cite[Thm.~3.10]{giusti}
see also \cite[Ch.~1.1.16]{mazsob}. Inspecting the corresponding proofs (which
are given via localization, Lipschitz diffeomorphism and symmetric reflection)
one easily recognizes that the extension mapping at the same time continuously
extends the $L^p$ spaces.
\end{proof}
Let us introduce an assumption on $\Omega$ and $\Gamma$ which will define the
geometrical framework relevant for us in the sequel.
%
\begin{assumption} \label{a-groegerregulaer}
\begin{itemize}
\item[a)] For any point $\mathrm x \in \partial \Omega$ there is an open
        neighborhood $\Upsilon_{\mathrm{x}}$ of $\mathrm x$ and a
        bi-Lipschitz mapping $\phi_{\mathrm{x}}$ from $\Upsilon_\mathrm x$
        into $\R^d$, such that $\phi_{\mathrm{x}}(\mathrm{x}) = 0$
	and $\phi_{\mathrm{x}} \bigl( (\Omega \cup \Gamma)
        \cap \Upsilon_{\mathrm{x}} \bigr) = \alpha K_-$ or $\alpha (K_- \cup
        \Sigma)$ or $\alpha (K_- \cup \Sigma_0)$ for some positive $\alpha =
        \alpha(\mathrm x)$.
\item[b)] Each mapping $\phi_{\mathrm{x}}$ is, in addition, volume-preserving.
\end{itemize}
\end{assumption}
%
%
\begin{remark} \label{r-groegerreg}
Assumption~\ref{a-groegerregulaer}~a) exactly characterizes Gr\"oger's
regular sets, introduced in his pioneering paper \cite{groeger89}. Note that
the additional property 'volume-preserving' also has been required in several
contexts (see \cite{ggkr} and \cite{groegerpar}).

It is not hard to see that every Lipschitz domain and also its closure is
regular in the sense of Gr\"oger, the corresponding model sets are then $K_-$
or $K_-\cup \Sigma$, respectively, see \cite[Ch~1.2]{grisvard85}. A simplifying topological
 characterization of Gr\"oger's regular sets for $d=2$ and $d=3$ will be given in
 Section~\ref{sec-Remarks}.

In particular, all domains with Lipschitz boundary (strongly Lipschitz
domains) satisfy Assumption~\ref{a-groegerregulaer}: if, after a shift and an 
orthogonal transformation, the domain lies locally beyond a graph of a Lipschitz
 function $\psi$, then one can define $\phi(x_1, \ldots, x_d) = 
(x_1 - \psi(x_2, \ldots, x_d) ,x_2,
\ldots, x_d)$. Obviously, the mapping $\phi$ is then bi-Lipschitz and the
determinant of its Jacobian is identically $1$. For further examples see Section~\ref{sec-Examples}.
\end{remark}
%
Next we have to introduce a boundary measure on $\partial \Omega$. Since in
our context $\Omega$ is not necessarily a domain with Lipschitz boundary, this
is not canonic. Let, according to the definition of a Lipschitz domain, for
every point $\mathrm x \in \partial \Omega$ an open neighborhood
$\Upsilon_\mathrm x$ of $\mathrm x$ and a bi-Lipschitz function
$\phi_\mathrm x : \Upsilon_\mathrm x \to \R^d$ be given, which satisfy
$\phi_\mathrm x(\Upsilon_\mathrm x \cap \Omega) = K_-$, $\phi_\mathrm x
( \Upsilon_\mathrm x \cap \partial \Omega ) = \Sigma$ and $\phi_\mathrm x
(\mathrm x) = \mathrm 0$. Let $\Upsilon_{\mathrm x_1}, \ldots,
\Upsilon_{\mathrm x_l}$ be a finite subcovering of $\partial \Omega$. Define
on $\partial \Omega \cap \Upsilon_{\mathrm x_j}$ the measure $\sigma_j$ as the
$\phi_{\mathrm x_j}^{-1}$-image of the $(d-1)$-dimensional Lebesgue measure on
$\Sigma$. Clearly, this measure is a positive, bounded Radon measure. Finally,
define the measure $\sigma$ on $\partial \Omega$ by
\[ \int_{\partial \Omega}  f \, d\sigma := \sum_{j=1}^l \int_{\partial \Omega
        \cap \Upsilon_{\mathrm x_j}} f \,d\sigma_j, \quad f \in C(\partial \Omega).
\]
Clearly, $\sigma$ also is a bounded, positive Radon measure. Furthermore, it is
not hard to see that the measure $\sigma$ -- simultaneously viewed as a
measure on $\R^d$ -- satisfies
\[ \sup_{\mathrm x \in \R^d} \sup_{r \in \left] 0, 1 \right[}
        \sigma(B(\mathrm x, r)) r^{1-d} < \infty,
\]
where, here and in the sequel, $B(\mathrm x, r)$ denotes the ball centered at 
$\mathrm x$ with radius $r$, compare \cite[Ch.~II.1]{jons}, in particular Example 1 there.

Later we will repeatedly need the following interpolation results from
\cite{ggkr}.
%
\begin{proposition} \label{p-interpol}
Let $\Omega$ and $\Gamma$ satisfy Assumption~\ref{a-groegerregulaer} a) and let
$\theta \in \left] 0, 1 \right[$.
\begin{enumerate}
\item Then for $q_0, q_1
	\in \left] 1, \infty \right[$ and $\frac 1q = \frac{1-\theta}{q_0} +
	\frac{\theta}{q_1}$ one has
\begin{alignat}{2} \label{e-interpol01}
  H_\Gamma^{\theta,q} &= \bigl[ L^{q_0},H_\Gamma^{1,q_1} \bigr]_{\theta}, \quad
	&& \text{if } \theta \neq \frac{1}{q}, \\
  \label{e-interpol02}
	H_\Gamma^{-\theta,q} &= \bigl[ L^{q_0},H_\Gamma^{-1,q_1}
	\bigr]_{\theta} \quad && \text{if } \theta \neq 1 - \frac{1}{q}
  \intertext{and}
  \label{e-interpol03}
  H^{\pm 1,q}_\Gamma &= \bigl[ H^{\pm 1,q_0}_\Gamma,
	H^{\pm 1,q_1}_\Gamma]_\theta.
\end{alignat}
\item
If additionally Assumption~\ref{a-groegerregulaer} b) is fulfilled and
$\frac{1}{q} \neq \theta \neq 1 - \frac{1}{q}$, then
\begin{equation} \label{e-interpol04}
  H_\Gamma^{\pm \theta,q} = \bigl[ H_\Gamma^{-1,q}, H_\Gamma^{1,q}
	\bigr]_\frac{1\pm \theta}{2}.
\end{equation}
\end{enumerate}
\end{proposition}
\begin{corollary} \label{c-antili}
Under the same assumptions as for \eqref{e-interpol03} one has
\begin{equation}  \label{e-interpol0003}
  \breve H^{-1,q}_\Gamma = \bigl[ \breve H^{- 1,q_0}_\Gamma,
 \breve	H^{-1,q_1}_\Gamma]_\theta.
\end{equation}
\end{corollary}
\begin{proof}
\eqref{e-interpol0003} may be deduced from \eqref{e-interpol03} by means of the retraction/coretraction
theorem (see \cite[Ch.~1.2.4]{triebel}), where the coretraction is the mapping which assigns to
$f \in \breve H^{-1,r}_\Gamma$ the linear form $ H^{1,r'}_\Gamma \ni \psi \to \langle f,
\overline \psi\rangle _{\breve H^{-1,r}_\Gamma}$. 
\end{proof}
%
Having this at hand, we can prove the following trace theorem.
%
\begin{theorem} \label{t-embedbound}
Assume $q \in \left] 1, \infty \right[$ and $\theta \in \bigl] \frac {1}{q}, 1\bigr[$.
 Let $\Pi$ be a Lipschitz hypersurface in $\overline{\Omega}$ and let $\varpi$ be any measure
 on $\Pi$ which satisfies
\[ \sup_{\mathrm x \in \R^d} \sup_{r \in \left] 0, 1\right[}
        \varpi(B(\mathrm x, r)) r^{1-d} < \infty.
\]
Then the trace operator $\mathrm{Tr}$ from $H^{\theta,q}(\Omega)$ to $L^{q}(\Pi,\varpi)$ is
continuous.
\end{theorem}
%
\begin{proof}
Since $\Omega$ is an extension domain for $H^{1,q}$ and $L^q$ simultaneously,
one has the inequality
\begin{equation} \label{e-mazyaS68}
  \| u|_{\Pi} \|_{L^q(\Pi,\varpi)} = \| u \|_{L^q(\overline \Omega,\varpi)}
        \le c \| u \|_{H^{1,q}(\Omega)}^{1/q} \| u \|_{L^{q}(\Omega)}^{1-1/q}
        \le c \| u \|_{H^{1,q}(\Omega)}, \quad u \in H^{1,q}(\Omega),
\end{equation}
for $q \in \left] 1, \infty \right[$, see \cite[Ch.~1.4.7]{mazsob}. But due to
a general interpolation principle (see \cite[Ch.~5, Prop.~2.10]{bennet})
this yields a continuous mapping
\begin{equation} \label{e-interppol1}
  \bigl( L^q(\Omega), H^{1,q}(\Omega) \bigr)_{\frac {1}{q},1} \ni u \mapsto
        u|_{\Pi} \in L^q(\Pi,\varpi).
\end{equation}
Since $\Omega$ is a Lipschitz domain, \eqref{e-interpol01} in particular
yields the equality $H^{\theta,q}(\Omega) = [ L^q(\Omega), H^{1,q}(\Omega)
]_\theta$ in view of $\theta > 1/q$. Thus, we have the continuous embedding
\[ H^{\theta,q}(\Omega) = \bigl[ L^q(\Omega), H^{1,q}(\Omega) \bigr]_\theta
        \hookrightarrow \bigl( L^q(\Omega), H^{1,q}(\Omega)
        \bigr)_{\frac {1}{q},1},
\]
see \cite[Ch.~1.10.3, Thm.~1 and Ch.~1.3.3]{triebel}. This, together with
\eqref{e-interppol1}, proves the theorem.
\end{proof}
We define the operator $A : H^{1,2}_\Gamma \to \breve H^{-1,2}_\Gamma$ by 
\begin{equation} \label{e-defellip}
  \langle A\psi, \varphi\rangle_{\breve H^{-1,2}_\Gamma} := \int_\Omega \mu
	\nabla \psi \cdot \nabla \overline \varphi \; \dd \mathrm{x} +
	\int_\Gamma \varkappa \, \psi\, \overline \varphi \; \dd \sigma, \quad
	\psi, \varphi \in H^{1,2}_\Gamma,
\end{equation}
where $\varkappa \in L^\infty(\Gamma, \dd\sigma)$. Note that in view
of \eqref{e-mazyaS68} the form in \eqref{e-defellip} is well defined.

In the special case $\varkappa = 0$, we write more suggestively $-\nabla \cdot
\mu \nabla$ instead of $A$.

The $L^2$ realization of $A$, i.e. the maximal restriction of $A$ to the space $L^2$, we
 denote by the same symbol $A$; clearly this is identical with the operator which is 
induced by the form on the right hand side of \eqref{e-defellip}. If $B$ is a selfadjoint
operator on $L^2$, then by the $L^p$ realization of $B$ we mean its restriction to $L^p$ if
 $p > 2$ and the $L^p$ closure of $B$ if $p \in \left[1, 2 \right[$.

We decided not to use different symbols for all these (and lateron also other) realizations
 of our operators in this paper, since we think that the gain in
exacteness would be largely outweighed by the resulting complexity of
notation. Naturally, this means that we have to pay attention to domains even
more thoroughly.

\begin{remark} \label{r-antili}
Following \cite[Ch.~1.4.2]{Ouh05} (see also \cite[Ch.~1]{berez}), we did not
define $A$ as an operator with values in the space of linear forms on
$H^{1,2}_\Gamma$, but in the space of anti-linear forms. This guarantees
that the restriction of this operator to $L^2$ equals the usual selfadjoint
operator that is induced by the sesquilinear form in \eqref{e-defellip}, which
is crucial for our analysis.
In this spirit, the duality between $\breve H^{-1,q}_\Gamma$ and
$H^{1,q'}_\Gamma$ is to be considered as the extended $L^2$ duality $L^2 \times L^2 \ni
(\psi,\varphi) \to \int_\Omega \psi \overline \varphi \; d\mathrm x$, where
$L^2$ acts as the set of anti-linear forms on itself.
Especially, all occurring adjoint operators are to be understood with respect to this
dual pairing.
\end{remark}

First, we collect some basic facts on $A$.
%
\begin{proposition} \label{p-basicl2}
\begin{enumerate}
\item $\nabla \cdot \mu \nabla$ generates an analytic semigroup on
        $\breve H^{-1,2}_\Gamma$.
\item $-\nabla \cdot \mu \nabla$ is selfadjoint on $L^2$ and bounded by $0$
        from below. The restriction of $- A$ to $L^2$ is densely defined and
        generates an analytic semigroup there.
\item If $\lambda > 0$ then the operator $( -\nabla \cdot \mu \nabla + \lambda
        )^{1/2} : H^{1,2}_\Gamma \to L^2$ provides a topological isomorphism;
        in other words: the domain of  $( -\nabla \cdot \mu \nabla + \lambda
        )^{1/2}$ on $L^2$ is the form domain $H^{1,2}_\Gamma$.
\item The form domain $H^{1,2}_\Gamma$ is invariant under multiplication with
        functions from $H^{1,q}$, if $q > d$.
\item Assume $\varkappa \ge 0$. Then, under Assumption \ref{a-groegerregulaer} a), for
       all $p \in \left] 1, \infty \right[$ the operator $-A$
       generates a semigroup of contractions on $L^p$. Additionally, it
	satisfies
\[  \| (A + \lambda)^{-1}
        \|_{\mathcal L(L^p)} \le \frac{c}{|\lambda|},
        \quad \Re \lambda \ge 0.
\]

\item Under Assumption \ref{a-groegerregulaer} a) $dom_{\breve H^{-1,q}_\Gamma}(-\nabla \cdot \mu \nabla)$
 embeds compactly into $\breve H^{-1,q}_\Gamma$ for every $q \in [2,\infty[$,
 i.e. the resolvent of $(-\nabla \cdot \mu \nabla)$ is compact on $\breve H^{-1,q}_\Gamma$.
\end{enumerate}
\end{proposition}
%
\begin{proof}
\begin{enumerate}
\item is proved in \cite[Thm.~1.55]{Ouh05}, see also \cite{groe/reh}.
\item The first assertion follows from a classical representation theorem for
        forms, see \cite[Ch.~VI.2.1]{kato}. Secondly, one verifies that the
        form $ H^{1,2}_\Gamma \ni \psi \mapsto \int_\Gamma \varkappa |\psi|^2
        \, \dd \sigma $ is form subordinated to the  -- positive -- form
        $H^{1,2}_\Gamma \ni \psi \mapsto  \int_\Omega \nabla \psi \cdot \mu
        \nabla \overline \psi + \psi \overline \psi \; \dd \mathrm x$ with
        arbitrarily small relative bound. In fact, thanks to
        \eqref{e-mazyaS68},
        \begin{align*}
          \biggl| \int_\Gamma \varkappa |\psi|^2 \dd \sigma \biggr| &\le
                \| \varkappa \|_{L^\infty(\Gamma)} \| \psi
                \|_{L^2(\partial \Omega)}^2 \le \| \varkappa
                \|_{L^\infty(\Gamma)} \| \psi \|_{H^{1,2}_\Gamma(\Omega)}
                \| \psi \|_{L^2(\Omega)} \\
          &\le \epsilon \| \psi \|^2_{H^{1,2}_\Gamma(\Omega)} +
                \frac{1}{\epsilon} \| \varkappa \|^2_{L^\infty(\Gamma)}
                \|\psi\|^2_{L^2(\Omega)}.
        \end{align*}
        Thus, the form \eqref{e-defellip} is also closed on $H^{1,2}_\Gamma$
        and sectorial. Moreover, the operator $-A$ generates an analytic
        semigroup by the representation theorem for sectorial forms, see also
        \cite[Ch.~VI.2.1]{kato}.
\item This follows from the second representation theorem of forms (see
        \cite[Ch.~VI.2.6]{kato}), applied to the operator $-\nabla \cdot \mu
        \nabla + \lambda$.
\item First, for $u \in C^\infty_\Gamma$ and $v \in C^\infty$ the product $u
        v$ is obviously in $C^\infty_\Gamma \subseteq H^{1,2}_\Gamma$. But,
        by definition of $H^{1,2}_\Gamma$, the set $C^\infty_\Gamma$ (see
        \eqref{e-defC}) is dense in $H^{1,2}_\Gamma$ and $C^\infty$ is dense
        in $H^{1,q}$. Thus, the assertion is implied by the continuity of the
        mapping
        \[ H^{1,2}_\Gamma \times H^{1,q} \ni (u,v) \mapsto uv \in H^{1,2},
        \]
        because $H^{1,2}_\Gamma$ is closed in $H^{1,2}$.
\item This is proved in \cite[Thm.~4.11, Thm.~5.2]{gkr}.
\item The operator $(-\nabla \cdot \mu \nabla +1)^{-1} $ has the following -- continuous --
	mapping properties
	\begin{equation} \label{e-inversA1}
	  (-\nabla \cdot \mu \nabla + 1)^{-1} : \breve H^{-1,2}_\Gamma \to
		H^{1,2}_\Gamma \hookrightarrow L^2
	\end{equation}
	and
	\begin{equation} \label{e-inversA2}
	  (-\nabla \cdot \mu \nabla + 1)^{-1} : \breve H^{-1,q}_\Gamma \to L^\infty
		\hookrightarrow L^{d+1} \quad \text {for } q \ge d + 1
	\end{equation}
	(see \cite{griehoel}). This shows that the resolvent is compact for
	$q = 2$ and for $q \ge d + 1$. If one takes in \eqref{e-inversA2}
	$q = d + 1$ and interpolates between \eqref{e-inversA1} and
	\eqref{e-inversA2}, one obtains a continuous mapping $(-\nabla \cdot
	\mu \nabla + 1)^{-1} : \breve H^{-1,q}_\Gamma \to L^q$ for every $q \in
	\left] 2, d+1 \right[$, see Corollary \ref{c-antili}.
        \qedhere
\end{enumerate}
\end{proof}
%
One essential instrument for our subsequent considerations are (upper) Gaussian estimates.
%
\begin{theorem} \label{t-gausss} 
The semigroup generated by $\nabla \cdot \mu \nabla$ in $L^2$ satisfies upper
Gaussian estimates, precisely:
\[ (\e^{t \nabla \cdot \mu \nabla} f)(\mathrm x) = \int_\Omega
	K_t(\mathrm x, \mathrm y) f(\mathrm y) \; \dd \mathrm y, \quad
	\mathrm x \in \Omega, \; f \in L^2,
\]
for some measurable function $K_t : \Omega \times \Omega \to \R_+$ and for all
$\epsilon > 0$ there exist constants $c, b > 0$,
such that
\begin{equation}\label{5.5}
  0 \leq K_t(\mathrm x, \mathrm y) \le \frac{c}{t^{d/2}} \;
        \e^{- b \frac{|\mathrm x - \mathrm y|^2}{t}} \e^{\epsilon t}, \quad t
	> 0, \; a.a. \; \mathrm x, \mathrm y
        \in \Omega.
\end{equation}
\end{theorem}
%
This follows from the following simplified version of Theorem~6.10 in
\cite{Ouh05} (see also \cite{arel}).
%
\begin{proposition}[Ouhabaz] \label{p-gauss}
Assume that $-\nabla \cdot \omega \nabla$, with $\omega \in L^\infty(\Omega;
\cL(\R^d))$ uniformly elliptic, is defined on the form domain $V \subseteq
H^{1,2}$ that satisfies
\begin{enumerate}
\item[a)] $V$ is closed in $ H^{1,2}$,
\item[b)] $H^{1,2}_0 \subseteq V$,
\item[c)] $V$ has the $L^1$-$H^{1,2}$ extension property,
\item[d)] $u \in V$ implies $\sign(u) \inf(1, |u|) \in V$, where $\sign(u) =
	u/|u|$ if $u \neq 0$ and $\sign(u) = 0$ else.
\item[e)] $u \in V$ implies $\e^\psi u \in V$ for every $\psi \in
        C^\infty(\R^d)$, such that $\psi$ and $|\nabla \psi|$ are bounded in
        $\R^d$.
\end{enumerate}
Then $\e^{t \nabla \cdot \omega \nabla}$ satisfies an upper Gaussian estimate as in \eqref{5.5}.
\end{proposition}
%
\begin{proof}[Proof of Theorem~\ref{t-gausss}]
We have to verify conditions a) -- e) from Proposition~\ref{p-gauss} for $V =
H^{1,2}_\Gamma$. a) and b) are obvious. For c) see Proposition~\ref{p-extend}
and d) is covered by \cite[Proposition~4.11]{Ouh05}. Finally, e) follows from
Proposition~\ref{p-basicl2}~iv).
\end{proof}
Another notion in our considerations will be the bounded holomorphic
functional calculus that we want to introduce briefly. Let $X$ be a
Banach space and $- B$ the generator of a bounded analytic semigroup on $X$.
Denoting, for $\kappa \in \left] 0, \pi \right]$,
\[ \Sigma_\kappa := \{ z \in \C \setminus \{ 0 \} : |\arg(z)| < \kappa \},
\]
we then have for some $\theta \in \left] 0, \pi/2 \right[$
\[ \sigma(B) \subseteq \Sigma_\theta \cup \{ 0 \}\quad \text{and} \quad
        \|R(\lambda, B) \|_{\cL(X)} \le \frac{M}{|\lambda|}, \quad \lambda \in
        \C \setminus \overline{\Sigma_\theta}.
\]
Following \cite{McI86} (see also \cite{DHP03}), for any angle $\kappa \in
\left] 0, \pi \right]$ we define the function spaces
\begin{align*}
\mathcal H^\infty(\Sigma_\kappa) &:= \{ \psi : \Sigma_\kappa \to \C, \text{ holomorphic and
       bounded} \} \quad \text{and} \\
\mathcal H^\infty_0(\Sigma_\kappa) &:= \bigl\{ \psi \in \mathcal H^\infty(\Sigma_\kappa) :
       \text{ there exist } C, \epsilon > 0 \text{ s.t. } |\psi(z)| \le C
       \frac{|z|^\epsilon}{(1 + |z|)^{2\epsilon}} \bigr\},
\end{align*}
both equipped with the norm $\| \psi \|_{\mathcal H^\infty_\kappa} := \sup_{z \in
\Sigma_\kappa} |\psi(z)|$. Then for $\psi \in \mathcal H^\infty_0(\Sigma_\kappa)$ with $\kappa > \theta$,
we may compute $\psi(B)$, using the Cauchy integral formula
\[ \psi(B) = \frac{1}{2 \pi i} \int_\angle \psi(z) R(z, B) \; \dd z,
\]
where the path $\angle$ is given by the two rays $t \e^{\pm i \varphi}$, $t >
0$, for some $\theta < \varphi < \kappa$. Note that this integral is absolutely
convergent in $\cL(X)$. We now say that $B$ has a \emph{bounded $\mathcal H^\infty$-calculus}, if there is a constant $C \ge 0$, such that
\[ \| \psi(B) \|_{\cL(X)} \le C \|\psi\|_{\mathcal{H}^\infty_\kappa}, \quad
	\psi \in \mathcal H^\infty_0(\Sigma_\kappa),
\]
for some $\kappa > \theta$. The infimum of all angles $\kappa$, for which this holds,
is called the \emph{$\mathcal H^\infty$-angle} $\varphi_B^\infty$ of $B$.

If $B$ admits a bounded $\mathcal H^\infty$-calculus for some $\kappa > \theta$, then the
mapping $\mathcal H^\infty_0(\Sigma_\kappa) \ni \psi \mapsto \psi(B) \in \cL(X)$ can be
extended uniquely to an algebra homomorphism between $\mathcal H^\infty(\Sigma_\kappa)$
and $\cL(X)$.
%
\begin{proposition} \label{p-infcalc}
Let $\partial \Omega \setminus \Gamma$ have nonzero boundary measure. Then the
following assertions hold for every $p \in \left] 1, \infty \right[$.
\begin{enumerate}
\item For sufficiently small $\gamma > 0$, the operator $-\nabla \cdot \mu
        \nabla - \gamma$ has a bounded $\mathcal H^\infty$-calculus on $L^p$ with
        $\mathcal H^\infty$-angle $\varphi^\infty_{-\nabla \cdot \mu \nabla - \gamma} =
        0$.
\item The set $\{ (-\nabla \cdot \mu \nabla)^{is} : s \in \R \}$ forms a
        strongly continuous group on $L^p$ admitting the estimate
        \[ \| ( -\nabla \cdot \mu \nabla )^{is} \|_{\mathcal L(L^p)} \le
                c_p \e^{|s| \vartheta},\quad s \in \R,
        \]
        with $0 \le  \vartheta < \pi/2$.
\end{enumerate}
\end{proposition}
%
\begin{proof}
Since the boundary measure of $\partial \Omega \setminus \Gamma$ is nonzero,
the operator $- \nabla \cdot \mu \nabla$ is continuously invertible in $L^2$,
i.e. $0$ does not belong to the spectrum. Hence, for sufficiently small
$\gamma > 0$, $- \nabla \cdot \mu \nabla -\gamma$ is still self-adjoint and
bounded by $0$ from below, cf. Proposition~\ref{p-basicl2}~ii). Thus, for
every $\delta \ge 0$ the operator $- \nabla \cdot \mu \nabla - \gamma +
\delta$ has a bounded $\mathcal H^\infty$-calculus on $L^2$ with $\mathcal H^\infty$-angle $0$.
Furthermore, taking $\delta > \gamma$, the semigroup generated by $\nabla
\cdot \mu \nabla + \gamma - \delta$ obeys the Gaussian estimate \eqref{5.5}
with $\epsilon = 0$. Thus, $- \nabla \cdot \mu \nabla - \gamma + \delta$ also
has a bounded $\mathcal H^\infty$-calculus on $L^p$ with $\mathcal H^\infty$-angle $0$ for all
$1 < p < \infty$ by \cite{duong/robin}.

In order to eliminate the `$\strut+\delta$', we observe that the spectrum of
$- \nabla \cdot \mu \nabla$ is $p$-independent, thanks to the Gaussian
estimates, see \cite{Kun99}. Thus, also in $L^p$ the spectrum of $- \nabla
\cdot \mu \nabla - \gamma$ is contained in the positive real axis. It was
shown in \cite[Prop. 6.10]{KKW06}, that in such a case, we may shift back the
operator without losing the bounded $\mathcal H^\infty$-calculus, as long as the
spectrum does not reach zero. This shows i).

As the functions $z \mapsto z^{is}$ belong to $\mathcal H^\infty(\Sigma_\phi)$ for
every $s \in \R$ and every $\phi \in \left] 0, \pi \right[$, part i) of this
proof yields $(- \nabla \cdot \mu \nabla)^{is} \in {\mathcal{L}}(L^p)$ with
$\|(-\nabla \cdot \mu \nabla)^{is}\|\le c $ for all $-1 \le s \le 1$. Thus,
ii) follows by \cite[Thm. III.4.7.1 and Cor. III.4.7.2]{amannbuch}.
\end{proof}
%
%
%
%
\section{Mapping properties for $(-\nabla \cdot \mu \nabla)^{1/2}$} \label{sec-WurzelIso}
%
%
%
%
In this chapter we prove that, under certain topological conditions on $\Omega$
and $\Gamma$, the mapping
\[ ( -\nabla \cdot \mu \nabla )^{1/2} : H^{1,q}_\Gamma \to L^q
\]
is a topological isomorphism for $q \in \left] 1, 2 \right[$. We abbreviate
$-\nabla \cdot \mu \nabla$ by $A_0$ throughout this chapter. Let us introduce
the following
%
\begin{assumption} \label{l-locmodel}
There is a bi-Lipschitz, volume-preserving mapping $\phi$ from a neighborhood
of $\overline \Omega$ into $\R^d$ such that $\phi(\Omega \cup \Gamma) = \alpha
K_-$ or $\alpha (K_- \cup \Sigma)$ or $\alpha (K_- \cup \Sigma_0)$ for some
 $\alpha >0$.
\end{assumption}
%
%
\begin{remark} \label{r-volume}
It is known that for a bi-Lipschitz mapping the property of being volume-preserving is
equivalent to the property that the absolute value of the determinant of the
Jacobian is one almost everywhere (see \cite[Ch.~3]{ev/gar}).
\end{remark}
%
The main results of this section are the following two theorems.
%
\begin{theorem} \label{t-mainsect}
Under the general assumptions made in Section 2 the following holds true:
If $\partial \Omega \setminus \Gamma$ has nonzero boundary measure, then,
for every $q \in \left] 1, 2 \right]$, the operator $A_0^{-1/2}$ is a
continuous operator from $L^q$ into $H^{1,q}_\Gamma$. Hence, it continuously
maps $\breve H^{-1,q}_\Gamma$ into $L^q$ for any $q \in \left[ 2, \infty \right[$.
\end{theorem}
%
%
\begin{theorem} \label{t-mainsectdown}
If in addition Assumption~\ref{l-locmodel} is fulfilled and $q \in \left] 1, 2
\right]$, then $A_0^{1/2}$ maps $H^{1,q}_\Gamma$ continuously into $L^q$.
Hence, it continuously maps $L^q$ into $\breve H^{-1,q}_\Gamma$ for any $q \in \left[
2, \infty \right[$.
\end{theorem}
%
\begin{remark} \label{r-dualsuff}
In both theorems the second assertion follows from the first by the selfadjointness of
 $A_0$ on $L^2$ and duality (see Remark \ref{r-antili}); thus we focus on the proof of
 the first assertions in the sequel.
\end{remark}
Let us first prove the continuity of the operator $A_0^{-1/2} : L^q \to
H^{1,q}_\Gamma$. In order to do so, we observe that this follows, whenever
\begin{itemize}
\item[1.] The Riesz transform $\nabla A_0^{- 1/2}$ is a bounded operator on
	$L^q$, and, additionally,
\item[2.] $A_0^{-1/2}$ maps $L^q$ into $H^{1,q}_\Gamma$.
\end{itemize}
The first item can be deduced from the following result of Duong and
$\rm M^c$Intosh (see \cite[Thm.~2]{DM99}) that is even true in a much more
general setting.
%
\begin{proposition} \label{p-duongMcI}
Let $B$ be a positive, selfadjoint operator on $L^2$, having the space $W$ as
its form domain and admitting the estimate $\|\nabla \psi \|_{L^2} \le c
\|B^{1/2}\psi \|_{L^2}$ for all $\psi \in W$. Assume that $W$ is invariant
under multiplication by bounded functions with bounded, continuous first
derivatives and that the kernel $K_t$ of the semigroup $\e^{-t B}$ satisfies
bounds
\begin{equation} \label{e-kernbound}
 |K_t(\mathrm x, \mathrm y)| \le \frac{C}{t^{d/2}} \biggl( 1 +
	\frac{| \mathrm x - \mathrm y|^2}{t} \biggr)^{-\beta}
\end{equation}
for some $\beta > d/2$. Then the operator $\nabla B^{-1/2}$ is of weak type~(1,1), and, thus
can be extended from $L^2$ to a bounded operator on $L^q$ for all $q \in \left] 1, 2 \right[$.
\end{proposition}
%
\begin{proof}[Proof of Theorem~\ref{t-mainsect}]
According to Theorem~\ref{t-gausss} the semigroup kernels corresponding to
the operator $A_0$ satisfy the estimate~\eqref{5.5}. Thus, considering the 
operator $A_0 + \epsilon$ for some $\epsilon > 0$, the corresponding kernels
satisfy again \eqref{5.5}, but without the factor $\e^{\epsilon t}$ now.
Next, we verify that $B := A_0 + \epsilon$ and $W := H^{1,2}_\Gamma$ satisfy
the assumptions of Proposition~\ref{p-duongMcI}. By Proposition~\ref{p-basicl2}, $W = H^{1,2}_\Gamma$ is the domain for $(A_0 +
\epsilon)^{1/2}$, thus $\| \nabla \psi \|_{L^2} \le c \| (A_0 + \epsilon
)^{1/2} \psi \|_{L^2}$ holds for all $\psi \in W$. The invariance property of
$W$ under multiplication is ensured by Proposition~\ref{p-basicl2}. Concerning
the bound \eqref{e-kernbound}, it is easy to see that the resulting Gaussian
bounds from Theorem~\ref{t-gausss} are even much stronger, since the function
$r \mapsto (1 + r)^\beta \e^{- b r}$, $r \ge 0$, is bounded for every $\beta >
0$. All this shows that $(A_0 + \epsilon)^{-1/2} : L^q \to H^{1,q}$ is
continuous for $q \in \left] 1, 2 \right]$.

Writing
\[ A_0^{-1/2} = (A_0 + \epsilon)^{-1/2} (A_0 + \epsilon)^{1/2} A_0^{-1/2},
\]
the assertion 1. follows, if we know that $(A_0 + \epsilon)^{1/2} A_0^{-1/2} :
L^q \to L^q$ is continuous. In order to see this, choose $\epsilon$ so small
that Proposition~\ref{p-infcalc}~i) ensures a bounded $\mathcal H^\infty$-calculus on
$L^q$ for $A_0 - \epsilon$, and observe that the function $z \mapsto (z +
2\epsilon)^{1/2} (z + \epsilon)^{-1/2}$ is in $\mathcal H^\infty(\Sigma_\phi)$ for any
$\phi \in \left]0, \pi \right[$.

It remains to show 2. The first point makes clear that $A_0^{-1/2}$ maps $L^q$
continuously into $H^{1,q}$, thus one has only to verify the correct boundary
behavior of the images. If $f \in L^2 \hookrightarrow L^q$, then one has
$A_0^{-1/2}f \in H^{1,2}_\Gamma \hookrightarrow H^{1,q}_\Gamma$. Thus, the
assertion follows from 1. and the density of $L^2$ in $L^q$. 
\end{proof}
%
\begin{remark} \label{r-nottrue}
Theorem \ref{t-mainsect} is not true for other values of $q$ in general: If it were, then,
 due to the case $q\le 2$ and duality, $A_0^{-1/2} : H^{-1,q}_\Gamma \to L^q$ and
$A_0^{-1/2} : L^q \to H^{1,q}_\Gamma$ would be continuous for a $q>2$. But for
any $q>2$ one can find a coefficient function $\mu$ such that the
corresponding operator $A_0^{-1}$ does not map $\breve H^{-1,q}_\Gamma$ into
$H^{1,q}_\Gamma$, see \cite{mey, e/k/r/s, e/r/s}, see also
\cite{auschmem} and the references therein.
\end{remark}
%
It follows the proof of Theorem \ref{t-mainsectdown}. It will be deduced from
the subsequent deep result on divergence operators with Dirichlet boundary
conditions and some permanence principles.
%
\begin{proposition}[Auscher/Tchamitchian, \cite{ausch/tcha01}] \label{p-ausch/tcha}
Let $q \in \left] 1, \infty \right[$ and $\,\Omega$ be  a strongly Lipschitz
domain. Then the root of the operator $A_0$, combined with a homogeneous
Dirichlet boundary condition, maps $H^{1,q}_0(\Omega)$ continuously into
$L^q(\Omega)$.
\end{proposition}
%
For further reference we mention the following immediate consequence of Theorem~\ref{t-mainsect}
and Proposition~\ref{p-ausch/tcha}.
%
\begin{corollary} \label{c-iso}
Under the hypotheses of Proposition \ref{p-ausch/tcha} the operator
$A_0^{-1/2}$ provides a topological isomorphism between $L^q$ and $H^{1,q}_0$,
if $q \in \left] 1, 2 \right]$.
\end{corollary}
%
In view of Assumption~\ref{l-locmodel} it is a natural idea to reduce our
considerations to the three model constellations mentioned there. In order to
do so, we have to show that the assertion of Theorem~\ref{t-mainsectdown} is
invariant under volume-preserving bi-Lipschitz transformations of the domain.
%
\begin{proposition} \label{p-transform}
Assume that $\phi$ is a mapping from a neighborhood of $\overline \Omega$
into $\R^d$ that is additionally bi-Lipschitz. Let us denote $\phi(\Omega) =
\Omega_\vartriangle$ and $\phi(\Gamma) = \Gamma_\vartriangle$. Define for any
function $f \in L^1(\Omega_\vartriangle)$
\[ (\Phi f) (\mathrm x) = f(\phi(\mathrm x)) = (f\circ \phi) (\mathrm x), \quad
        \mathrm x \in \Omega.
\]
Then
\begin{enumerate}
\item The restriction of $\Phi$ to any $L^p(\Omega_\vartriangle)$, $1 \le p <
        \infty$, provides a linear, topological isomorphism between this space
        and $L^p(\Omega)$.
\item For any $p \in \left]1, \infty\right[$, the mapping $\Phi$ induces a
        linear, topological isomorphism
        \[ \Phi_p : H_{\Gamma_\vartriangle}^{1,p}(\Omega_\vartriangle) \to
                H^{1,p}_{\Gamma}(\Omega).
        \]
\item $\Phi_{p'}^*$ is a linear, topological isomorphism between
        $\breve H_\Gamma^{-1,p}(\Omega)$ and
        $\breve H^{-1,p}_{\Gamma_\vartriangle}(\Omega_\vartriangle)$ for any $p \in
	\left] 1, \infty \right[$.
\item One has
        \begin{equation}\label{e-transformat} 
          \Phi_{p'}^* A_0 \Phi_p = -\nabla \cdot \mu_\vartriangle \nabla
        \end{equation}
        with
        \begin{equation}\label{e-transf}
          \mu_\vartriangle (\mathrm{y}) =
                \frac{1}{\big| \det(D \phi) (\phi^{-1}({\mathrm{y})}) \big|}
                (D \phi) (\phi^{-1}(\mathrm{y})) \;
                \mu(\phi^{-1}(\mathrm{y})) \; \bigl( D \phi \bigr)^T
                (\phi^{-1} (\mathrm{y}))
        \end{equation}
        for almost all $\mathrm{y} \in \Omega_\vartriangle$. Here, $D\phi$
        denotes the Jacobian of $\phi$ and $\det(D\phi)$ the corresponding
        determinant.
\item \label{p-transform:enum-iv} $\mu_\vartriangle$ also is bounded, Lebesgue
        measurable, elliptic and takes real, symmetric matrices as values.
\item The restriction of $\Phi_{2}^* \Phi$  to $L^2(\Omega_\vartriangle)$
	equals the multiplication operator which is induced by the function
        $\big| \det (D\phi)(\phi^{-1}(\cdot)) \big|^{-1}$. Consequently,
        if $|\det(D\phi)| = 1$ a.e., then the restriction of $\Phi_{2}^*\Phi$
        to $L^2(\Omega_\vartriangle)$ is the identity operator on
        $L^2(\Omega_\vartriangle)$, or, equivalently,
        $(\Phi_{2}^*)^{-1}|_{L^2(\Omega_\vartriangle)} =
        \Phi|_{L^2(\Omega_\vartriangle)}$.
\end{enumerate}
\end{proposition}
%
\begin{proof}
For i) see \cite[Ch.~1.1.7]{mazsob}. The proof of ii) is contained in \cite[Thm.~2.10)]{ggkr}
 and iii) follows from ii) by duality (see Remark \ref{r-antili}). Assertion iv) is well
 known, see \cite{hall} for an explicit verification, while v) is implied by \eqref{e-transf}
 and the fact that for a bi-Lipschitz mapping $\phi$ the Jacobian $D \phi$ and its inverse
 $\bigl (D \phi)^{-1}$ are essentially bounded (see \cite[Ch.~3.1]{ev/gar}).
We prove vi). For every $f \in L^2(\Omega_\vartriangle)$ and $g \in
H_{\Gamma_\vartriangle}^{1,2}(\Omega_\vartriangle)$ we calculate:
\begin{align*}
  \langle \Phi_{2}^* \Phi f, g
        \rangle_{\breve H_{\Gamma_\vartriangle}^{-1,2}(\Omega_\vartriangle)}
        &= \langle \Phi f, \Phi g \rangle_{\breve H_{\Gamma}^{-1,2}(\Omega)} =
        \langle  f \circ \phi, g \circ \phi
        \rangle_{\breve H_{\Gamma}^{-1,2}(\Omega)} = \int_\Omega f(\phi(\mathrm{x}))
        \overline g ( \phi (\mathrm{x})) \;\dd\mathrm{x} \\
  &= \int_{\Omega_\vartriangle} f(\mathrm{y}) \overline g (\mathrm{y})
        \frac{1}{\big| \det (D\phi)(\phi^{-1}{(\mathrm{y})})\big|}
        \;\dd\mathrm{y}.
\end{align*}
Thus, the anti-linear form $\Phi_{2}^*  \Phi f$ on
$H_{\Gamma_\vartriangle}^{1,2}(\Omega_\vartriangle)$ is represented by
$\big| \det (D\phi)(\phi^{-1}({\mathrm{\cdot}})) \big|^{-1} \in
L^\infty(\Omega_\vartriangle)$.
\end{proof}
%
\begin{lemma} \label{l-wurzel}
Let $p \in \left] 1, \infty \right[$. Suppose further that $\partial \Omega
\setminus \Gamma$ does not have boundary measure zero and that $|\det(D\phi)|
= 1$ almost everywhere in $\Omega$. Then, in the notation of the preceding
proposition, the operator $\bigl( -\nabla \cdot \mu_\vartriangle \nabla
\bigr)^{1/2}$ maps $H^{1,p}_{\Gamma_\vartriangle}(\Omega_\vartriangle)$
continuously into $L^p(\Omega_\vartriangle)$ if and only if $A_0^{1/2}$ maps
$H^{1,p}_{\Gamma}(\Omega)$ continuously into $L^p(\Omega)$.
\end{lemma}
%
\begin{proof}
We will employ the formula
\begin{equation} \label{e-wurxel}
  B^{-1/2} = \frac{1}{\pi} \int_0^\infty t^{-1/2} (B + t)^{-1} \; \dd t,
\end{equation}
$B$ being a positive operator on a Banach space $X$, see
\cite[Ch.~1.14/1.15]{triebel} or \cite[Ch.~2.6]{pazy}. Obviously, the integral
converges in the $\mathcal L(X)$-norm.

It is clear that our hypotheses of $\partial \Omega \setminus \Gamma$ not
having boundary measure zero implies that $\partial \Omega_\vartriangle
\setminus \Gamma_\vartriangle$ also has positive boundary measure. Thus, both,
$A_0$ and $-\nabla \cdot \mu_\vartriangle \nabla$ do not have spectrum in zero
and are positive operators in the sense of \cite[Ch.~1.14]{triebel} on any
$L^p$ (see Proposition \ref{p-basicl2}). From \eqref{e-transformat} and vi) of the
 preceding proposition one deduces
\begin{equation} \label{e-transformat1}
  \Phi_{2}^* \bigl( A_0 + t \bigr) \Phi_2 = -\nabla \cdot \mu_\vartriangle
        \nabla + t
\end{equation}
for every $t > 0$. This leads to
\[ \Phi_{2}^{-1} \bigl( A_0 + t \bigr)^{-1} \bigl( \Phi_2^* \bigr)^{-1} =
        \bigl( -\nabla \cdot \mu_\vartriangle \nabla + t \bigr)^{-1}.
\]
Restricting this last equation to elements from $L^2(\Omega_\vartriangle)$ and
making once more use of vi) in Proposition~\ref{p-transform}, we get the following
operator equation on $L^2(\Omega_\vartriangle)$:
\[ \Phi^{-1} \bigl( A_0 + t \bigr)^{-1} \Phi|_{L^2(\Omega_\vartriangle)} =
        \bigl( -\nabla \cdot \mu_\vartriangle \nabla + t \bigr)^{-1}.
\]
Integrating this equation with weight $\frac{t^{-1/2}}{\pi}$, one obtains,
according to \eqref{e-wurxel},
\begin{equation}\label{e-transformat4} 
  \Phi^{-1} A_0^{-1/2} \Phi|_{L^2(\Omega_\vartriangle)} = \bigl( -\nabla \cdot
        \mu_\vartriangle \nabla \bigr)^{-1/2},
\end{equation} 
again as an operator equation on $L^2(\Omega_\vartriangle)$. We recall that
the operators $A_0^{-1/2} : L^2(\Omega) \to H^{1,2}_\Gamma(\Omega)$, $( -\nabla
\cdot \mu_\vartriangle \nabla )^{-1/2} : L^2(\Omega_\vartriangle) \to
H^{1,2}_{\Gamma_\vartriangle}(\Omega_\vartriangle)$, $\Phi_2 :
H_{\Gamma_\vartriangle}^{1,2}(\Omega_\vartriangle) \to
H^{1,2}_{\Gamma}(\Omega)$ and $\Phi : L^2(\Omega_\vartriangle) \to
L^2(\Omega)$ all are topological isomorphisms. In particular, for any $f \in
L^2(\Omega_\vartriangle)$ the element $A_0^{-1/2} \Phi f$ is from
$H^{1,2}_\Gamma(\Omega)$. Thus, we may write \eqref{e-transformat4} as
\begin{equation}\label{e-transformat04} 
  \Phi^{-1}_2 A_0^{-1/2} \Phi|_{L^2(\Omega_\vartriangle)} = \bigl( -\nabla
        \cdot \mu_\vartriangle \nabla \bigr)^{-1/2}
\end{equation} 
and afterwards invert \eqref{e-transformat04}. We get the following operator
equation on $H_{\Gamma_\vartriangle}^{1,2}(\Omega_\vartriangle)$:
\[ \Phi^{-1} A_0^{1/2} \Phi_2 = \bigl( -\nabla \cdot \mu_\vartriangle \nabla
        \bigr )^{1/2}.
\]
In the sequel we make use of the fact that $\Phi_p :
H_{\Gamma_\vartriangle}^{1,p}(\Omega_\vartriangle) \to
H^{1,p}_{\Gamma}(\Omega)$ and $\Phi : L^p(\Omega_\vartriangle) \to
L^p(\Omega)$ are topological isomorphisms for all $p \in \left] 1, \infty
\right[$. Thus, first considering the case $p \in \left] 1, 2 \right[$ and
assuming that $A_0^{1/2}$ maps $H^{1,p}_\Gamma(\Omega)$ continuously into
$L^p(\Omega)$, we may estimate for all $\psi \in
H^{1,2}_{\Gamma_\vartriangle}(\Omega_\vartriangle)$
\begin{align} \label{e-estim0}
  \| \bigl( -\nabla \cdot \mu_\vartriangle \nabla \bigr)^{1/2} \psi
        \|_{L^p(\Omega_\vartriangle)} &= \| \Phi^{-1} A_0^{1/2} \Phi_2 \psi
        \|_{L^p(\Omega_\vartriangle)} \\
  &\le \| \Phi_{p}^{-1} \|_{\mathcal L(L^p(\Omega); L^p(\Omega_\vartriangle))}
        \| A_0^{1/2} \|_{\mathcal L(H^{1,p}_{\Gamma}(\Omega);L^p(\Omega))} \|
        \Phi_{2} \psi \|_{H_{\Gamma}^{1,p}(\Omega)}. \nonumber
\end{align}
Observing that $\Phi_2$ is only the restriction of $\Phi_p$, one may estimate
the last factor in \eqref{e-estim0}:
\begin{equation} \label{e-lastfac}
  \| \Phi_2 \psi \|_{H_{\Gamma}^{1,p}(\Omega)} \le \| \Phi_{p}
        \|_{\mathcal L(H_{\Gamma_\vartriangle}^{1,p}(\Omega_\vartriangle);
        H^{1,p}_{\Gamma}(\Omega))} \| \psi
        \|_{H_{\Gamma_\vartriangle}^{1,p}(\Omega_\vartriangle)}.
\end{equation}
This means that $( -\nabla \cdot \mu_\vartriangle \nabla )^{1/2}$ maps 
$H_{\Gamma_\vartriangle}^{1,2}(\Omega_\vartriangle)$, equipped with the induced 
$H_{\Gamma_\vartriangle}^{1,p}(\Omega_\vartriangle)$-norm, continuously into
 $L^p(\Omega_\vartriangle)$ and, consequently, extends to a continuous mapping from
 the whole $H_{\Gamma_\vartriangle}^{1,p}(\Omega_\vartriangle)$ into $L^p(\Omega_\vartriangle)$
 by density.

If $p \in \left] 2, \infty \right[$, one has the same estimates
\eqref{e-estim0} and \eqref{e-lastfac}, in this case only for elements $\psi \in
H^{1,p}_{\Gamma_\vartriangle}(\Omega_\vartriangle) \subseteq
H^{1,2}_{\Gamma_\vartriangle}(\Omega_\vartriangle)$.

Finally, the equivalence stated in the assertion follows by simply
interchanging the roles of $\mu$ and $\mu_\vartriangle$.
\end{proof}
%
\begin{remark} \label{r-volinv}
It is the property of 'volume-preserving' which leads, due to vi) of
Proposition~\ref{p-transform}, to \eqref{e-transformat1} and then to
\eqref{e-transformat4} and thus allows to hide the complicated geometry of
the boundary in $\Phi$ and $\mu_\vartriangle$.

It turns out that 'bi-Lipschitz' together with 'volume-preserving' is
not a too restrictive condition. In particular, there are such mappings --
although not easy to construct -- which map the ball onto the cylinder, the
ball onto the cube and the ball onto the half ball, see \cite{ghkr}, see also
\cite{fonse}. The general message is that this class has enough flexibility
 to map 'non-smooth objects' onto smooth ones.
\end{remark}
%
Lemma \ref{l-wurzel} allows to reduce the proof of Theorem \ref{t-mainsectdown} to
$\Omega = \alpha K_-$ and the three cases $\Gamma = \emptyset$, $\Gamma =
\alpha \Sigma$ or $\Gamma = \alpha \Sigma_0$. The first case, $\Gamma =
\emptyset$, is already contained in Proposition \ref{p-ausch/tcha}. In order
to treat the second one, we will use a reflection argument.

To this end we define for any $\mathrm{x} = (x_1, \dots, x_d) \in \R^d$ the
symbol $\mathrm{x}_- := (x_1, \dots, x_{d-1}, - x_d)$ and for a $d \times d$
matrix $\omega$, the matrix $\omega^-$ by
\[ \omega^-_{j,k} := \begin{cases}
                \omega_{j,k}, & \text{if } j,k < d,\\
                -\omega_{j,k}, & \text{if } j = d \text{ and } k \neq d
                        \text{ or } k = d \text{ and } j\neq d,\\
                \omega_{j,k} , & \text{if } j = k = d.
        \end{cases}
\]
Corresponding to the coefficient function $\mu$ on $K_-$, we then define the
coefficient function $\hat \mu$ on $K$ by
\[ \hat \mu (\mathrm{x}) := \begin{cases}
                \mu(\mathrm{x}), &\text{if } \mathrm{x} \in K_-, \\
                \bigl( \mu( \mathrm{x}_-)\bigr )^-, &\text{if } \mathrm{x_-}
                        \in K_-,\\
                0, & \text{if }  \mathrm{x} \in \Sigma.
        \end{cases}
\]
Finally, we define for $\varphi \in L^1(K)$ the reflected function $\varphi_-$
by $\varphi_-(x) = \varphi(x_-)$ and, using this, the extension and
restriction operators
\begin{align*}
  \mathfrak{E} &: L^1(K_-) \to L^1(K), &(\mathfrak{E} f) (x) &=
	\begin{cases}
	 f(x), \quad & \text{if } x \in K_-,\\
	 f(x_-), \quad & \text{if } x_- \in K_-,
	\end{cases} \\
  \mathfrak{S} &: \breve H^{-1,2}_\Sigma(K_-) \to \breve H^{-1,2}(K), &\left\langle
	\mathfrak{S} f, \varphi \right\rangle_{\breve H^{-1,2}(K)} &= \left\langle
	f, \varphi|_{K_-} + \varphi_-|_{K_-}
	\right\rangle_{\breve H^{-1,2}_\Sigma(K_-)}, \\
  \mathfrak{R} &: L^1(K) \to L^1(K_-), &\mathfrak{R} f &= f|_{K_-}.
\end{align*}
%
%
\begin{proposition}\label{p-spiegel}
\begin{enumerate}
\item If $\psi \in H^{1,2}_\Sigma(K_-)$ satisfies $A_0 \psi = f \in
        \breve H^{-1,2}_\Sigma(K_-)$, then
        \[ -\nabla \cdot \hat \mu \nabla \mathfrak{E} \psi = \mathfrak{S} f \in
		\breve H^{-1,2}(K).
        \]
\item \label{p-spiegel:enum-ii} The
	operator $\mathfrak{S} : \breve H_\Sigma^{-1,2}(K_-) \to \breve H^{-1,2}(K)$ is continuous.
\end{enumerate}
\end{proposition}
%
\begin{proof}
\begin{enumerate}
\item It is known that $\mathfrak{E}\psi$ belongs to
	$H^{1,2}_0(K)$, see \cite[Lemma~3.4]{giusti}. Thus, the assertion is
	obtained by the
        definitions of $\mathfrak{E}\psi$, $\mathfrak{S}f$, $A_0$, $-\nabla
	\cdot \hat \mu \nabla$ and straightforward calculations, based on
        Proposition~\ref{p-transform} when applied to the transformation
        $\mathrm {x} \mapsto \mathrm {x}_-$.
\item The operator under consideration is the adjoint of $ H^{1,2}_0(K) \ni \varphi \mapsto
        (\varphi|_{K_-} + \varphi_-|_{K_-}) \in  H_\Sigma^{1,2}(K_-)$.
        \qedhere
\end{enumerate}
\end{proof}
We are now in the position to prove Theorem \ref{t-mainsectdown} for the case
$\Gamma = \alpha \Sigma$. Up to a homothety we may focus on the case $\alpha
= 1$. First, we note that for any function $\varphi \in L^2(K_-)$ one finds
$\mathfrak E \varphi = \mathfrak S \varphi$, where we identified the functions
$\varphi$ and $\mathfrak E \varphi$ with the corresponding regular
distributions. Thus, one obtains from Proposition~\ref{p-spiegel}~i) that
$\bigl(A_0 + t \bigr)\psi = f \in \breve H^{-1,2}_\Sigma(K_-)$ implies
\[ \bigl(-\nabla \cdot \hat \mu \nabla + t \bigr) \mathfrak E \psi =
	\mathfrak S f,
\]
or, equivalently,
\[ \mathfrak E \psi = \bigl( -\nabla \cdot \hat \mu \nabla + t \bigr)^{-1}
        \mathfrak S f
\]
for every $t \in \left[ 0, \infty \right[$. Expressing $\psi = \bigl(A_0 + t
\bigr)^{-1} f$, this yields
\[ \mathfrak E \bigl(A_0 + t \bigr)^{-1} f = \bigl( -\nabla \cdot \hat \mu
        \nabla + t \bigr)^{-1} \mathfrak S f.
\]
Multiplying this by $\frac{t^{-1/2}}{\pi}$ and integrating over $t$, one
obtains in accordance with \eqref{e-wurxel}
\begin{equation} \label{e-wurstel}
  \mathfrak E A_0^{-1/2}f = \bigl( -\nabla \cdot \hat \mu \nabla \bigr)^{-1/2}
        \mathfrak S f , \quad f \in \breve H^{-1,2}_\Sigma(K_-).
\end{equation}
Applying the restriction operator $\mathfrak R$ to both sides of \eqref{e-wurstel}, we get
\begin{equation} \label{e-wurstel028}  
A_0^{-1/2}f = \mathfrak R \bigl( -\nabla \cdot \hat \mu \nabla \bigr)^{-1/2}
        \mathfrak S f, \quad f \in \breve H^{-1,2}_\Sigma(K_-).
\end{equation}
Considering in particular elements $ f \in L^2(K_-)$ and taking for these into account
$\mathfrak E f = \mathfrak S f$, \eqref{e-wurstel028} implies
\begin{equation} \label{e-wurstel05}  
A_0^{-1/2}f = \mathfrak R \bigl( -\nabla \cdot \hat \mu \nabla \bigr)^{-1/2}
        \mathfrak E f, \quad f \in L^2(K_-).
\end{equation}
Since both operators $- A_0$ and $\nabla \cdot \hat \mu \nabla$
generate contraction semigroups on any $L^p$, and $0$ does not belong to the
spectrum for both of them, the operators $A_0^{-1/2}$ and 
$\bigl( -\nabla \cdot \hat \mu \nabla \bigr)^{-1/2}$ are bounded also on $L^p(K_-)$
 and $L^p(K)$, respectively. 
Hence, \eqref{e-wurstel05} remains true for any $f \in L^p(K_-)$ with $p \in \left]1, 2 \right[$.
Now, on one hand it is clear that $\mathfrak E (L^p(K_-))$ equals the symmetric
part of $L^p(K)$, i.e. the set of functions which satisfy $\varphi =
\varphi_-$. Using the definition of the coefficient function $\hat \mu$ and
formula \eqref{e-transformat}, one recognizes that the resolvent of $-\nabla
\cdot \hat \mu \nabla$ commutes with the mapping $\varphi \mapsto \varphi_-$.
Again exploiting formula \eqref{e-wurxel}, this shows that $( -\nabla
\cdot \hat \mu \nabla )^{-1/2}$ also commutes with the mapping $\varphi
\mapsto \varphi_-$. Thus, $\bigl( -\nabla \cdot \hat \mu \nabla \bigr)^{-1/2}$
maps the set of symmetric functions, satisfying $\varphi = \varphi_-$, into
itself and also the set of antisymmetric functions, satisfying $\varphi =
-\varphi_-$. Consequently, $\bigl( -\nabla \cdot \hat \mu \nabla \bigr)^{-1/2}
\mathfrak E (L^p(K_-))$ must equal the symmetric part of $H^{1,p}_0(K)$ because
$\bigl( -\nabla \cdot \hat \mu \nabla \bigr)^{-1/2}$ is a surjection onto the
whole $H^{1,p}_0(K)$ by Corollary \ref{c-iso}. But, it is known (see
\cite[Thm.~3.10]{giusti}) that for any given function $h \in
H^{1,p}_\Sigma(K_-)$ the symmetric extension belongs to $H^{1,p}_0(K)$. Thus
$\mathfrak R \bigl( -\nabla \cdot \hat \mu \nabla \bigr)^{-1/2} \mathfrak E =
A_0^{-1/2}$ is a surjection onto $H^{1,p}_\Sigma(K_-)$. Since, by
Theorem~\ref{t-mainsect} $A_0^{-1/2} : L^p(K_-) \to H^{1,p}_\Sigma(K_-)$ is
continuous, the continuity of the inverse is implied by the open mapping
theorem.

\bigskip

In order to prove the same for the third model constellation, i.e. $\Gamma = \Sigma_0$, we show
%
\begin{lemma} \label{l-umform}
For every $\alpha >0$ there is a volume-preserving, bi-Lipschitz mapping $\phi
: \R^d \to \R^d$ that maps $\alpha (K_- \cup \Sigma_0)$ onto $\alpha(K_- \cup
\Sigma)$.
\end{lemma}
%
\begin{proof}
Up to a homothety we may focus on the case $\alpha =1$. Let us first consider
the case $d = 2$. We define on the lower halfspace $\{ (x,y) \in \R^2 \with y \le 0 \}$
\[ \rho_1(x,y) :=
        \begin{cases} 
                (x-y/2,y/2),  \; &\text{if } x \le 0, \ y \ge x, \\
                (x/2,-x/2+y), \; &\text{if } x \le 0, \ y < x,\\
                (x/2,x/2+y),  \; &\text{if } x > 0, \ y < -x,\\
                (x+y/2,y/2),  \; &\text{if } x > 0, \ y \ge -x.
        \end{cases}
\]
Observing that $\rho_1$ acts as the identity on the $x$-axis, we may define
$\rho_1$ on the upper half space $\{(x,y) \in \R^2 \with y > 0 \}$ by $\rho_1(x,y) =
(x_0, -y_0)$ with $(x_0,y_0) = \rho_1(x, -y)$. In this way we obtain a globally
bi-Lipschitz transformation $\rho_1$ from $\R^2$ onto itself that transforms
$K_- \cup \Sigma_0$ onto the triangle shown in \figref{fig-nachrho1}.

\begin{figure}[htbp]
  \includegraphics[scale=0.3]{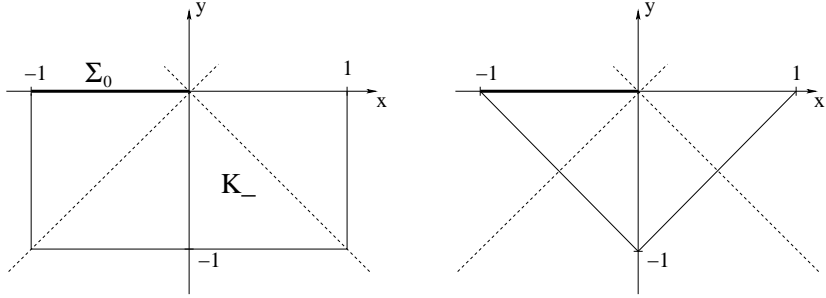}
  \caption{\label{fig-nachrho1} $K_- \cup \Sigma_0$ and $\rho_1(K_- \cup
        \Sigma_0)$}
\end{figure}

Next we define the bi-Lipschitz mapping $\rho_2 : \R^2 \to \R^2$ by
\[ \rho_2(x,y) :=
        \begin{cases}
                (x,x+2y+1)  , \; &\text{if } x \le 0, \\
                (x,-x+2y+1) , \; &\text{if } x > 0,
        \end{cases}
\]
in order to get the geometric constellation in \figref{fig-nachrho2}.

\begin{figure}[htbp]
  \centerline{\includegraphics[scale=0.28]{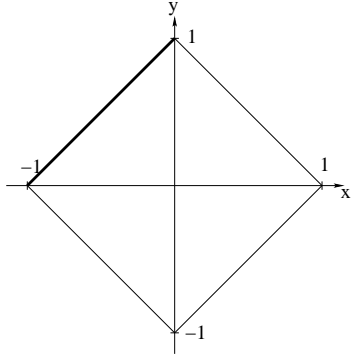}}
  \caption{\label{fig-nachrho2} $\rho_2(\rho_1(K_- \cup \Sigma_0))$}
\end{figure}

If $\rho_3$ is the (clockwise) rotation of $\pi/4$, we thus achieved that
$\rho := \rho_3 \rho_2 \rho_1 : \R^2 \to \R^2$ is bi-Lipschitzian and satisfies
\[ \rho(K_- \cup \Sigma_0) = \Bigl\{ (x,y) \in \R^2 \with -\frac{1}{\sqrt 2} < x <
        \frac{1}{\sqrt 2},\ - \frac{1}{\sqrt 2} < y \le \frac{1}{\sqrt 2}
        \Bigr\}.
\]
Let $\rho_4 : \R^2 \to \R^2$ be the affine mapping $(x,y) \mapsto (\sqrt 2 x,
\frac{1}{\sqrt 2 } y - \frac {1}{2})$. Then $\phi = \phi_2 := \rho_4 \rho$  maps $K_-
\cup \Sigma_0$ bi-Lipschitzian onto $ K_- \cup \Sigma$ in the $2$-$d$ case. As
is easy to check, the modulus of the determinant of the Jacobian is identically
one a.e. Hence, $\phi_2$ is volume-preserving.

If $d \ge 3$, one simply puts $\phi(x_1, \dots, x_d) := (x_1,\dots, x_{d-2},
\phi_2(x_{d-1}, x_d))$.
\end{proof}
Thus, the proof of Theorem~\ref{t-mainsectdown} in the case $\Gamma = \alpha
\Sigma_0$ results from the case $\Gamma = \alpha \Sigma$, Lemma~\ref{l-wurzel}
and Lemma~\ref{l-umform}.
%
\begin{remark} \label{r-missfort}
Let us mention that Lemma~\ref{l-wurzel}, only applied to $\Omega =K$ and $\Gamma =
\emptyset$ (the pure Dirichlet case) already provides a zoo of geometries
which is not covered by \cite{ausch/tcha01}. Notice in this context that the
image of a strongly Lipschitz domain under a bi-Lipschitz transformation needs
not to be a strongly Lipschitz domain at all, cf.
Subsection~\ref{subsec-balks}, see also \cite[Ch.~1.2]{grisvard85}.
\end{remark}
%
%
%
%
%
%
\section{Maximal parabolic regularity for $A$}
\label{sec-Consequences}
%
%
%
%
In this section we intend to prove the first main result of this work
announced in the introduction.
Let us first recall the notion of maximal parabolic $L^s$ regularity.
%
\begin{definition} \label{d-maxreg}
Let $1 < s < \infty$, let $X$ be a Banach space and let $J := \left]T_0, T \right[ 
\subseteq \R$ be a bounded interval. Assume that $B$ is a closed
operator in $X$ with dense domain $ D$ (in the sequel always equipped with the graph norm).
We say that $B$ satisfies \emph {maximal parabolic $L^s(J; X)$ regularity}, if for any
$f\in L^s(J; X)$ there exists a unique function $u \in W^{1,s}(J; X) \cap
L^s(J; D)$ satisfying
\[  u' + Bu = f,\quad \quad u(T_0) = 0,
\]
where the time derivative is taken in the sense of $X$-valued distributions on
$J$ (see \cite[Ch~III.1]{amannbuch}).
\end{definition}
%
%
\begin{remark} \label{r-ebed}
\begin{enumerate}
\item It is well known that the property of maximal parabolic regularity of
        an operator $B$ is independent of $s \in \left] 1, \infty \right[$ and
        the specific choice of the interval $J$ (cf. \cite{dore}). Thus, in the
        following we will say for short that $B$ admits maximal parabolic
        regularity on $X$.
\item If an operator satisfies maximal parabolic regularity on a Banach space
	$X$, then its negative generates an analytic semigroup on $X$ (cf.
	\cite{dore}). In particular, a suitable left half plane belongs to its
	resolvent set.
\item If $X$ is a Hilbert space, the converse is also true: The negative of
	every generator of an analytic semigroup on $X$ satisfies maximal
	parabolic regularity, cf. \cite{deSimon} or \cite{dore}.
\item 
       If $-B$ is a generator of an analytic semigroup on a Banach space $X$,
	we define
	\[  B \bigl( \frac{\partial}{\partial t} + B
		\bigr)^{-1} : C(\overline{J}; dom_X(B)) \to L^s(J; X)
	\]
	by
	\[\Bigl ( B \bigl( \frac{\partial}{\partial t}
		+ B \bigr)^{-1} f \Bigr) (t) := B \int_{T_0}^t \e^{(s-t)B}
		f(s) \,\dd s.
	\]
	Then, by definition of the distributional time derivative, it is easy
	to see that $B$ has maximal parabolic regularity on $X$ if and only if
	the operator $ B\bigl (\frac{\partial}{\partial t} + B \bigr)^{-1}$ continuously
         extends to an operator from $L^s(J; X)$ into itself.
\item Observe that
        \begin{equation} \label{e-embedcont}
          W^{1,s}(J; X) \cap L^s(J; D) \hookrightarrow
                C(\overline{J}; (X, D)_{1-\frac{1}{s},s}).
        \end{equation}
\end{enumerate}
\end{remark}
%
Let us first formulate the following lemma, needed in the sequel. 
%
\begin{lemma} \label{l-maxint}
Suppose that $X,Y$ are Banach spaces, which are contained in a third
Banach space $Z$ with continuous injections. Let $B$ be a linear operator on
$Z$ whose restriction to each of the spaces $X,Y$ induce closed,
densely defined operators there. Assume that the induced operators fulfill
maximal parabolic regularity on $X$ and $Y$, respectively. Then $B$ satisfies
maximal parabolic regularity on each of the interpolation spaces
$[X, Y]_{\theta}$ and $(X,Y)_{\theta,s}$ with $\theta \in \left] 0, 1
\right[$, $s \in \left] 1, \infty \right[$.
 \end{lemma}
%
\begin{proof}
By supposition, $(X,Y)$ forms an interpolation couple. In this case it is known 
 (see \cite[Ch.~1.18.4]{triebel}) that one has for any
$\theta \in \left] 0, 1 \right[$ and any $s \in \left] 1, \infty \right[$ the
interpolation identities
\begin{align} \label{e-int01}
  \bigl[ L^s(J;X), L^s(J;Y) \bigr]_{\theta} &= L^s(J;[X, Y]_{\theta})
  \intertext{and}
  \label{e-int02}
	\bigl( L^s(J;X), L^s(J;Y) \bigr)_{\theta,s} &= L^s(J;(X,Y)_{\theta,s}).
\end{align}
Due to Remark~\ref{r-ebed}~ii), $-B$ generates an analytic semigroup on $X$ and 
$Y$, respectively. Obviously, the corresponding resolvent estimates are maintained under
real and complex interpolation, so $-B$ also generates an analytic semigroup on the 
corresponding interpolation spaces. 
Taking into account \eqref{e-int01} or \eqref{e-int02} and invoking Remark~\ref{r-ebed}~iv),
 the operators
\begin{align*}
 B \bigl( \frac{\partial}{\partial t} + B \bigr)^{-1} &: L^s(J;X)
	\to L^s(J;X)
 \intertext{and}
 B \bigl( \frac{\partial}{\partial t} + B \bigr)^{-1} &:
        L^s(J;Y) \to L^s(J;Y)
\end{align*}
are continuous, if $s \in \left] 1, \infty \right[$. Thus, interpolation
together with \eqref{e-int01} (\eqref{e-int02}, respectively) tells us that
$ B  \bigl( \frac{\partial}{\partial t} + B \bigr)^{-1}$ also maps
$L^s(J;[X, Y]_{\theta})$ and $L^s(J;(X, Y)_{\theta,s})$ continuously into
itself. So the assertion follows again by Remark~\ref{r-ebed}~iv).
\end{proof}
This lemma will lead to the main result of this section, maximal regularity of
$A$ in various distribution spaces, as soon as we can show this in the space
$\breve H^{-1, q}_\Gamma$, what we will do now. Precisely, we will show the following
result.
%
\begin{theorem} \label{t-central}
 Let $\Omega$, $\Gamma$ fulfill
Assumption~\ref{a-groegerregulaer} and set $q_{\mathrm{iso}}:= \sup M_{\mathrm{iso}}$, where
\[ M_{\mathrm{iso}}:= \{ q \in \left[ 2, \infty \right[ \with - \nabla
	\cdot \mu \nabla + 1 : H^{1,q}_\Gamma \to \breve H^{-1, q}_\Gamma \text{ is a
	topological isomorphism} \}.
\]
Then $-\nabla \cdot \mu \nabla$ satisfies maximal parabolic regularity on
$\breve H^{-1,q}_\Gamma$ for all $q \in \left[ 2, q_{\mathrm{iso}}^* \right[$, where
by $r^*$ we denote the Sobolev conjugated index of $r$, i.e.
\[ r^* = \begin{cases}
             \infty, &\text{if } r \ge d, \\
              \bigl (\frac {1}{r}-\frac {1}{d}\bigr )^{-1}, &\text{if } r \in
		\left[ 1, d \right[.
        \end{cases}
\]
\end{theorem}
%
%
\begin{remark} \label{r-darunter}
\begin{enumerate}
\item If $\Omega$, $\Gamma$ fulfill Assumption~\ref{a-groegerregulaer}~a), then
	$q_{\mathrm{iso}} > 2$, see \cite{groe/reh} and also \cite{groeger89}.
\item It is clear by Lax-Milgram and interpolation (see
	Proposition~\ref{p-interpol} and Corollary \ref{c-antili}) that $ M_{\mathrm{iso}}$ is the
	interval $\left[ 2, q_{\mathrm{iso}} \right[$ or $\left[ 2,
	q_{\mathrm{iso}} \right]$. Moreover, it can be concluded from a deep theorem
        of Sneiberg \cite{snei} (see also \cite[Lemma~4.16]{auschmem}) that the second
       case cannot occur.
\end{enumerate}
\end{remark}
In a first step we show
%
\begin{theorem} \label{t-maxparloc}
Let $\Omega, \Gamma$ fulfill Assumption~\ref{l-locmodel}. Then $- \nabla \cdot
\mu \nabla$ satisfies maximal parabolic regularity on $\breve H^{-1,q}_\Gamma$ for
all $q \in \left[ 2, \infty \right[$.
\end{theorem}
%
This will be a consequence of the following lemma.
%
\begin{lemma} \label{l-bounded}
Let $\Omega, \Gamma$ satisfy Assumption~\ref{l-locmodel}. Then for all $q \in
\left[ 2, \infty \right[$ the set $\{ (-\nabla \cdot \mu \nabla)^{is} \with s
\in \R \}$ forms a strongly continuous group on $\breve H^{-1,q}_\Gamma$, satisfying
the estimate
\begin{equation} \label{e-estibound}
  \| (- \nabla \cdot \mu \nabla)^{is} \|_{\mathcal L(\breve H^{-1,q}_\Gamma)} \le
        c \e^{|s| \vartheta}, \quad s \in \R,
\end{equation}
for some $\vartheta  \in [0, \frac {\pi}{2}[$.

Moreover, we have the following resolvent estimate
\begin{equation} \label{e-resolv}
  \| (-\nabla \cdot \mu \nabla + \lambda)^{-1}
        \|_{\mathcal L(\breve H^{-1,q}_\Gamma)} \le \frac{c}{1+|\lambda|},
        \quad \Re \lambda \ge 0.
\end{equation}
\end{lemma}
%
\begin{proof}
We first note that Assumption~\ref{l-locmodel} in particular implies that
the Dirichlet boundary part $\partial \Omega \setminus \Gamma$ has non-zero
boundary measure. Thus, by Proposition~\ref{p-infcalc}~i), we may fix some
$\epsilon > 0$, such that $-\nabla \cdot \mu \nabla - \epsilon$ has
a bounded $\mathcal H^\infty$-calculus on $L^q$.
Since the functions $z \mapsto (z + \epsilon)^{is} =
(z + \epsilon)^{1/2} (z+\epsilon)^{is} (z+\epsilon)^{-1/2}$, $s \in \R$, and $z \mapsto
(z +\epsilon + \lambda)^{-1} = (z + \epsilon)^{1/2} (\lambda + z +
\epsilon)^{-1} (z + \epsilon)^{-1/2}$, $\mathrm{Re}\,\lambda \ge 0$, are in
$\mathcal H^\infty(\Sigma_\phi)$ for all $\phi \in \left] 0, \pi \right[$, one has the
operator identities
\begin{align} \label{e-wurz}
  \bigl( -\nabla \cdot \mu \nabla \bigr)^{is} &= \bigl( -\nabla \cdot \mu
        \nabla \bigr)^{1/2} \bigl( -\nabla \cdot \mu \nabla \bigr)^{is} \bigl(
        -\nabla \cdot \mu \nabla \bigr)^{-1/2}, \quad s \in \R,
  \intertext{and}
  \label{e-wurz1}
  \bigl( -\nabla \cdot \mu \nabla + \lambda \bigr)^{-1} &= \bigl( -\nabla \cdot
        \mu \nabla \bigr)^{1/2} \bigl( -\nabla \cdot \mu \nabla + \lambda
        \bigr)^{-1} \bigl( -\nabla \cdot \mu \nabla \bigr)^{-1/2}, \quad
        \Re \lambda \ge 0,
\end{align}
on $L^q$.
Under Assumption~\ref{l-locmodel} $(-\nabla \cdot \mu \nabla)^{1/2}$ is a
topological isomorphism between $L^q$ and $\breve H^{-1, q}_\Gamma$
for every $q \in \left[ 2, \infty \right[$, thanks to Theorem~\ref{t-mainsect}
and Theorem~\ref{t-mainsectdown}. Thus, one can estimate for every $f \in L^q$
\begin{align*}
  & \| (-\nabla \cdot \mu \nabla)^{is} f \|_{\breve H^{-1,q}_\Gamma} \\
  \le\;& \| (-\nabla \cdot \mu \nabla)^{1/2}
        \|_{\mathcal L(L^q,\breve H^{-1,q}_\Gamma)}
        \| (-\nabla \cdot \mu \nabla)^{is} \|_{\mathcal L(L^q)}
        \| (-\nabla \cdot \mu \nabla)^{-1/2}
        \|_{\mathcal L(\breve H^{-1,q}_\Gamma,L^q)} \|f\|_{\breve H^{-1,q}_\Gamma}.
\end{align*}
Since $L^q$ is dense in $\breve H^{-1,q}_\Gamma$, this inequality extends to all of
$\breve H^{-1,q}_\Gamma$. Together with Proposition~\ref{p-infcalc}~ii) this yields
the estimate~\eqref{e-estibound}, which also implies the group property, see 
\cite[Thm. III.4.7.1 and Cor. III.4.7.2]{amannbuch}.

\eqref{e-resolv} is proved analogously to \eqref{e-estibound}, only
using \eqref{e-wurz1} instead of \eqref{e-wurz} and the corresponding
resolvent estimate in $L^q$, cf. Proposition~\ref{p-basicl2}~v) (note that
here $-\nabla \cdot \mu \nabla$ is continuously invertible).
\end{proof}
%

%
It follows the proof of Theorem \ref{t-maxparloc}: By Theorems~\ref{t-mainsect} and
 \ref{t-mainsectdown}, $\breve H^{-1,q}_\Gamma$ is an isomorphic image of the UMD~space $L^q$ and,
 hence, a UMD~space itself. Since by Lemma~\ref{l-bounded} the operator $-\nabla \cdot \mu \nabla$
 generates an analytic semigroup and has bounded imaginary powers with the right bound, maximal
 parabolic regularity follows by the Dore-Venni result \cite{dorevenni}.

\medskip

Now we intend to 'globalize' Theorem~\ref{t-maxparloc}, in other words: We
prove that $- \nabla \cdot \mu \nabla$ satisfies maximal parabolic regularity
on $\breve H^{-1,q}_\Gamma$ for suitable $q$ if $\Omega$, $\Gamma$ satisfy only
Assumption~\ref{a-groegerregulaer}, i.e. if $\alpha K_-$, $\alpha (K_- \cup
\Sigma)$ and $\alpha (K_- \cup \Sigma_0)$ need only to be model sets for the
constellation around boundary points. Obviously, then the variety of
admissible $\Omega$'s and $\Gamma$'s increases considerably, in particular,
$\Gamma$ may have more than one connected component.
\subsection{Auxiliaries}
We continue with some results which in essence allow to restrict distributions
to subdomains and, on the other hand, to extend them to a larger domain --
including the adequate boundary behavior.
%
\begin{lemma} \label{l-restr/ext}
Let $\Omega, \Gamma$ satisfy Assumption~\ref{a-groegerregulaer} and let
$\Upsilon \subseteq \R^d$ be open, such that $\Omega_\bullet := \Omega \cap
\Upsilon$ is also a Lipschitz domain. Furthermore, we put $\Gamma_\bullet :=
\Gamma \cap \Upsilon$ and fix an arbitrary function $\eta \in
C^\infty_0(\R^d)$ with $\supp(\eta) \subseteq \Upsilon$. Then for any $q \in
\left] 1, \infty \right[$ we have the following assertions.
\begin{enumerate}
\item \label{l-restr/ext:enum-i} If $v \in H^{1,q}_\Gamma(\Omega)$, then $\eta
        v|_{\Omega_\bullet} \in H^{1,q}_{\Gamma_\bullet}(\Omega_\bullet)$ and
        the mapping 
        \[ H^{1,q}_\Gamma(\Omega) \ni v \mapsto \eta v|_{\Omega_\bullet} \in
        H^{1,q}_{\Gamma_\bullet}(\Omega_\bullet)
        \]
        is continuous.
\item Let for any $v \in L^1(\Omega_\bullet) $ the symbol $\tilde v$
        indicate the  extension of $v$ to $\Omega$ by zero.
        Then the mapping
        \[ H^{1,q}_{\Gamma_\bullet}(\Omega_\bullet) \ni v \mapsto
                \widetilde{\eta v}
        \]
        has its image in $H^{1,q}_\Gamma (\Omega)$ and is continuous.
\end{enumerate}
\end{lemma}
%
\begin{proof}
For the proof of both items we will employ the following well known set
inclusion (cf. \cite[Ch.~3.8]{dieu})
\begin{equation} \label{e-setincl}
  (\partial \Omega \cap  \Upsilon) \cup ( \Omega \cap  \partial \Upsilon)
        \subseteq \partial \Omega_\bullet \subseteq (\partial \Omega \cap
        \Upsilon) \cup ( \overline{\Omega} \cap  \partial \Upsilon).
\end{equation}
\begin{enumerate}
\item First one observes that the multiplication with $\eta$ combined with the
        restriction is a continuous mapping from $H^{1,q}_\Gamma(\Omega)$ into
        $H^{1,q}(\Omega_\bullet)$. Thus, we only have to show that the image
        is contained in $H^{1,q}_{\Gamma_\bullet}(\Omega_\bullet)$, which, in
	turn, is sufficient to show for elements of the dense subset
        \[ \left\{ v|_\Omega \with v \in C^\infty(\R^d), \supp(v) \cap
                (\partial \Omega \setminus \Gamma) = \emptyset \right\}
        \]
	only. By \eqref{e-setincl} we get for such functions
        \[ \supp (\eta v) \cap ( \partial \Omega_\bullet \setminus
                \Gamma_\bullet) \subseteq \supp (\eta) \cap \supp (v) \cap
                \bigl[ \bigl( (\partial \Omega \cap \Upsilon) \cup
                ( \overline{\Omega} \cap \partial \Upsilon) \bigr) \setminus
                \bigl( \Gamma \cap \Upsilon \bigr) \bigr].
        \]
        Since $(\overline{\Omega} \cap \partial \Upsilon) \cap (\Gamma \cap
        \Upsilon) = \emptyset$, we see
        \begin{align*}
          \bigl( (\partial \Omega \cap \Upsilon) \cup
                ( \overline{\Omega} \cap \partial \Upsilon) \bigr) \setminus
                \bigl( \Gamma \cap \Upsilon \bigr) &= \bigl( (\partial
                \Omega \cap \Upsilon) \setminus (\Gamma \cap \Upsilon)
                \bigr) \cup \bigl( (\overline{\Omega} \cap \partial \Upsilon)
                \setminus (\Gamma \cap \Upsilon) \bigr) \\
          &= \bigl( (\partial \Omega \setminus \Gamma) \cap \Upsilon \bigr)
                \cup (\overline{\Omega} \cap \partial \Upsilon).
        \end{align*}
        This, together with $\supp (\eta) \subseteq \Upsilon$, yields
        \begin{align*}
        \supp (\eta v) \cap ( \partial \Omega_\bullet \setminus
                \Gamma_\bullet) &\subseteq \supp (\eta) \cap \supp (v) \cap
                \bigl( (\partial \Omega \setminus \Gamma) \cap \Upsilon
                \bigr) = \emptyset.
        \end{align*}
\item Let $v \in C^\infty(\R^d)$ with $\supp(v) \cap (\partial \Omega_\bullet
        \setminus \Gamma_\bullet) = \emptyset$. Since by the left hand side of
        \eqref{e-setincl} we have
        \[ \partial \Omega_\bullet \setminus \Gamma_\bullet \supseteq (\partial
                \Omega \cap \Upsilon) \setminus \Gamma_\bullet = \Upsilon \cap
                (\partial \Omega \setminus \Gamma),
        \]
        it follows $\supp(v) \cap \bigl( \Upsilon \cap (\partial \Omega
        \setminus \Gamma) \bigr) = \emptyset$. Combining this with $\supp(\eta)
        \subseteq \Upsilon$, we obtain
        \[ \supp (\eta v) \cap (\partial \Omega \setminus \Gamma) = \supp (\eta
                v) \cap \bigl( \Upsilon \cap (\partial \Omega \setminus
                \Gamma \bigr) = \emptyset,
        \]
        so $\eta v|_\Omega \in H^{1,q}_\Gamma(\Omega)$. Furthermore, it is not
        hard to see that $\|\eta v \|_{H^{1,q}(\Omega)} \le c_\eta
        \|v\|_{H^{1,q}(\Omega_\bullet)}$, where the constant
	$c_\eta$ is independent from
	$v$. Thus, the assertion follows, since $\{ v|_{\Omega_\bullet} \with v
	\in C^\infty(\R^d), \ \supp(v) \cap (\partial \Omega_\bullet \setminus
	\Gamma_\bullet ) = \emptyset \}$ is dense in
	$H^{1,q}_{\Gamma_\bullet}(\Omega_\bullet)$ and
	$H^{1,q}_\Gamma(\Omega)$ is closed in $H^{1,q}(\Omega)$.
        \qedhere
\end{enumerate}
\end{proof}
%
\begin{lemma} \label{l-project}
Let $\Omega$, $\Gamma$, $\Upsilon$, $\eta$, $\Omega_\bullet$ and
$\Gamma_\bullet$ be as in the preceding lemma, but assume $\eta$ to be real
valued. Denote by $\mu_\bullet$ the restriction of the coefficient function
$\mu$ to $\Omega_\bullet$ and assume $v \in H^{1,2}_\Gamma(\Omega)$ to be the
solution of
\[ - \nabla \cdot \mu \nabla v = f \in \breve H_\Gamma^{-1,2}(\Omega).
\]
Then the following holds true:
\begin{enumerate}
\item For all $q \in \left]1, \infty \right[$ the anti-linear form
        \[ f_\bullet : w \mapsto \langle f, \widetilde {\eta w}
                \rangle_{\breve H^{-1,2}_\Gamma}
        \]
        (where $\widetilde {\eta w}$ again means the extension of $\eta
	w$ by
	zero to the whole $\Omega$) is well defined and continuous on
        $H^{1,q'}_{\Gamma_\bullet}(\Omega_\bullet)$, whenever $f$ is an
        anti-linear form from $\breve H^{-1,q}_\Gamma(\Omega)$. The mapping
        $\breve H^{-1,q}_\Gamma(\Omega) \ni f \mapsto f_\bullet \in
	\breve H^{-1,q}_{\Gamma_\bullet}(\Omega_\bullet)$
        is continuous.
\item If we denote the anti-linear form
        \[ H^{1,2}_{\Gamma_\bullet}(\Omega_\bullet) \ni w \mapsto
                \int_{\Omega_\bullet} v \mu_\bullet \nabla \eta \cdot \nabla
                \overline w \,\dd \mathrm{x}
        \]
        by $I_v$, then $u := \eta v|_{\Omega_\bullet}$ satisfies
        \[ - \nabla \cdot \mu_\bullet \nabla u = - \mu_\bullet \nabla
                v|_{\Omega_\bullet} \cdot \nabla \eta|_{\Omega_\bullet} + I_v
                + f_\bullet.
        \]
\item  For every $q \ge 2$ and all $r \in \left[ 2, q^* \right[$ ($q^*$
	denoting again the Sobolev conjugated index of $q$) the mapping
        \begin{equation} \label{e-stoerterm}
          H^{1,q}_\Gamma(\Omega) \ni v \mapsto - \mu_\bullet \nabla
                v|_{\Omega_\bullet} \cdot \nabla \eta|_{\Omega_\bullet} + I_v
                \in \breve H^{-1,r}_{\Gamma_\bullet}(\Omega_\bullet)
        \end{equation}
        is well defined and continuous.
\end{enumerate}
\end{lemma}
%
\begin{proof}
\begin{enumerate}
\item The mapping $f \mapsto f_\bullet$ is the adjoint to $v \mapsto
        \widetilde {\eta v }$ which maps by the preceding lemma
        $H^{1,q'}_{\Gamma_\bullet}(\Omega_\bullet)$ continuously into
        $H^{1,q'}_\Gamma(\Omega)$.
\item  For every $w \in H^{1,2}_{\Gamma_\bullet}(\Omega_\bullet)$ we have
        \begin{align}
          & \langle -\nabla \cdot \mu_\bullet \nabla u, w
                \rangle_{\breve H^{-1,2}_{\Gamma_\bullet}(\Omega_\bullet)} =
                \int_{\Omega_\bullet} \mu_\bullet \nabla (\eta v) \cdot \nabla
                \overline w \,\dd\mathrm x \nonumber \\
          \label{e-localizeeq}
                =& -\int_{\Omega_\bullet} \overline w \; \mu_\bullet \nabla v
                \cdot \nabla \eta \,\dd\mathrm x + \int_{\Omega_\bullet} v
                \mu_\bullet \nabla \eta \cdot \nabla \overline w
                \,\dd\mathrm x + \int_{\Omega} \mu \nabla v \cdot \nabla
                \widetilde {(\overline{\eta w})} \,\dd\mathrm x.
        \end{align}
        An application of the definitions of $I_v$ and $f_\bullet $ yields the
        assertion.
\item 
	We regard the terms on the right hand side of \eqref{e-stoerterm} from
        left to right: $|\nabla \eta| \in L^\infty(\Omega_\bullet)$ and
        $|\mu_\bullet \nabla v|_{\Omega_\bullet}| \in L^{q}(\Omega_\bullet)$,
        consequently $\mu_\bullet \nabla v|_{\Omega_\bullet} \cdot \nabla
	\eta|_{\Omega_\bullet} \in L^{q}(\Omega_\bullet)$. This gives by
	Sobolev embedding and duality $\mu_\bullet \nabla v|_{\Omega_\bullet}
	\cdot \nabla \eta|_{\Omega_\bullet} \in (H^{1,r'}(\Omega_\bullet))'
	\hookrightarrow \breve H^{-1,r}_{\Gamma_\bullet}(\Omega_\bullet)$. On the
	other hand, we have $v \in H^{1,q}_\Gamma(\Omega) \hookrightarrow
	L^r(\Omega)$. Thus, concerning $I_v$, we can estimate
        \[ | \langle I_v, w \rangle_{\breve H_{\Gamma_\bullet}^{-1,r}(\Omega_\bullet)}
                | \le \| v \|_{L^{r}(\Omega_\bullet)} \; \| \mu
                \|_{L^\infty(\Omega;\mathcal L(\C^d))} \; \| \nabla \eta
                \|_{L^\infty(\Omega_\bullet)} \; \| w
                \|_{H_{\Gamma_\bullet}^{1,r'}(\Omega_\bullet)},
        \]
what implies the assertion.
        \qedhere
\end{enumerate}
\end{proof}
%
\begin{remark} \label{r-limit}
It is the lack of integrability for the gradient of $v$ (see the
counterexample in \cite[Ch.~4]{e/r/s}) together with the quality of the needed
Sobolev embeddings which limits the quality of the correction terms. In the
end it is this effect which prevents the applicability of the localization
procedure in Subsection~\ref{subsec-core} in higher dimensions -- at least
when one aims at a $q>d$.
\end{remark}
%
%
\begin{remark} \label{r-shrink}
If $v \in L^2(\Omega)$ is a regular distribution, then $v_\bullet$ is the regular 
distribution $(\eta v)|_{\Omega_\bullet}$.
\end{remark}
%
%
\begin{lemma} \label{l-reinterpret}
Let in the terminology of Lemma \ref{l-project} $\chi \in C^\infty(\R^d)$ be a
function with $\supp (\chi) \subseteq \Upsilon$ and $\chi \equiv 1$ in a
neighborhood of $\supp (\eta)$. Furthermore, for $q \in \left] 1, \infty
\right[$, we define for every $f \in
\breve H^{-1,q}_{\Gamma_\bullet}(\Omega_\bullet)$ the element
$f^\bullet \in \breve H^{-1,q}_\Gamma(\Omega)$ by $\langle f^\bullet, \psi
\rangle_{\breve H^{-1,q}_\Gamma(\Omega)} := \langle f, (\chi \psi)|_{\Omega_\bullet}
\rangle_{\breve H^{-1,q}_{\Gamma_\bullet}(\Omega_\bullet)}$, $\psi \in
H^{1,q'}_{\Gamma}(\Omega)$. (The definition is justified by Lemma
\ref{l-restr/ext}.) Then
\begin{enumerate}
\item For every $f \in \breve H^{-1,q}_{\Gamma_\bullet}(\Omega_\bullet)$ one has
        $f^\bullet \in \breve H^{-1,q}_{\Gamma}(\Omega)$, and the mapping 
        \[ 
\breve H^{-1,q}_{\Gamma_\bullet}(\Omega_\bullet) \ni f \mapsto f^\bullet
        \in \breve H^{-1,q}_{\Gamma}(\Omega) 
        \]
        is continuous.
\item For any $f \in \breve H^{-1,q}_{\Gamma}(\Omega)$ one has the identity $\bigl(
        f_\bullet \bigr)^\bullet = \eta f \in \breve H^{-1,q}_{\Gamma}(\Omega)$.
\item If $v \in H^{1,2}_\Gamma(\Omega)$ and $-\nabla \cdot \mu_\bullet \nabla
        (\eta v|_{\Omega_\bullet} ) \in
        \breve H^{-1,q}_{\Gamma_\bullet}(\Omega_\bullet)$, then
        \[ \bigl( -\nabla \cdot \mu_\bullet \nabla ( \eta  v|_{\Omega_\bullet}
                ) \bigr)^\bullet = -\nabla \cdot \mu \nabla ( \eta v ) \in
                \breve H^{-1,q}_{\Gamma}(\Omega).
        \]
\end{enumerate}
\end{lemma}
%
\begin{proof}
\begin{enumerate}
\item The mapping $f \mapsto f^\bullet$ is the adjoint to
        $H^{1,q'}_\Gamma(\Omega) \ni v \mapsto (\chi v)|_{\Omega_\bullet}$
        which acts continuously into
        $H^{1,q'}_{\Gamma_\bullet}(\Omega_\bullet)$, see
        Lemma~\ref{l-restr/ext}.
\item We only need to prove the assertion for elements $f \in L^q(\Omega)$,
        because $L^q(\Omega)$ is dense in $\breve H^{-1,q}_\Gamma(\Omega)$ and the
        mappings $\breve H^{-1,q}_\Gamma(\Omega) \ni f \mapsto \bigl( f_\bullet
        \bigr)^\bullet \in \breve H^{-1,q}_\Gamma(\Omega)$ and
        $\breve H^{-1,q}_\Gamma(\Omega) \ni f \mapsto \eta f \in
        \breve H^{-1,q}_\Gamma(\Omega)$ are both continuous. For $f \in L^q(\Omega)$
        the assertion follows directly from the definitions of $f_\bullet$ and
        $f^\bullet$.
\item For any $\psi \in H^{1,q'}_\Gamma (\Omega)$ we have
        \begin{align*}
          \bigl\langle \bigl( -\nabla \cdot \mu_\bullet \nabla ( \eta
                v|_{\Omega_\bullet} ) \bigr)^\bullet, \psi
                \bigr\rangle_{\breve H^{-1,q}_\Gamma(\Omega)} &= \bigl\langle -\nabla
                \cdot \mu_\bullet \nabla ( \eta v|_{\Omega_\bullet} ), ( \chi
                \psi ) |_{\Omega_\bullet}
                \bigr\rangle_{\breve H^{-1,q}_{\Gamma_\bullet} (\Omega_\bullet)} \\
          &= \int_{\Omega_\bullet} \mu_\bullet \nabla (\eta v) \cdot \nabla
                (\chi \psi) \, \dd \mathrm x = \int_{\Omega} \mu \nabla (\eta
                v) \cdot \nabla (\chi \psi) \, \dd \mathrm x \\
          &= \int_{\Omega} \mu \nabla (\eta v) \cdot \nabla \psi \, \dd
                \mathrm x = \langle -\nabla \cdot \mu \nabla (\eta v), \psi
                \rangle_{\breve H^{-1,q}_\Gamma (\Omega)},
        \end{align*}
        because $\eta \equiv 0$ on $\Omega \setminus \Upsilon$ and $\chi
        \equiv 1$ on $\supp(\eta)$.
        \qedhere
\end{enumerate}
\end{proof}
\subsection{Core of the proof of Theorem \ref{t-central}}\label{subsec-core}
We are now in the position to start the proof of Theorem~\ref{t-central}. We
first note that in any case the operator $- \nabla \cdot \mu \nabla$ admits
maximal parabolic regularity on the Hilbert space $\breve H^{-1,2}_\Gamma$,
since its negative generates an analytic semigroup on this space by
Proposition~\ref{p-basicl2}, cf. Remark~\ref{r-ebed}~iii). Thus, defining
\[ M_{\mathrm{MR}}:= \{ q \ge 2 \with - \nabla \cdot \mu \nabla
	\text{ admits maximal regularity on } \breve H^{-1, q}_\Gamma \}
\]
and $q_{\mathrm{MR}} := \sup M_{\mathrm{MR}}$, yields $q_{\mathrm{MR}}\ge 2$.
In the same way as for $q_{\mathrm{iso}}$ and using Lemma~\ref{l-maxint}, we
see by interpolation that $ M_{\mathrm{MR}}$ is $\{2 \}$ or an interval with
left endpoint $2$.

Our aim is to show that in fact $q_{\mathrm{MR}} \ge q_{\mathrm{iso}}^*$, so we assume that
 $q_{\mathrm{MR}} < q_{\mathrm{iso}}^*$. The main step towards a contradiction is contained
 in the following lemma.
%
\begin{lemma} \label{l-shrinkextend}
Let $\Omega$, $\Gamma$, $\Upsilon$, $\eta$, $\Omega_\bullet$,
$\Gamma_\bullet$, $\mu_\bullet$ be as before. Assume that $-\nabla \cdot
\mu_\bullet \nabla$ satisfies maximal parabolic regularity on
$\breve H^{-1,q}_{\Gamma_\bullet}(\Omega_\bullet)$ for all $q \in \left[ 2, \infty
\right[$ and that $-\nabla \cdot \mu \nabla$ satisfies maximal parabolic
regularity on $\breve H^{-1,q}_\Gamma(\Omega)$ for some $q \in \left[ 2, q_{\mathrm{iso}}
\right[$. If $r \in \left[ q, q^* \right[$ and $G \in
L^s(J;\breve H^{-1,r}_{\Gamma}(\Omega)) \hookrightarrow
L^s(J;\breve H^{-1,q}_{\Gamma}(\Omega))$, then the unique solution $V \in
W^{1,s}(J; \breve H^{-1,q}_\Gamma(\Omega)) \cap
L^s(J; dom_{\breve H_\Gamma^{-1,q}(\Omega)}(- \nabla \cdot \mu \nabla))$ of
\begin{equation} \label{e-parequation}
  V' - \nabla \cdot \mu \nabla V = G, \qquad V(T_0) = 0,
\end{equation}
even satisfies
\[  \eta V \in W^{1,s}(J;\breve H^{-1,r}_{\Gamma}(\Omega)) \cap
        L^s(J;dom_{\breve H^{-1,r}_{\Gamma}(\Omega)}(-\nabla \cdot \mu \nabla)).
\]
\end{lemma}
%
\begin{proof}
$V \in L^s(J;dom_{\breve H^{-1,q}_\Gamma(\Omega)}(-\nabla \cdot \mu \nabla))$ implies, due to
 our supposition $q \in \left[ 2, q_{\mathrm{iso}} \right[$ and
Remark~\ref{r-darunter}~ii), $V \in L^s(J;H^{1,q}_\Gamma(\Omega))$.
Of course, equation \eqref{e-parequation} is to be read as follows: For almost
all $t \in J$ it holds $-\nabla \cdot \mu \nabla \bigl( V(t) \bigr) = G(t) -
V'(t)$, where $V'$ is the derivative in the sense of $\breve H^{-1,q}_\Gamma$-valued
distributions. Hence, Lemma \ref{l-project}~ii) implies for almost all $t \in
J$
\begin{equation} \label{e-absch35}
  (V'(t))_\bullet - \nabla \cdot \mu_\bullet \nabla \bigl( (\eta
        V(t))|_{\Omega_\bullet} \bigr) = - \mu_\bullet \nabla
        V(t)|_{\Omega_\bullet} \cdot \nabla \eta|_{\Omega_\bullet} + I_{V(t)}
        + (G(t))_\bullet.
\end{equation}
Since by Lemma~\ref{l-project}~i) the mapping $\breve H^{-1,r}_\Gamma(\Omega) \ni
f \mapsto f_\bullet \in \breve H^{-1,r}_{\Gamma_\bullet}(\Omega_\bullet)$ is
continuous, we have $\bigl( G(\cdot) \bigr)_\bullet \in
L^s(J;\breve H^{-1,r}_{\Gamma_\bullet}(\Omega_\bullet))$. Moreover, the property $V
\in L^s(J;H^{1,q}_\Gamma(\Omega))$ and iii) of Lemma~\ref{l-project} assure
$- \mu_\bullet \nabla V(\cdot)|_{\Omega_\bullet} \cdot \nabla
\eta|_{\Omega_\bullet} + I_{V(\cdot)} \in
L^s(J;\breve H^{-1,r}_{\Gamma_\bullet}(\Omega_\bullet))$. Thus, the right hand side
of \eqref{e-absch35} is contained in
$L^s(J;\breve H^{-1,r}_{\Gamma_\bullet}(\Omega_\bullet))
\hookrightarrow L^s(J;\breve H^{-1,q}_{\Gamma_\bullet}(\Omega_\bullet))$.

Let us next inspect the term $(V'(t))_\bullet$: Since
$\breve H^{-1,q}_\Gamma(\Omega) \ni w \mapsto w_\bullet \in
\breve H^{-1,q}_{\Gamma_\bullet}(\Omega_\bullet)$ is linear and continuous, it equals
$(V_\bullet)'(t)$. But by Remark~\ref{r-shrink} the function $t \mapsto V_\bullet(t)$ is
 identical to the function $t \mapsto \bigl( \eta V(t)\bigr)
|_{\Omega_\bullet}$. Hence, $\bigl( \eta V(\cdot) \bigr)|_{\Omega_\bullet}$
satisfies the following equation in
$\breve H^{-1,q}_{\Gamma_\bullet}(\Omega_\bullet)$:
\begin{equation} \label{e-absch36}
  \bigl( (\eta V )|_{\Omega_\bullet} \bigr)' (t) - \nabla \cdot \mu_\bullet
        \nabla \bigl( (\eta V(t))|_{\Omega_\bullet} \bigr) = - \mu_\bullet
        \nabla V(t)|_{\Omega_\bullet} \cdot \nabla \eta|_{\Omega_\bullet} +
        I_{V(t)} + (G(t))_\bullet.
\end{equation}
By supposition, $-\nabla \cdot \mu_\bullet \nabla$ fulfills maximal parabolic
regularity in $\breve H^{-1,r}_{\Gamma_\bullet}(\Omega_\bullet)$. As the right
hand side of \eqref{e-absch36} is in fact from
$L^s(J;\breve H^{-1,r}_{\Gamma_\bullet}(\Omega_\bullet))$, this implies that there is
a unique function $U \in W^{1,s}(J;\breve H^{-1,r}_{\Gamma_\bullet}(\Omega_\bullet))
\cap L^s(J;dom_{\breve H^{-1,r}_{\Gamma_\bullet}(\Omega_\bullet)}(-\nabla \cdot
\mu_\bullet \nabla))$ which satisfies $U(T_0) = 0$ and
\begin{equation} \label{e-absch37}
  U'(t) - \nabla \cdot \mu_\bullet \nabla \bigl( U(t) \bigr) = - \mu_\bullet
        \nabla V(t)|_{\Omega_\bullet} \cdot \nabla \eta|_{\Omega_\bullet} +
        I_{V(t)} + (G(t))_\bullet
\end{equation}
as an equation in $L^s(J;\breve H^{-1,r}_{\Gamma_\bullet}(\Omega_\bullet))$. However,
this last equation can (using the embedding $\breve H^{-1,r}_{\Gamma_\bullet}(\Omega_\bullet)
\hookrightarrow \breve H^{-1,q}_{\Gamma_\bullet}(\Omega_\bullet)$) also be read as an
equation in $L^s(J;\breve H^{-1,q}_{\Gamma_\bullet}(\Omega_\bullet))$. Since the
solution is unique in $L^s(J;\breve H^{-1,q}_{\Gamma_\bullet}(\Omega_\bullet))$,
\eqref{e-absch36} and \eqref{e-absch37} together imply  $U = \bigl( \eta
V(\cdot) \bigr)|_{\Omega_\bullet}$ and, consequently,
\begin{equation} \label{e-contain07}
  \bigl( V(\cdot) \bigr)_\bullet = \bigl( \eta V(\cdot) \bigr)
        |_{\Omega_\bullet} \in
        W^{1,s}(J;\breve H^{-1,r}_{\Gamma_\bullet}(\Omega_\bullet)) \cap
        L^s(J;dom_{\breve H^{-1,r}_{\Gamma_\bullet}(\Omega_\bullet)}(-\nabla \cdot
        \mu_\bullet \nabla)),
\end{equation}
see Remark~\ref{r-shrink}.

We now aim at a re-interpretation of this regularity in terms of the space
$W^{1,s}(J;\breve H^{-1,r}_{\Gamma}(\Omega)) \cap 
L^s(J;dom_{\breve H^{-1,r}_{\Gamma}(\Omega)}(-\nabla \cdot \mu \nabla))$.
Observe that \eqref{e-contain07} implies 
$-\nabla \cdot \mu_\bullet \nabla \bigl( \bigl( \eta V(\cdot) \bigr)|_{\Omega_\bullet} \bigr) \in
L^s(J;\breve H^{-1,r}_{\Gamma_\bullet}(\Omega_\bullet))$. Applying Lemma~\ref{l-reinterpret}~iii),
 this gives
\begin{equation} \label{e-re0}
 -\nabla \cdot \mu \nabla \bigl( \eta V(\cdot) \bigr) \in
        L^s(J;\breve H^{-1,r}_{\Gamma}(\Omega)).
\end{equation}
Obviously, $V \in L^s(J;H^{1,q}_\Gamma)$ yields $\eta V \in
L^s(J;H^{1,q}_\Gamma)$, while $r \in \left] q, q^* \right[$ implies the
embedding $H^{1,q}_\Gamma \hookrightarrow L^{r} \hookrightarrow
\breve H^{-1,r}_\Gamma$. Hence, one obtains
\begin{equation} \label{e-imbedhq}
  \eta V \in L^s(J;H^{1,q}_\Gamma) \hookrightarrow  L^s(J;\breve H^{-1,r}_\Gamma).
\end{equation}
Combining this with \eqref{e-re0}, we find
\[ \eta V(\cdot) \in L^s(J;dom_{\breve H^{-1,r}_{\Gamma}(\Omega)}(-\nabla
	\cdot \mu \nabla)).
\]
On the other hand, \eqref{e-contain07} implies
\[ \bigl( \bigl( V(\cdot) \bigr)_\bullet \bigr)' \in
	L^{s}(J;\breve H^{-1,r}_{\Gamma_\bullet}(\Omega_\bullet)).
\]
By Lemma~\ref{l-reinterpret}~i), we have
$\bigl( \bigl( ( V(\cdot) )_\bullet \bigr)' \bigr)^\bullet \in L^s(J;\breve H^{-1,r}_{\Gamma}(\Omega))$.
 But as before $\bigl( \bigl( ( V(\cdot)
)_\bullet \bigr)' \bigr)^\bullet$ equals $\bigl( \bigl( ( V(\cdot) )_\bullet
\bigr)^\bullet \bigr)'$, which, by Lemma~\ref{l-reinterpret}~ii), is
$\bigl( \eta V(\cdot) \bigr)'$. Summing up, we get
\[ \bigl( \eta V(\cdot) \bigr)' \in L^{s}(J;\breve H^{-1,r}_{\Gamma}(\Omega)).
\]
Taking into account \eqref{e-imbedhq} again, this gives
\[ \eta V(\cdot) \in W^{1,s}(J;\breve H^{-1,r}_{\Gamma}(\Omega)),
\]
what proves the lemma.
\end{proof}

\begin{proof}[Proof of Theorem~\ref{t-central}]
For every $x \in \Omega$ let $\Xi_{\mathrm{x}} \subseteq \Omega$ be an open
cube, containing ${\mathrm{x}}$. Furthermore, let for any point ${\mathrm{x}}
\in \partial \Omega$ an open neighborhood be given according to the
supposition of the theorem (see Assumption \ref{a-groegerregulaer}).
Possibly shrinking this neighborhood to a smaller one, one obtains a new
neighborhood $\Upsilon_{\mathrm{x}}$, and a bi-Lipschitz, volume-preserving
mapping $\phi_{\mathrm{x}}$ from a neighborhood of
$\overline{\Upsilon_{\mathrm{x}}}$ into $\R^d$ such that $\phi_{\mathrm{x}}
( \Upsilon_{\mathrm{x}} \cap (\Omega \cup \Gamma) ) = \beta K_-$, $\beta
(K_- \cup \Sigma) $ or $\beta (K_- \cup \Sigma_0)$ for some $\beta
= \beta(\mathrm x)>0$. 

Obviously, the $\Xi_{\mathrm{x}}$ and $\Upsilon_{\mathrm{x}}$ together form an
open covering of $\overline \Omega$. Let $\Xi_{{\mathrm{x}}_1}, \ldots,
\Xi_{{\mathrm{x}}_k}, \Upsilon_{\mathrm{x}_{k+1}}, \ldots,
\Upsilon_{\mathrm{x}_l}$ be a finite subcovering and $\eta_1, \ldots, \eta_l$
a $C^\infty$ partition of unity, subordinate to this subcovering. Set
$\Omega_j := \Xi_{\mathrm{x}_j} = \Xi_{\mathrm{x}_j} \cap \Omega$ for $j \in
\{1, \ldots, k\}$ and $\Omega_j := \Upsilon_{\mathrm{x}_j} \cap \Omega$ for $j
\in \{k+1, \ldots, l\}$. Moreover, set $\Gamma_j := \emptyset$ for $j \in \{1,
\ldots, k\}$ and $\Gamma_j := \Upsilon_{\mathrm{x}_j} \cap \Gamma$ for $j \in
\{ k+1, \ldots, l \}$.

Denoting the restriction of $\mu$ to $\Omega_j$ by $\mu_j$, each operator
$-\nabla \cdot \mu_j \nabla$ satisfies maximal parabolic regularity in
$\breve H^{-1,q}_{\Gamma_j}(\Omega_j)$ for all $q \in [2,\infty[$ and all $j$, according to
 Theorem~\ref{t-maxparloc}.

Assuming now $q_{\mathrm{MR}} < q_{\mathrm{iso}}^*$, we may choose some $q \in
\left[2, q_{\mathrm{iso}} \right[$ with $q_{\mathrm{MR}} < q^*$. In order to
see this, we first observe that
\begin{equation} \label{e-SobMon}
  p \le q \iff p^* \le q^*
\end{equation}
holds, whenever $p^* < \infty$. Setting $q = \max \{ 2, d \tilde q/(d +
\tilde q) \}$ for some $\tilde q \in \left] q_{\mathrm{MR}},
q_{\mathrm{iso}}^* \right[$, this, together with $(d \tilde q / ( d +
\tilde q))^* = \tilde q$, yields immediately that $q^* = \max \{ 2^*,
\tilde q \} \ge \tilde q > q_{\mathrm{MR}}$. Furthermore, again by
\eqref{e-SobMon}, we have $q < q_{\mathrm{iso}}$, since $q^* <
q_{\mathrm{iso}}^*$ and finally $q \ge 2$ is guaranteed by the choice of
$q$.
Having the so chosen $q$ at hand, we take some $ r \in \left] \max \{ q,
q_{\mathrm{MR}} \}, q^* \right[$, which is possible due to $q < q^*$. Now, let
$G \in L^s(J;\breve H^{-1,r}_\Gamma)$ be given. Then by Lemma~\ref{l-shrinkextend}
the unique solution $V \in W^{1,s}(J;\breve H^{-1,q}_\Gamma) \cap
L^s(J;H^{1,q}_\Gamma)$ of \eqref{e-parequation} satisfies $\eta_j V \in
W^{1,s}(J;\breve H^{-1,r}_{\Gamma}(\Omega)) \cap L^s(J;dom_{\breve H^{-1,r}_{\Gamma}(\Omega)}
(-\nabla \cdot \mu \nabla))$ for every
$j$. This implies maximal parabolic regularity for $-\nabla \cdot \mu \nabla$ on 
$\breve H^{-1,r}_\Gamma$, in contradiction to $r > q_{\mathrm{MR}}$. Thus we
have $q_{\mathrm{MR}} \ge q_{\mathrm{iso}}^*$ and the proof is finished.
\end{proof}
%
\begin{remark} \label{r-REMArk}
Note that Theorem~\ref{t-central} already yields maximal regularity
of $-\nabla \cdot \mu \nabla$ on $\breve H^{-1, q}_\Gamma$ for all $q \in [2, 2^*[$
without any additional information on $dom_{\breve H^{-1,q}_{\Gamma}}(-\nabla \cdot
\mu \nabla)$ nor on $dom_{\breve H^{-1,q}_{\Gamma_j}(\Omega_j)}(-\nabla
\cdot \mu_j \nabla)$.

In the 2-$d$ case this already implies maximal regularity for every $q \in
[2,\infty[$. Taking into account Remark~\ref{r-darunter}~i), without further
knowledge on the domains we get in the 3-$d$ case every $q \in [2, 6 +
\varepsilon[$ and in the $4$-$d$ case every $q \in [2, 4 + \varepsilon[$,
where $\varepsilon$ depends on $\Omega, \Gamma, \mu$.
\end{remark}
\subsection{The operator $A$}
Next we carry over the maximal parabolic regularity result, up to
now proved for $-\nabla \cdot \mu \nabla$ on the spaces $\breve H^{-1,q}_\Gamma$, to
the operator $A$ and to a much broader class of distribution spaces. For this
we need the following perturbation result.

\begin{lemma} \label{l-relativ}
Suppose $q \ge 2$, $\varsigma \in \bigl] 1 -  \frac{1}{q}, 1 \bigr]$ and
$\varkappa \in L^\infty(\Gamma,\dd\sigma)$ and let $\Omega, \Gamma$ satisfy
Assumption~\ref{a-groegerregulaer}. 
If we define the mapping $Q
:dom_{\breve H^{-\varsigma,q}_\Gamma}(-\nabla \cdot \mu \nabla) \to \breve H^{-\varsigma,q}_\Gamma$ by
\[ \langle Q \psi, \varphi \rangle_{H^{-\varsigma,q}_\Gamma} := \int_\Gamma
        \varkappa \, \psi \, \overline \varphi \,\dd\sigma, \quad \varphi \in
        H^{\varsigma,q'}_\Gamma,
\]
then $Q$ is well defined and continuous. Moreover, it is relatively bounded with respect to
$-\nabla \cdot \mu \nabla$, when considered on the space
$\breve H^{-\varsigma,q}_\Gamma$, and the relative bound may be taken arbitrarily small.
\end{lemma}
%
\begin{proof}
One has for every $\psi \in dom_{\breve H^{-\varsigma,q}_\Gamma}(-\nabla \cdot \mu
\nabla) \hookrightarrow dom_{\breve H^{-1,q}_\Gamma}(-\nabla \cdot \mu \nabla)
\hookrightarrow H^{1,2}_\Gamma$
\begin{align}
  \| Q \psi \|_{\breve H^{-\varsigma,q}_\Gamma} &=
        \sup_{\|\varphi\|_{H^{\varsigma,q'}_\Gamma} = 1} | \langle Q \psi,
        \varphi \rangle_{\breve H^{-\varsigma,q}_\Gamma} | =
        \sup_{\|\varphi\|_{H^{\varsigma,q'}_\Gamma} = 1} \biggl| \int_\Gamma
        \varkappa \psi \overline \varphi \,\dd\sigma \biggr| \nonumber \\
  \label{e-Qesti}
        &\le \|\varkappa\|_{L^\infty(\Gamma,\dd\sigma)} 
       \| \psi \|_{L^q(\partial \Omega,\dd\sigma)}
        \sup_{\|\varphi\|_{H^{\varsigma,q'}_\Gamma} = 1} \| \varphi
        \|_{L^{q'}(\partial \Omega,\dd\sigma)},
\end{align}
where the last factor is finite according to Theorem \ref{t-embedbound}.
Let us first consider the case $q = 2$. Then \eqref{e-Qesti} can be further estimated (see \eqref{e-mazyaS68})
\[ \le  c \| \psi \|_{L^2(\partial \Omega,\dd\sigma)} \le c \| \psi
	\|_{H_\Gamma^{1,2}}^{1/2} \| \psi \|_{L^{2}}^{1/2}  \le c\| \psi
	\|_{H_\Gamma^{1,2}}^{3/4} \| \psi \|_{\breve H_\Gamma^{-1,2}}^{1/4} \le
	\epsilon \| \psi \|_{H_\Gamma^{1,2}} + \frac{c}{\epsilon^3} \| \psi
	\|_{\breve H_\Gamma^{-1,2}}
\]
by Young's inequality. Taking into account $dom_{\breve H^{-1,2}_\Gamma}(-\nabla
\cdot \mu \nabla) = H^{1,2}_\Gamma$, this proves the case $q = 2$. Concerning
the case $q > 2$, we make use of the embedding
\begin{equation} \label{e-hoelderemb}
  dom_{\breve H^{-\varsigma,q}_\Gamma} \bigl( -\nabla \cdot \mu \nabla \bigr)
	\hookrightarrow dom_{\breve H^{-1,q}_\Gamma} \bigl( -\nabla \cdot \mu \nabla
	\bigr) \hookrightarrow C^\alpha(\Omega)\quad \text{for some } \alpha
	= \alpha(q) > 0,
\end{equation}
if $q > d$ (see \cite{griehoel}). Thus, for $q > d + \frac{1}{2}$ the term
$\|\psi\|_{L^q(\partial \Omega,\dd\sigma)}$ in \eqref{e-Qesti} can be
estimated by $(\sigma(\partial \Omega))^\frac{1}{q} \| \psi
\|_{C(\overline \Omega)}$, what shows, due to \eqref{e-hoelderemb}, the
asserted continuity of $Q$, if $q > d + \frac{1}{2}$. Since
$dom_{\breve H^{-\varsigma,q}_\Gamma} \bigl( -\nabla \cdot \mu \nabla \bigr)
\hookrightarrow C^\alpha(\Omega) \hookrightarrow C(\overline \Omega)$ is
compact and $C(\overline \Omega) \hookrightarrow \breve H^{-\varsigma,q}_\Gamma$ is
continuous and injective, we may apply Ehrling's lemma (see \cite[Ch.~I,
Prop.~7.3]{wloka}) and estimate
\[ 
\| \psi \|_{C(\overline \Omega)} \le \epsilon \| \psi \|_{dom_{\breve H^{-\varsigma,q}_\Gamma}
   (-\nabla\cdot\mu\nabla )} + \beta(\epsilon) \| \psi \|_{\breve H^{-\varsigma,q}_\Gamma},
 \quad \psi \in dom_{\breve H^{-\varsigma,q}_\Gamma} ( -\nabla \cdot \mu \nabla ),
\]
for arbitrary $\epsilon > 0$. Together with \eqref{e-Qesti} this yields the
second assertion for $q > d + \frac{1}{2}$.

Concerning the remaining case $q \in \left] 2, d + \frac{1}{2} \right]$, we employ the representation
\begin{equation} \label{e-thetadef}
  \breve H_\Gamma^{-1,q} = [\breve H_\Gamma^{-1,2d}, \breve H_\Gamma^{-1,2}]_{\theta} \quad \text
	{with } \theta = \frac{1}{q} \cdot \frac{2d-q}{d-1}
\end{equation}
(see Corollary \ref{c-antili}) and will invest the knowledge
$dom_{\breve H_\Gamma^{-1,2d}}(-\nabla\cdot\mu\nabla ) \hookrightarrow L^\infty$ and
$dom_{\breve H_\Gamma^{-1,2}}(-\nabla\cdot\mu\nabla ) = H^{1,2}_\Gamma$. Clearly,
\eqref{e-thetadef} implies
\begin{equation} \label{e-interpoldefber}
  dom_{\breve H_\Gamma^{-1,q}}(-\nabla \cdot \mu \nabla) = 
	[dom_{\breve H_\Gamma^{-1,2d}}(-\nabla \cdot \mu \nabla),
	dom_{\breve H_\Gamma^{-1,2}}(-\nabla \cdot \mu \nabla)]_\theta.
\end{equation}
Taking $q = 2d$ in \eqref{e-hoelderemb} and combining this with
the embedding $C^\alpha \hookrightarrow L^r$ for
any finite $r$, \eqref{e-interpoldefber} yields
\[ dom_{\breve H_\Gamma^{-1,q}}(-\nabla \cdot \mu \nabla) \hookrightarrow
[L^r, H^{1,2}_\Gamma]_\theta = H_\Gamma^{\theta, \frac{2}{\theta} - \delta(r, \theta)},
\]
where $\delta(r, \theta) \searrow 0$ for $r \to \infty$, see Proposition \ref{p-interpol}.
If  $q \in \left] 2, d + \frac{1}{2} \right]$, then it is clear from the definition of $\theta$
 that $\theta \ge \frac{1}{q} \cdot \frac{d-\frac{1}{2}}{d-1} > \frac{1}{q}$. On the other hand,
 one easily verifies $\frac{2}{\theta} \in \bigl] q, q \frac{2(d-1)}{d-\frac{1}{2}} \bigr]$.
 Thus, choosing $r$ large enough, one gets for every $q \in \bigl] 2,d + \frac{1}{2} \bigr]$ a
 continuous embedding
\[ dom_{\breve H_\Gamma^{-1,q}}(-\nabla \cdot \mu \nabla) \hookrightarrow
 H_\Gamma^{\frac{1}{q} \frac{d-\frac{1}{2}}{d-1},q},
\]
what gives a compact embedding
\begin{equation} \label{e-glei2}
  dom_{\breve H_\Gamma^{-\varsigma,q}}(-\nabla \cdot \mu \nabla) \hookrightarrow
	dom_{\breve H_\Gamma^{-1,q}}(-\nabla \cdot \mu \nabla) \hookrightarrow
 H_\Gamma^{\frac{1}{q} \frac{d-\frac{3}{4}}{d-1},q}.
\end{equation}
Due to Theorem \ref{t-embedbound}, the term $\| \psi
\|_{L^q(\partial \Omega,\dd\sigma)}$ in \eqref{e-Qesti} may be estimated by
$c\| \psi \|_{H_\Gamma^{\frac {1}{q}\frac {d-\frac{3}{4}}{d-1},q}}$.
But, in view of the compactness of the mapping \eqref{e-glei2} and the
continuity of the injection
$H_\Gamma^{\frac{1}{q} \frac{d-\frac{3}{4}}{d-1},q} \hookrightarrow
\breve H_\Gamma^{-\varsigma,q}$ one may also here apply Ehrling's lemma and estimate
\[ \| \psi \|_{H_\Gamma^{\frac {1}{q}\frac {d-\frac {3}{4}}{d-1},q}} \le
	\epsilon \| \psi \|_{dom_{\breve H_\Gamma^{-\varsigma,q}}(-\nabla \cdot \mu
	\nabla)} + \beta(\epsilon) \|\psi\|_{\breve H_\Gamma^{-\varsigma,q}}
\]
for $\epsilon$ arbitrarily small. Together with \eqref{e-Qesti} this shows the
assertion in the last case.
\end{proof}
%
\begin{theorem} \label{t-randterm}
Suppose $q \ge 2$, $\varkappa \in L^\infty(\Gamma,\dd\sigma)$ and let $\Omega, \Gamma$ satisfy
Assumption~\ref{a-groegerregulaer}.
\begin{enumerate} 
\item If $\varsigma \in \bigl] 1 - \frac {1}{q}, 1 \bigr]$, then
        $dom_{\breve H^{-\varsigma,q}_\Gamma}(-\nabla \cdot \mu \nabla) =
        dom_{\breve H^{-\varsigma,q}_\Gamma}(A)$.
\item If $\varsigma \in \bigl] 1 - \frac{1}{q}, 1 \bigr]$ and $-\nabla \cdot
	\mu \nabla$ satisfies maximal parabolic regularity on
	$\breve H^{-\varsigma,q}_\Gamma$, then $A$ also does.
\item The operator $A$ satisfies maximal parabolic regularity on $L^2$. If
        $\varkappa \ge 0$, then $A$ satisfies maximal parabolic regularity on
        $L^p$ for all $p \in \left] 1, \infty \right[$.
\item   Suppose that $-\nabla \cdot \mu \nabla$ satisfies maximal parabolic
	regularity on $\breve H^{-1,q}_\Gamma$. Then $A$ satisfies maximal parabolic
	regularity on any of the interpolation spaces
        \begin{align*}
          [L^2, \breve H^{-1,q}_\Gamma]_\theta &, \quad  \theta \in
                [0,1],
          \intertext{or}
                (L^2, \breve H^{-1,q}_\Gamma)_{\theta,s} &, \quad 
                \theta \in [0,1], \ s \in \left] 1, \infty \right[.
        \end{align*}
	Let $\varkappa \ge 0$ and $p \in \left] 1, \infty \right[$ in case of
	$d = 2$ or $p \in [\bigl( \frac{1}{2} + \frac{1}{d} \bigr)^{-1},
	\infty[$ if $d \ge 3$. Then $A$ also satisfies maximal parabolic
	regularity on any of the interpolation spaces
        \begin{equation} \label{e-interpolXc0}
          [L^p, \breve H^{-1,q}_\Gamma]_\theta , \quad  \theta \in [0,1],
\end{equation}
          or
\begin{equation} \label{e-interpolXmodi}
                (L^p, \breve H^{-1,q}_\Gamma)_{\theta,s} , \quad \theta \in [0,1], \
		s \in \left] 1, \infty \right[.
        \end{equation}
\end{enumerate}
\end{theorem}
%
\begin{proof}
\begin{enumerate}
\item By Lemma~\ref{l-relativ}, if $\psi \in dom_{\breve H^{-\varsigma,q}}(-\nabla
        \cdot \mu \nabla)$, then $Q \psi$ is well defined and one has the
        equality $A \psi = -\nabla \cdot \mu \nabla \psi + Q \psi$ by
        definition of $A$. Thus, the assertion follows from the relative
        boundedness with relative bound smaller than $1$, shown in
        Lemma~\ref{l-relativ}, and a classical perturbation theorem, see
        \cite[Ch.~IV.1]{kato}.
\item The assertion is also proved by means of a -- highly nontrivial --
        perturbation theorem (see \cite{kunst/weis}), which states that, if
	$X$ is a UMD space and a densely defined, closed operator $B$
	satisfies maximal parabolic regularity on $X$, then $B + B_0$ also
	satisfies maximal parabolic regularity on $X$, provided $dom_X(B_0)
	\supseteq dom_X(B)$ and $B_0$ is relatively bounded with respect to
	$B$ with arbitrarily small relative bound. In our case,
	$H^{-1,q}_\Gamma$ is -- as the dual of the closed subspace
	$H^{1,q'}_\Gamma$ of the UMD space $H^{1,q'}$ -- itself a UMD space,
	see \cite[Ch.~III.4.5]{amannbuch} and \cite[Ch.~6.1]{arend}.
        $H^{-1,q}_\Gamma$ is the isometric image of $ \breve H^{-1,q}_\Gamma$
	under the mapping which assigns to $f \in  \breve H^{-1,q}_\Gamma$ the
	linear form $ H^{1,q'}_\Gamma \ni \psi \to \langle f, \overline
	\psi\rangle _{\breve H^{-1,q}_\Gamma}$. Hence, $ \breve
	H^{-1,q}_\Gamma$ is also a UMD space. Finally, $\breve
	H^{-\varsigma,q}_\Gamma$ is a  complex interpolation space between the
	UMD space $\breve H^{-1,q}_\Gamma$ and the UMD space $L^q$ (see Remark
	\ref{r-charact} below), and consequently also a UMD space. Hence, an
	application  of Lemma~\ref{l-relativ} yields the result.
\item The first assertion follows from Proposition~\ref{p-basicl2}~ii) and
	Remark \ref{r-ebed} iii). The second is shown in \cite[Thm.~7.4]{gkr}.
\item   Under the given conditions on $p$, we
	have the embedding $L^p \hookrightarrow \breve H^{-1,2}_\Gamma$.
        Thus, the assertion follows from the preceding points and Lemma~\ref{l-maxint}.
	\qedhere
\end{enumerate}
\end{proof}
%
\begin{remark} \label{r-charact}
The interpolation spaces 
$ [L^p, H^{-1,q}_\Gamma]_\theta $  ($\theta \in [0,1]$) and 
$(L^p, H^{-1,q}_\Gamma)_{\theta,s}$ ($ \theta \in [0,1], s \in \left] 1, \infty \right[$)
 are characterized in \cite{ggkr}, see in particular Remark~3.6. Identifying each 
$f \in L^q$ with the anti-linear form $ L^{q'} \ni \psi \to \int_\Omega f \overline{\psi} \;
\dd \mathrm x$ and using again the retraction/coretraction theorem with the coretraction
from Corollary \ref{c-antili}, one easily identifies the interpolation spaces in \eqref{e-interpolXc0}
and  \eqref{e-interpolXmodi}.
 In particular, this yields $  \bigl[ L^{q_0},\breve H_\Gamma^{-1,q_1}
 \bigr]_{\theta}=\breve H_\Gamma^{-\theta,q}$ if $ \theta \neq 1 - \frac{1}{q}$.
\end{remark}
%
%
\begin{corollary} \label{c-analyt}
Let $\Omega$ and $\Gamma$ satisfy Assumption~\ref{a-groegerregulaer}. The operator
$-A$ generates analytic semigroups on all spaces $\breve H^{-1,q}_\Gamma$ if $q \in
[2,q^*_{\mathrm{iso}}[$ and on all the interpolation spaces occurring in Theorem~\ref{t-randterm},
there $q$ also taken from $[2,q^*_{\mathrm{iso}}[$. Moreover, if $\varkappa
\ge 0$, the following resolvent estimates are valid:
\begin{equation} \label{e-resol}
  \| (A + 1 + \lambda)^{-1} \|_{\mathcal L(\breve H^{-1,q}_\Gamma)} \le
        \frac{c_q}{1 + |\lambda|}, \quad \Re \lambda \ge 0 .
\end{equation}
\end{corollary}
%
\begin{proof}
The first assertion is implied by Theorem \ref{t-central} and
Remark~\ref{r-ebed} ii), which gives \eqref{e-resol} for $\lambda \in \gamma
+\Sigma_\kappa$ with a fixed $\gamma \in \R$ and fixed $\kappa >\pi/2$. On the
other hand, the resolvent of $A_0$ is compact (see
Proposition~\ref{p-basicl2}), what, due to Lemma \ref{l-relativ}, remains true
also for $A$, see \cite[Ch.~IV.1]{kato}. Since no $\lambda$ with $\Re \lambda
\le 0$ is an eigenvalue,
\[
\sup_{\lambda \in \{\lambda: \Re \lambda \ge 0\}\setminus (\gamma +\Sigma_\kappa)} (|\lambda| + 1)
 \| (A + 1 + \lambda)^{-1} \|_{\mathcal L(\breve H^{-1,q}_\Gamma)} < \infty, 
\]
because $\{\lambda: \Re \lambda \ge 0\}\setminus (\gamma +\Sigma_\kappa)$ is compact.
\end{proof}
%
%
%
%
\section{Nonlinear parabolic equations} \label{sec-Quasilin}
%
%
%
%
In this chapter we will apply maximal parabolic regularity for the treatment
of quasilinear parabolic equations which are of the (formal) type
\eqref{e-quasiformal}. Concerning all the occurring operators we will formulate
precise requirements in Assumption~\ref{a-daten} below.

The outline of the chapter is as follows: First we give a motivation for the
choice of the Banach space we will regard
\eqref{e-quasiformal}/\eqref{e-boundcond77} in. Afterwards we show that
maximal parabolic regularity, combined with regularity results for the
elliptic operator, allows to solve this problem. Below we will transform
\eqref{e-quasiformal}/\eqref{e-boundcond77} to a problem
\begin{equation} \label{5.2}
  \left\{ \begin{aligned}
        u'(t) + \mathcal B \bigl( u(t) \bigr) u(t) &= \mathcal S(t,u(t)),\quad
		t \in J, \\
        u(T_0) &= u_0.
        \end{aligned} \right.
\end{equation}
To give the reader already here an idea what properties of the operators
$-\nabla \cdot \mathcal G(u) \mu \nabla$ and of the corresponding Banach space
are required, we first quote the result on existence and uniqueness for
abstract quasilinear parabolic equations (due to Cl\'ement/Li \cite{cle/li} and
Pr\"uss \cite{pruess}) on which our subsequent considerations will base.
%
\begin{proposition} \label{p-pruess}
Suppose that $B$ is a closed operator on some Banach space $X$ with dense
domain $ D$, which satisfies maximal parabolic regularity on $X$.
Suppose further $u_0 \in (X, D)_{1-\frac {1}{s},s}$ and
$\mathcal{B}: J \times (X, D)_{1-\frac {1}{s},s} \to \mathcal{L}
({D}, X)$ to be continuous with $B = \mathcal{B} (T_0,u_0)$. Let,
in addition, $\mathcal S : J \times (X, D)_{1-\frac {1}{s},s} \to X$ be a
Carath\'eodory map and assume the following Lipschitz conditions on
$\mathcal B$ and $\mathcal S$:
\begin{itemize}
\item[$({\bf B})$] For every $M > 0$ there exists a constant $C_M > 0$, such
        that for all $t \in J$
        \[ \| \mathcal B(t,u) - \mathcal B(t, \tilde{u}) \|_{\mathcal
                L( D,X)} \le C_M \; \| u - \tilde{u}
                \|_{(X, D)_{1-\frac {1}{s},s}} \;\; \text {if} \;\;
                \|u\|_{(X, D)_{1-\frac {1}{s},s}},\; \| \tilde{u}
                \|_{(X, D)_{1-\frac {1}{s},s}} \le M.
        \]
\item[$({\bf R})$] $\mathcal S(\cdot, 0) \in L^s(J;X)$ and for each $M>0$ there is a
        function $h_M \in L^s(J)$, such that
        \[ \| \mathcal S(t,u) - \mathcal S(t,\tilde{u})\|_X \le h_M(t) \; \| u
		- \tilde{u} \|_{(X, D)_{1-\frac {1}{s},s}}
        \]
        holds for a.a. $t \in J$, if
        $\|u\|_{(X, D)_{1-\frac {1}{s},s}},
        \|\tilde{u}\|_{(X, D)_{1-\frac {1}{s},s}} \leq M$.
\end{itemize}
Then there exists $T^* \in J$, such that \eqref{5.2} admits a unique solution
$u$ on $]T_0,T^*[$ satisfying
\[ u \in W^{1,s}(]T_0,T^*[;X) \cap L^s(]T_0,T^*[; D).
\]
\end{proposition}
%
%
\begin{remark} \label{r-reell}
Up to now we were free to consider complex Banach spaces.
But the context of equations like \eqref{e-quasiformal} requires real spaces, in particular in
view of the quality of the superposition operator $\mathcal F$. Therefore, from
this moment on we use the real versions of the spaces. In particular,
$H^{-\varsigma,q}_\Gamma$ is now understood as the dual of the real space
$H^{\varsigma,q'}_\Gamma$ and clearly can be identified with the set of
anti-linear forms on the complex space $H^{\varsigma, q'}_\Gamma$ that take
real values when applied to real functions.

Fortunately, the property of maximal parabolic regularity is maintained for the
restriction of the operator $A$ to the real spaces in case of a real function
$\varkappa$, as $A$ then commutes with complex conjugation.
\end{remark}
%
We will now give a motivation for the choice of the Banach space $X$ we will
use later. It is not hard to see that $X$ has -- in view of the applicability
of Proposition~\ref{p-pruess} -- to fulfill the subsequent demands:
\begin{itemize}
\item[a)] The operators $A$, or at least the operators $-\nabla \cdot \mu
        \nabla$, defined in \eqref{e-defellip}, must satisfy maximal parabolic
        regularity on $X$.
\item[b)] As in the classical theory (see \cite{lady}, \cite{gia/stru},
        \cite{stru} and references therein) quadratic gradient terms of the
        solution should be admissible for the right hand side.
\item[c)] The operators $-\nabla \cdot \mathcal G(u) \mu \nabla$ should behave
        well concerning their dependence on $u$, see condition ({\bf B})
        above.
\item[d)] $X$ has to contain certain measures, supported on Lipschitz
        hypersurfaces in $\Omega$ or on $\partial \Omega$ in order to allow
        for surface densities on the right hand side or/and for inhomogeneous
        Neumann conditions.
\end{itemize}
The condition in a) is assured by Theorem~\ref{t-central} and
Theorem~\ref{t-randterm} for a great variety of Banach spaces, among them
candidates for $X$.
Requirement b) suggests that one should have $dom_X(-\nabla \cdot \mu \nabla)
\hookrightarrow H^{1,q}_\Gamma$ and $L^\frac{q}{2} \hookrightarrow X$. Since
$- \nabla \cdot \mu \nabla$ maps $H^{1,q}_\Gamma$ into $H^{-1,q}_\Gamma$, this
altogether leads to the necessary condition
\begin{equation} \label{e-hypoth} 
  L^{\frac{q}{2}} \hookrightarrow X \hookrightarrow H^{-1,q}_\Gamma.
\end{equation}
Sobolev embedding shows that $q$ cannot be smaller than the space dimension
$d$. Taking into account d), it is clear that $X$ must be a space of
distributions which (at least) contains surface densities.
In order to recover the desired property $dom_X(- \nabla \cdot \mu \nabla)
\hookrightarrow H^{1,q}_\Gamma$ from the necessary condition in
\eqref{e-hypoth}, we make for all what follows
this general
%
\begin{assumption} \label{a-q}
There is a $q > d$, such that $-\nabla \cdot \mu \nabla + 1:
H^{1,q}_\Gamma \to H^{-1,q}_\Gamma$ is a topological isomorphism.
\end{assumption}
%
%
\begin{remark} \label{r-exmp}
By Remark~\ref{r-darunter}~i) Assumption~\ref{a-q} is always fulfilled for
$d=2$. On the other hand for $d \ge 4$ it is generically false in case of mixed
boundary conditions, see \cite{shamir} for the famous counterexample.
Moreover, even in the Dirichlet case, when the domain $\Omega$ has only a
Lipschitz boundary or the coefficient function $\mu$ is constant within
layers, one cannot expect $q \ge 4$, see \cite{jer/ke} and
\cite{e/k/r/s}.

In Section~\ref{sec-Examples} we will present examples for domains $\Omega$,
coefficient functions $\mu$ and Dirichlet boundary parts $\Omega \setminus
\Gamma$, for which Assumption~\ref{a-q} is fulfilled.
\end{remark}
%
From now on we fix some $q > d$, for which Assumption~\ref{a-q} holds.

As a first step we will show that Assumption~\ref{a-q} carries over to a broad
class of modified operators.
%
\begin{lemma} \label{l-modi}
Assume that $\xi$ is a real valued, uniformly continuous function on $\Omega$
that admits a lower bound $\underline \xi > 0$. Then the operator $-\nabla
\cdot \xi \mu \nabla + 1$ also is a topological isomorphism between
$H^{1,q}_\Gamma$ and $H^{-1,q}_\Gamma$.
\end{lemma}
%
\begin{proof}
We identify $\xi$ with its (unique) continuous continuation to the closure
$\overline{\Omega}$ of $\Omega$. Furthermore, we observe that for any
coefficient function $\omega$ the inequality
\begin{equation} \label{e-koeffz}
  \|\nabla \cdot \omega \nabla\|_{\mathcal L(H^{1,q}_\Gamma,H^{-1,q}_\Gamma)}
        \le \|\omega\|_{L^\infty(\Omega;\mathcal L(\R^d))}
\end{equation}
holds true. Next, by Assumption~\ref{a-q} and Corollary~\ref{c-analyt} it is
clear that
\[ \sup_{\mathrm{y} \in \overline{\Omega}} \|\bigl( -\nabla \cdot
        \xi(\mathrm{y}) \mu \nabla + 1 \bigr)^{-1}
        \|_{\mathcal L(H^{-1,q}_\Gamma,H^{1,q}_\Gamma)} \le
        \frac{1}{\underline \xi} \sup_{\mathrm{y} \in \overline{\Omega}} \|
        \bigl( -\nabla \cdot \mu \nabla + (\xi(\mathrm{y}))^{-1} \bigr)^{-1}
        \|_{\mathcal L(H^{-1,q}_\Gamma,H^{1,q}_\Gamma)} =: \gamma
\]
is finite. Let for any $\mathrm{x} \in \overline{\Omega}$ a ball
$\mathcal B_{\mathrm{x}}$ around $\mathrm{x}$ be given, such that
\begin{equation} \label{e-defini}
  \gamma \; \sup_{\mathrm{y} \in \mathcal B_\mathrm{x} \cap \overline\Omega}
        | \xi(\mathrm{x}) - \xi(\mathrm{y}) |
        \|\mu\|_{L^\infty(\Omega;\mathcal L(\R^d))} < 1.
\end{equation}
Then, we choose a finite subcovering $\mathcal B_{\mathrm x_1}, \dots, \mathcal B_{\mathrm x_k}$
 of $\overline \Omega$ and a partition of unity $\eta_1,
\dots, \eta_k$ subordinate to this subcovering, and we set
$\Lambda_{\mathrm x} := \mathcal B_{\mathrm x} \cap \Omega$.

Assume that $f \in H^{-1,q}_\Gamma \subseteq H^{-1,2}_\Gamma$ and $v \in
H^{1,2}_\Gamma$ is a solution of $-\nabla \cdot \xi \mu \nabla v + v = f$.
Then a calculation, completely analogous to \eqref{e-localizeeq} (choose there
$\Upsilon$ so big that $\overline {\Omega} \subseteq \Upsilon$) shows that the
function $u := \eta_j v$ satisfies the equation
\begin{equation} \label{e-locglei00}
  - \nabla \cdot \xi \mu \nabla u + u = \eta_j f - \xi \mu \nabla v \cdot
        \nabla \eta_j + I_j
\end{equation}
in $H^{-1,2}_\Gamma$, where $I_j$ is the distribution $w \mapsto \int_\Omega v
\xi \mu \nabla \eta_j \cdot \nabla w \, \dd \mathrm{x}$. Then applying Lemma~\ref{l-project}~iii)
 with the same 'big' $\Upsilon$, we get that the
right hand side of \eqref{e-locglei00} is from $H^{-1,q}_{\Gamma}$, since $f
\in H^{-1,q}_\Gamma$. If we define the function $\xi_j$ on $\Omega$ by
\[ \xi_j(\mathrm{y}) = \begin{cases}
                \xi(\mathrm{y}), &\text{if } \mathrm{y} \in
                        \Lambda_{\mathrm x_j} \\
                \xi(\mathrm{x}_j), &\text{elsewhere in } \Omega,
        \end{cases}
\]
then $u = \eta_j v$ satisfies besides \eqref{e-locglei00} also the equation
\[ - \nabla \cdot \xi_j \mu \nabla u + u = \eta_j f - \xi \mu \nabla v \cdot
        \nabla \eta_j + I_j,
\]
because $\xi_j = \xi$ on the support of $u$.
But we have, according to \eqref{e-koeffz} and \eqref{e-defini}
\begin{align*}
  & \bigl\| \bigl( -\nabla \cdot \xi_j \mu \nabla + 1 - (-\nabla \cdot
        \xi(\mathrm{x}_j) \mu \nabla + 1) \bigr) \bigl( - \nabla \cdot
        \xi(\mathrm{x}_j) \mu \nabla + 1 \bigr)^{-1}
        \bigr\|_{\mathcal L(H^{-1,q}_\Gamma)} \\
  \le\;& \| -\nabla \cdot \xi_j \mu \nabla + 1 - (-\nabla \cdot
        \xi(\mathrm{x}_j) \mu \nabla + 1)
        \|_{\mathcal L(H^{1,q}_\Gamma, H^{-1,q}_\Gamma)} \| (- \nabla \cdot
        \xi(\mathrm{x}_j) \mu \nabla + 1 )^{-1}
        \|_{\mathcal L(H^{-1,q}_\Gamma, H^{1,q}_\Gamma)} \\
  \le\;& \gamma \; \sup_{\mathrm{y} \in \Lambda_{\mathrm x_j}}
        | \xi(\mathrm{x}_j) - \xi(\mathrm{y}) |
        \|\mu\|_{L^\infty(\Omega;\mathcal L(\R^d))} < 1.
\end{align*}
Thus, by a classical perturbation result (see \cite[Ch.~IV.1]{kato}), the
operator $-\nabla \cdot \xi_j \mu \nabla +1$ also provides a topological
isomorphism between $H^{1,q}_\Gamma$ and $H^{-1,q}_\Gamma$. Hence, for every
$j$ we have $\eta_j v \in H^{1,q}_\Gamma$, and, hence, $v \in H^{1,q}_\Gamma$.
So the assertion is implied by the open mapping theorem.
\end{proof}
In this spirit, one could now suggest $X := H^{-1,q}_\Gamma$ to be a good
choice for the Banach space, but in view of condition ({\bf R}) the right hand
side of \eqref{5.2} has to be a continuous mapping from an interpolation space
$(dom_X(A),X)_{1 - \frac{1}{s},s}$ into $X$. Chosen  $X := H^{-1,q}_\Gamma$,
for elements $\psi \in (dom_X(A),X)_{1 - \frac{1}{s},s} = (H^{1,q}_\Gamma,
H^{-1,q}_\Gamma)_{1 - \frac{1}{s},s}$ the expression $|\nabla \psi|^2$ cannot
be properly defined and, if so, will not lie in $H^{-1,q}_\Gamma$ in general.
This shows that $X := H^{-1,q}_\Gamma$ is not an appropriate choice, but we
will see that $X := H^{-\varsigma,q}_\Gamma$, with $\varsigma$ properly
chosen, is.
%
\begin{lemma} \label{l-space}
Put $X := H^{-\varsigma,q}_\Gamma$ with $\varsigma \in \left[ 0, 1 \right[ \setminus
\{\frac {1}{q}, 1-\frac {1}{q}\}$.
Then
\begin{enumerate}
\item For every $\tau \in \bigl] \frac{1+\varsigma}{2},1 \bigr[$ there is a continuous embedding
        $(X, dom_X(-\nabla \cdot \mu \nabla))_{\tau,1}
        \hookrightarrow H^{1,q}_\Gamma$.
\item If $\varsigma \in [\frac {d}{q}, 1]$, then $X$ has a predual $X_* =
        H^{\varsigma,q'}_\Gamma$ which admits the continuous, dense injections
        $H^{1,q'}_\Gamma \hookrightarrow X_* \hookrightarrow
        L^{(\frac {q}{2})'}$ that by duality clearly imply \eqref{e-hypoth}.
        Furthermore, $H^{1,q}_\Gamma$ is a multiplier space for $X_*$.
\end{enumerate}
\end{lemma}
%
\begin{proof}
\begin{enumerate}
\item
        $-\nabla \cdot \mu \nabla$ satisfies resolvent estimates
        \begin{equation} \label{e-interpol00}
          \| \bigl( -\nabla \cdot \mu \nabla + 1 + \lambda \bigr)^{-1}
                \|_{\mathcal L(Y)} \le \frac{c}{1+\lambda}, \quad \lambda \in
                \left[ 0, \infty \right[,
        \end{equation}
        if $Y = H^{-1,q}_\Gamma$ or $Y = L^q$, see Corollary~\ref{c-analyt}.
        In view of \eqref{e-interpol02} then \eqref{e-interpol00} also holds
        for $X$. This enables us to define fractional
        powers for $-\nabla \cdot \mu \nabla + 1$ on each of the occurring
        spaces. According to \eqref{e-interpol04} and Assumption~\ref{a-q} one
	has
        \begin{align*}
          H^{-\varsigma,q}_\Gamma &= [H^{-1,q}_\Gamma, H^{1,q}_\Gamma
                ]_\frac{1-\varsigma}{2} = [H^{-1,q}_\Gamma,
                dom_{H^{-1,q}_\Gamma}(-\nabla \cdot \mu \nabla + 1)
                ]_\frac{1-\varsigma}{2} \\
          &\hookrightarrow
                dom_{H^{-1,q}_\Gamma}((-\nabla \cdot \mu \nabla + 1)^\varrho),
        \end{align*}
        if $\varrho \in \left] 0, \frac{1 - \varsigma}{2} \right[$, see
        \cite[Ch.~1.15.2]{triebel}. Thus, $(-\nabla \cdot \mu \nabla + 1
        )^\varrho \in \mathcal L(H^{-\varsigma,q}_\Gamma,H^{-1,q}_\Gamma)$, if
        $\varrho \in \left] 0, \frac{1 - \varsigma}{2} \right[$.
        Consequently, we can estimate
        \begin{align*}
          & \| ( -\nabla \cdot \mu \nabla + 1)^{\varrho-1}
                \|_{\mathcal L(H^{-\varsigma,q}_\Gamma,H^{1,q}_\Gamma)} \\
          \le\;& \| ( -\nabla \cdot \mu \nabla + 1)^\varrho
                \|_{\mathcal L(H^{-\varsigma,q}_\Gamma,H^{-1,q}_\Gamma)}
                \| ( -\nabla \cdot \mu \nabla + 1)^{-1}
                \|_{\mathcal L(H^{-1,q}_\Gamma,H^{1,q}_\Gamma)} < \infty.
        \end{align*}
        Clearly, this means $dom_{H^{-\varsigma,q}_\Gamma} \bigl( (-\nabla
	\cdot \mu \nabla + 1)^{1-\varrho} \bigr) \hookrightarrow
	H^{1,q}_\Gamma$. Putting $\tau:=1-\varrho$, this implies
        \[ \bigl(H^{-\varsigma,q}_\Gamma, dom_{H^{-\varsigma,q}_\Gamma}
                (-\nabla \cdot \mu \nabla + 1) \bigr)_{\tau,1}
                \hookrightarrow dom_{H^{-\varsigma,q}_\Gamma} \bigl( ( -\nabla
                \cdot \mu \nabla + 1)^{\tau} \bigr) \hookrightarrow
                H^{1,q}_\Gamma
        \]
        for $\tau \in \bigl] \frac{1 + \varsigma}{2}, 1 \bigr[$, see 
        \cite[Ch.~1.15.2]{triebel}.
        \item The first assertion is clear by Sobolev embedding.
        The second follows from known multiplier results, see
        \cite[Ch.~1.4]{grisvard85} or \cite{mazyamult}.\qedhere
\end{enumerate}
\end{proof}
Next we will consider requirement c), see condition ({\bf B}) in
Proposition~\ref{p-pruess}.
%
\begin{lemma} \label{l-definv}
Let $q$ be a number from Assumption \ref{a-q} and let $X$ be a Banach space
with predual $X_*$ that admits the continuous and dense injections
\begin{equation} \label{e-densinj}
  H^{1,q'}_\Gamma \hookrightarrow X_* \hookrightarrow L^{(\frac {q}{2})'}.
\end{equation}
\begin{enumerate}
\item If $\xi \in H^{1,q}$ is a multiplier on $X_*$, then
        $dom_X(-\nabla \cdot \mu \nabla) \hookrightarrow dom_X(-\nabla \cdot
        \xi \mu \nabla)$.
\item If $H^{1,q}$ is a multiplier space for $X_*$, then the (linear) mapping
        $H^{1,q} \ni \xi \mapsto -\nabla \cdot \xi \mu \nabla \in
        \mathcal{L}(dom_X(- \nabla \cdot \mu \nabla),X)$ is well defined and
        continuous.
\end{enumerate}
\end{lemma}
%
\begin{proof}
The supposition $q > d \ge 2$ and \eqref{e-densinj} imply the existence of a
continuous and dense injection $H^{1,2}_\Gamma \hookrightarrow X_*$. Thus, it
is not hard to see that $\psi$ belongs to $dom_X(-\nabla \cdot \mu \nabla)$
iff the linear form
\[ \varphi \mapsto \int_\Omega \nabla \psi \cdot \mu \nabla {\varphi} \; \dd
        \mathrm x
\]
is continuous on $H^{1,2}_\Gamma$, when $H^{1,2}_\Gamma$ is equipped with the
$X_*$ topology. We denote the set $H^{1,2}_\Gamma \cap \{\varphi \in X_* :
\|\varphi\|_{X_*} = 1\}$ by $\mathcal M$. Assuming $\psi \in
dom_X(-\nabla \cdot \mu \nabla)$, we can estimate
\begin{align}
  \|- \nabla \cdot \xi \mu \nabla \psi \|_X &= \sup_{\varphi \in \mathcal M}
        \biggl| \int_\Omega \xi \mu \nabla \psi \cdot \nabla {\varphi}
        \; \dd \mathrm x \biggr| \nonumber \\
  &\le \sup_{\varphi \in \mathcal M} \biggl| \int_\Omega \nabla \psi \cdot
        \mu \nabla (\xi \varphi) \; \dd \mathrm x \biggr| +
        \sup_{\varphi \in \mathcal M} \biggl| \int_\Omega \nabla \psi \cdot
        \mu \varphi \nabla \xi \; \dd \mathrm x \biggr| \nonumber \\
 &\le \| \psi \|_{dom_X(-\nabla \cdot \mu \nabla)}
        \sup_{\varphi \in \mathcal M} \|\xi \varphi \|_{X_*} + \| \psi
        \|_{H^{1,q}} \|\mu\|_{L^\infty} \| \xi \|_{H^{1,q}}
        \sup_{\varphi \in \mathcal M} \|\varphi \|_{L^{(\frac{q}{2})'}}.
        \label{e-esti001}
\end{align}
We observe that the supposition $H^{1,q'}_\Gamma \hookrightarrow X_*$
together with Assumption~\ref{a-q} leads to the continuous embedding
$dom_X(-\nabla \cdot \mu \nabla) \hookrightarrow H^{1,q}$. Thus,
\eqref{e-esti001} is not larger than
\[ m_\xi \, \| \psi\|_{dom_X(-\nabla \cdot \mu \nabla)} + \| \xi \|_{H^{1,q}}
        \|\mu\|_{L^\infty} \Emb \bigl( dom_X(-\nabla \cdot \mu \nabla),
        H^{1,q} \bigr) \Emb(X_*, L^{(\frac{q}{2})'}) \|\psi
        \|_{dom_X(-\nabla \cdot \mu \nabla)},
\]
where $m_\xi$ denotes the norm of the multiplier on $X_*$ induced by $\xi$ and
$\Emb(\cdot, \cdot)$ stands again for the corresponding embedding constants.

Assertion ii) also results from the estimates in the proof of i).
\end{proof}
\begin{corollary} \label{c-defber}
If $\xi$ additionally to the hypotheses of
Lemma~\ref{l-definv} i) has a
positive lower bound, then
\[ dom_X(-\nabla \cdot \xi \mu \nabla) = dom_X(-\nabla \cdot \mu \nabla).
\]
\end{corollary}
%
\begin{proof}
According to Lemma~\ref{l-definv}~i) one has only to show $dom_X(-\nabla \cdot
\xi \mu \nabla) \hookrightarrow dom_X(-\nabla \cdot \mu \nabla)$. By
Lemma~\ref{l-modi} we have $dom_{H^{-1,q}_\Gamma}(-\nabla \cdot \xi \mu
\nabla) = H^{1,q}_\Gamma$. Thus, one can apply Lemma~\ref{l-definv} to the
situation $\tilde \mu = \xi \mu$ and $\tilde \xi = \frac {1}{\xi}$.
\end{proof}
Next we will show that functions on $\partial \Omega$ or on a Lipschitz
hypersurface, which belong to a suitable summability class, can be understood
as elements of the distribution space $H^{-\varsigma,q}_\Gamma$.
%
\begin{theorem} \label{t-invtrace}
Assume $q \in \left] 1, \infty \right[$, $\varsigma \in \bigl] 1 -
\frac{1}{q}, 1 \bigr[ \setminus \{ \frac{1}{q} \}$ and let $\Pi, \varpi$ be as
in Theorem~\ref{t-embedbound}. Then the adjoint trace operator
$(\mathrm{Tr})^*$ maps $L^q(\Pi)$ continuously into $\bigl(
H^{\varsigma, q'}(\Omega) \bigr)' \hookrightarrow H^{-\varsigma,q}_\Gamma$.
\end{theorem}
%
\begin{proof}
The result is obtained from Theorem~\ref{t-embedbound} by duality.
\end{proof}
%
\begin{remark} \label{r-unglatt}
Here we restricted the considerations to the case of Lipschitz hypersurfaces, since
this is the most essential insofar as it gives the possibility of prescribing
jumps in the normal component of the current $j := \mathcal G(u) \mu \nabla u$
along hypersurfaces where the coefficient function jumps. This case is of high
relevance in view of applied problems and has attracted much attention also
from the numerical point of view, see e.g. \cite{adams}, \cite{cagin} and
references therein.

In fact, it is possible to include much more general sets where distributional right
hand sides live. For the identification of (singular) measures as distribtions
on lower dimensional sets, see also \cite[Ch.~4]{ziemer} and
\cite[Ch.~VI.]{jons}. We did not make explicit use of this here, because
at present we do not see direct applications.
\end{remark}
%
From now on we fix once and for all a number $\varsigma \in \bigl] \max \{ 1 -
\frac{1}{q}, \frac{d}{q} \}, 1 \bigr[$ and set for all what follows $X :=
H^{-\varsigma,q}_\Gamma$.

\medskip

Next we introduce the requirements on the data of problem
\eqref{e-quasiformal}/\eqref{e-boundcond77}.
%
\begin{assumption} \label{a-daten}
\begin{enumerate}
\item[{\bf Op)}] For all what follows we fix a number $s >
        \frac{2}{1-\varsigma}$.
\item[{\bf Su)}] There exists $f \in C^2(\R)$, positive, with strictly
        positive derivative, such that $\mathcal F$ is the superposition
        operator induced by $f$.
\item[{\bf Ga)}] The mapping $\mathcal G: H^{1,q} \to H^{1,q}$ is locally
        Lipschitz continuous.
\item[{\bf Gb)}] For any ball in $H^{1,q}$ there exists $\delta > 0$, such
        that $\mathcal G(u) \ge \delta$ for all $u$ from this ball.
\item[{\bf Ra)}] The function $\mathcal R: J \times H^{1,q} \to X$ is of
	Carath\'eodory type, i.e. $\mathcal R(\cdot,u)$ is measurable for all
	$u \in H^{1,q}$ and $\mathcal R(t,\cdot)$ is continuous for a.a. $t
	\in J$.
\item[{\bf Rb)}] $\mathcal R(\cdot, 0) \in L^s(J;X)$ and for $M > 0$ there
	exists $h_M \in L^s(J)$, such that
        \[ \| \mathcal R(t,u) - \mathcal R(t, \tilde{u}) \| _{X} \le h_M(t) \|
		u - \tilde{u} \|_{H^{1,q}}, \quad t \in J,
        \]
        provided $\max(\|u\|_{H^{1,q}}, \|\tilde{u}\|_{H^{1,q}}) \leq M$.
\item[{\bf BC)}] $b$ is an operator of the form $b(u)= Q(b_\circ(u))$, where
	$b_\circ$ is a (possibly nonlinear), locally Lipschitzian operator from
	$C(\overline \Omega)$ into itself (see Lemma \ref{l-relativ}).
\item[{\bf Gg)}] $g \in L^q(\Gamma)$.
\item[{\bf IC)}] $u_0 \in (X, dom_X(-\nabla \cdot \mu \nabla))_{1-\frac{1}{s},s}$.
\end{enumerate}
\end{assumption}
%
%
\begin{remark} \label{r-gterm}
At the first glance the choice of $s$ seems indiscriminate. The point is,
however, that generically in applications the explicit time dependence of the
reaction term $\mathcal R$ is essentially bounded. Thus, in view of condition
{\bf Rb)} it is justified to take $s$ as any arbitrarily large number, whose magnitude
 needs not to be controlled explicitely, see Example~\ref{ex-zerlegung}.

Note that the requirement on $\mathcal G$ allows for
nonlocal operators. This is essential if the current depends on an additional
potential governed by an auxiliary equation, what is usually the case in
drift-diffusion models, see \cite{amann}, \cite{gaj/groe95} or
\cite{selberherr84}.

The conditions~{\bf Ra)} and {\bf Rb)} are always satisfied if $\mathcal R$ is
a mapping into $L^{q/2}$ with the analog boundedness and continuity properties,
see Lemma~\ref{l-space}~ii).

The estimate in \eqref{e-Qesti} shows that $Q$ in fact is well defined on $C(\overline \Omega)$, therefore condition {\bf BC)} makes sense, see also
\eqref{e-hoelderemb}. In particular, $b_\circ$ may be a superposition operator,
induced by a $C^1(\R)$ function. Let us emphasize that in this case the
inducing function needs not to be positive. Thus, non-dissipative boundary
conditions are included.

Finally, the condition {\bf IC)} is an 'abstract' one and hardly to verify,
because one has no explicit characterization of
$(X,dom_X(-\nabla \cdot \mu \nabla))_{1-\frac{1}{s},s}$ at hand. Nevertheless,
the condition is reproduced along the trajectory of the solution by means of
the embedding \eqref{e-embedcont}.
\end{remark}
%
In order to solve \eqref{e-quasiformal}/\eqref{e-boundcond77}, we will
consider instead \eqref{5.2} with
\begin{equation} \label{e-opdef}
  \mathcal B(u) := -\nabla \cdot \frac{\mathcal G(u)}{\mathcal F'(u)} \mu
        \nabla
\end{equation}
and the right hand side $\mathcal S$
\begin{equation} \label{e-quasiformal1}
  \mathcal S(t,u) := \frac{\mathcal R(t,u)}{\mathcal F'(u)} + \Bigl( \nabla
        \frac{1}{\mathcal F'(u)} \Bigr) \cdot \Bigl( \mathcal G(u) \mu \nabla
        u \Bigr) - \frac{Q(b_\circ(u))}{\mathcal F'(u)} +
        \frac{(\mathrm{Tr})^*g}{\mathcal F'(u)},
\end{equation}
seeking the solution in the space $W^{1,s}(J;X) \cap L^s(J;dom_X(-\nabla \cdot
\mu \nabla))$.
%
\begin{remark} \label{r-justify}
Let us explain this reformulation: as is well known in the theory of  boundary
value problems, the boundary condition \eqref{e-boundcond77} is incorporated
by introducing the boundary terms $-\varkappa b_\circ(u)$ and $g$ on the right
hand side. In order to understand both as elements from $X$, we write
$Q(b_\circ(u))$ and $(\mathrm{Tr})^*g$, see Lemma \ref{l-relativ} and
Theorem~\ref{t-invtrace}. On the other hand, our aim was to eliminate the
nonlinearity under the time derivation: we formally differentiate
$(\mathcal F(u))'=\mathcal F'(u)u'$ and afterwards divide the whole equation
by $\mathcal F'(u)$. Finally, we employ the equation
\begin{equation} \label{e-umform}
  - \frac{1}{\mathcal F'(u)} \nabla \cdot {\mathcal G(u)} \mu \nabla u =
        -\nabla \cdot \frac{\mathcal G(u)}{\mathcal F'(u)} \mu \nabla u -
        \Bigl( \nabla \frac{1}{\mathcal F'(u)} \Bigr) \cdot \Bigl(
        \mathcal G(u) \mu \nabla u \Bigr),
\end{equation}
which holds for any $u \in dom_X(-\nabla \cdot \mathcal G(u) \mu \nabla) =
dom_X(-\nabla \cdot \mu \nabla)$ as an equation in $X$, compare Lemma \ref{l-space} ii) 
and Corollary~\ref{c-defber}.
\end{remark}
%
%
\begin{theorem} \label{t-existence}
Let Assumption~\ref{a-q} be satisfied and assume that
the data of the problem satisfy Assumption~\ref{a-daten}. Then \eqref{5.2}
has a local in time, unique solution in $W^{1,s}(J;X) \cap L^s(J;dom_X(-\nabla
\cdot \mu \nabla))$, provided that $\mathcal B$ and $\mathcal S$ are given by
\eqref{e-opdef} and \eqref{e-quasiformal1}, respectively.
\end{theorem}
%
\begin{proof}
First of all we note that, due to {\bf Op)}, $1 - \frac{1}{s} >
\frac{1+\varsigma}{2}$. Thus, if $\tau \in ]\frac{1+\varsigma}{2},1 - \frac{1}{s}[$
by a well known interpolation result (see
\cite[Ch.~1.3.3]{triebel}) and Lemma~\ref{l-space}~i) we have
\begin{equation} \label{e-embeddin}
  (X, dom_X(-\nabla \cdot \mu \nabla))_{1-\frac{1}{s},s} \hookrightarrow
        (X, dom_X(-\nabla \cdot \mu \nabla))_{\tau ,1}
        \hookrightarrow H^{1,q}.
\end{equation}
Hence, by {\bf IC)}, $u_0 \in H^{1,q}$. Consequently, due to the suppositions
on $\mathcal F$ and $\mathcal G$, both the functions
$\frac{\mathcal G(u_0)}{\mathcal F'(u_0)}$ and
$\frac{\mathcal F'(u_0)}{\mathcal G(u_0)}$ belong to $H^{1,q}$ and are bounded
from below by a positive constant. Denoting $-\nabla \cdot
\frac{\mathcal G(u_0)}{\mathcal F'(u_0)} \mu \nabla$ by $B$,
Corollary~\ref{c-defber} gives $dom_X(-\nabla \cdot \mu \nabla) = dom_X(B)$. 
This implies $u_0 \in (X, dom_X(B))_{1-\frac{1}{s},s}$. Furthermore, the so
defined $B$ has maximal parabolic regularity on $X$, thanks to
\eqref{e-interpolXc0} in Theorem~\ref{t-randterm} with $p=q$.

Condition ({\bf B}) from Proposition~\ref{p-pruess} is implied by
Lemma~\ref{l-definv}~ii) in cooperation with ii) of Lemma~\ref{l-space}, the
fact that the mapping $H^{1,q} \ni \phi \mapsto \frac{\mathcal G(\phi)}
{\mathcal F'(\phi)} \in H^{1,q}$ is boundedly Lipschitz and \eqref{e-embeddin}.

It remains to show that the 'new' right hand side $\mathcal S$ satisfies condition
({\bf R}) from Proposition~\ref{p-pruess}. We do this for every term in \eqref{e-quasiformal1}
 separately, beginning from the left: concerning the first, one again uses
\eqref{e-embeddin}, the asserted conditions {\bf Ra)} and {\bf Rb)} on $\mathcal R$,
the local Lipschitz continuity of the mapping $H^{1,q} \ni u \mapsto
\frac {1}{\mathcal F'(u)} \in H^{1,q}$ and the fact that $H^{1,q}$ is a
multiplier space over $X$. The second term can be treated in the same spirit, if
one takes into account the embedding $L^{q/2} \hookrightarrow X$ and applies
H\"older's inequality. The assertion for the last two terms results
from \eqref{e-embeddin}, the assumptions {\bf BC)}/{\bf Gg)}, Lemma \ref{l-relativ} and
Theorem~\ref{t-invtrace}.
\end{proof}
%
%
\begin{remark} \label{r-justi}
According to \eqref{e-umform} it is clear that the solution $u$ satisfies the
equation
\begin{equation} \label{e-backtransf}
  \mathcal F'(u) u' - \nabla \cdot \mathcal G(u) \mu \nabla u + Q(b_\circ(u))
	= \mathcal R(t,u) + (\mathrm{Tr})^*g
\end{equation}
as an equation in $X$. Note that, if $\mathcal R$ takes its values only
in the space $L^{q/2} \hookrightarrow X$, then -- in the light of
Lemma~\ref{l-relativ} -- the elliptic operators incorporate the boundary
conditions \eqref{e-boundcond77} in a generalized sense, see
\cite[Ch.~II.2]{ggz} or \cite[Ch.~1.2]{cia}.
\end{remark}
%
The remaining problem is to identify $\mathcal F'(u) u'$ with $\bigl(
\mathcal F(u) \bigr)'$ where the prime has to be understood as the
distributional derivative with respect to time. This identification (technically rather involved) is proved in \cite{hie/reh} for the case where
the Banach space $X$ equals $L^{q/2}$, but can be carried over to the case $X =
H^{-\varsigma,q}_\Gamma$ -- word by word.

We will now show that the solution $u$ is H\"older continuous simultaneously
in space and time, even more:
%
\begin{corollary} \label{c-hoel}
There exist $\alpha, \beta > 0$ such that the solution $u$ of
\eqref{e-backtransf} belongs to the space $C^\beta(J; H^{1,q}_\Gamma(\Omega))
\hookrightarrow C^\beta(J; C^\alpha(\Omega))$.
\end{corollary}
%
\begin{proof}
During this proof we write for short $D := dom_X(B)$. A straightforward
application of H\"older's inequality yields the embedding
\[ W^{1,s}(J; X) \hookrightarrow C^\delta ( J; X ) \quad \text {with} \quad
        \delta = 1- \frac {1}{s}.
\]
Take $\lambda$ from the interval $\bigl] \frac{1+\varsigma}{2} \bigl( 1 -\frac{1}{s} \bigr)^{-1},
 1 \bigr[$, which is nonempty in view of {\bf Op)}. Using Lemma~\ref{l-space}~i) and the reiteration
 theorem for real interpolation, one can estimate
\begin{align*}
  \frac{\| u(t_1) - u(t_2) \|_{H^{1,q}}}{|t_1 - t_2|^{\delta (1-\lambda)}} &\le
        c \,\frac{\| u(t_1) - u(t_2) \|_{(X, D)_{\lambda (1-\frac {1}{s}),1 }}}
        {|t_1 - t_2|^{\delta (1-\lambda)}} \le c \,
        \frac{\| u(t_1) - u(t_2) \|_{(X,(X,D)_{1-\frac {1}{s},s})_{\lambda,1}}}
        {|t_1 - t_2|^{\delta(1- \lambda)}} \\
  &\le c \, \frac{\| u(t_1) - u(t_2) \|^{1-\lambda}_X}
        {|t_1 - t_2|^{\delta(1- \lambda)}} \, \| u(t_1) - u(t_2)
        \|^{\lambda}_{(X,D)_{1-\frac {1}{s},s}} \\
  &\le c \, \Bigl( \frac{\|u(t_1) - u(t_2) \|_{X}}{|t_1-t_2|^\delta}
        \Bigr)^{1-\lambda} \, \Bigl( 2 \sup_{t \in J} \|u(t)
        \|_{(X,D)_{1-\frac {1}{s},s}} \Bigr)^\lambda.
\end{align*}
\end{proof}
Finally, we will have a closer look at the semilinear case. It turns out that one can achieve
 satisfactory results here without Assumption \ref{a-q},
at least when the nonlinear term depends only on the 
function itself and not on its gradient.
%
\begin{theorem} \label{t-semil}
Assume that $-\nabla \cdot \mu \nabla$ satisfies maximal parabolic regularity
on $H^{-1,q}_\Gamma$ for some $q>d$. Suppose further that the function
$\mathcal R : J \times C(\overline \Omega) \to H^{-1,q}_\Gamma$ is of
Carath\'eodory type, i.e. $\mathcal R(\cdot,u)$ is measurable for all $u \in
C(\overline \Omega)$ and $\mathcal R(t,\cdot)$ is continuous for a.a. $t \in
J$ and, additionally, obeys the following condition: $\mathcal R(\cdot, 0) \in
L^s(J;H^{-1,q}_\Gamma)$ and for all $M > 0$ there exists $h_M \in L^s(J)$, such that
\[ \| \mathcal R(t,u) - \mathcal R(t, \tilde{u}) \| _{H^{-1,q}_\Gamma} \le
	h_M(t) \| u - \tilde{u} \|_{C(\overline \Omega)}, \quad t \in J.
\]
Then the equation
\[ u' -\nabla \cdot \mu \nabla u = \mathcal R(t,u) ,\quad \quad u(T_0) = 0
\]
admits exactly one local in time solution.
\end{theorem}
%
\begin{proof}
It is clear that $\mathcal R$ satisfies the abstract conditions on the
reaction term, posed in Proposition~\ref{p-pruess}, if we can show
$[H^{-1,q}_\Gamma,dom_{H^{-1,q}_\Gamma}(-\nabla \cdot \mu \nabla)]_\theta
\hookrightarrow C(\overline \Omega)$ for some large $\theta \in \left]
0, 1 \right[$. This we will do: using the embedding
$dom_{H^{-1,q}_\Gamma}(-\nabla \cdot \mu \nabla) \hookrightarrow C^\alpha$ for
some positive $\alpha$ (see \cite{griehoel}) and the reiteration theorem for
complex interpolation, one can write
\begin{align*}
  [H^{-1,q}_\Gamma,dom_{H^{-1,q}_\Gamma}(-\nabla \cdot \mu \nabla)]_\theta
	&= \bigl[ [H^{-1,q}_\Gamma,
	dom_{H^{-1,q}_\Gamma}(-\nabla \cdot \mu \nabla) ]_\frac{1}{2},
	dom_{H^{-1,q}_\Gamma}(-\nabla \cdot \mu \nabla) \bigr]_{2 \theta -1} \\
  &\hookrightarrow \bigl[ [H^{-1,2}_\Gamma,H^{1,2}_\Gamma ]_\frac{1}{2},
	C^\alpha \bigr]_{2 \theta -1} = [L^2,C^\alpha]_{2 \theta -1}.
\end{align*}
But based on the results of Triebel \cite{triebelc}, in \cite[Ch.~7]{gkr} it
is shown that this last space continuously embeds into another H\"older space,
if $\theta $ is chosen large enough.
\end{proof}
%
%
%
%
\section{Examples} \label{sec-Examples}
%
%
%
%
\noindent
In this section we describe geometric configurations for which our
Assumption~\ref{a-q} holds true and we present concrete examples of mappings
$\mathcal G$ and reaction terms $\mathcal R$ fitting into our framework. Another part
of this section is then devoted to the special geometry of two crossing beams
that is interesting, since this is not a domain with Lipschitz boundary, but
it falls into the scope of our theory, as we will show.
\subsection{Geometric constellations}
While our results in Sections~\ref{sec-WurzelIso} and \ref{sec-Consequences}
on the square root of $-\nabla \cdot \mu \nabla$ and maximal parabolic
regularity are valid in the general geometric framework of
Assumption~\ref{a-groegerregulaer}, we additionally had to impose
Assumption~\ref{a-q} for the treatment of quasilinear equations in
Section~\ref{sec-Quasilin}. Here we shortly describe geometric constellations,
in which this additional condition is satisfied.

Let us start with the observation that the 2-$d$ case is covered by
Remark~\ref{r-darunter}~i).
%

%
Admissible three-dimensional settings may be described as follows.
%
\begin{proposition}\label{p-geo3d}
Let $\Omega \subseteq \R^3$ be a bounded Lipschitz domain. Then there exists a
$q > 3$ such that $- \nabla \cdot \mu \nabla +1$ is a topological isomorphism
from $H_{\Gamma}^{1,q}$ onto $H_{\Gamma}^{-1,q}$, if one of the following
conditions is satisfied:
\begin{enumerate}
\item $\Omega $ has a Lipschitz boundary. $\Gamma = \emptyset$ or $\Gamma
        = \partial \Omega$. $\Omega_\circ \subseteq \Omega$ is another domain
        which is $C^1$ and which does not touch the boundary of $\Omega$.
        $\mu|_{\Omega_\circ} \in BUC(\Omega_\circ)$ and
        $\mu|_{\Omega \setminus \overline{\Omega_\circ}} \in
        BUC(\Omega \setminus \overline{\Omega_\circ})$.
\item $\Omega$ has a Lipschitz boundary. $\Gamma = \emptyset$.
        $\Omega_\circ \subseteq \Omega$ is a Lipschitz domain, such that
        $\partial \Omega_\circ \cap \Omega$ is a $C^1$ surface and 
        $\partial \Omega$ and $\partial \Omega_\circ $ meet suitably (see
	\cite{e/r/s} for details).
        $\mu|_{\Omega_\circ} \in BUC(\Omega_\circ)$ and
        $\mu|_{\Omega \setminus \overline{\Omega_\circ}} \in
        BUC(\Omega \setminus \overline{\Omega_\circ})$.
\item $\Omega$ is a three dimensional Lipschitzian polyhedron. $\Gamma =
        \emptyset$. There are hyperplanes $\mathcal H_1, \dots,\mathcal H_n$
        in $\R^3$  which meet at most in a vertex of the polyhedron such that
        the coefficient function $\mu$ is constantly a real, symmetric,
        positive definite $3 \times 3$ matrix on each of the connected
        components of $\Omega \setminus \cup_{l=1}^n \mathcal H_l$. Moreover,
        for every edge on the boundary, induced by a hetero interface
        $\mathcal H_l$, the angles between the outer boundary plane and the
        hetero interface do not exceed $\pi$ and at most one of them may equal
        $\pi$.
\item $\Omega$ is a convex polyhedron, $\overline{\Gamma} \cap
        (\partial \Omega \setminus \Gamma)$ is a finite union of line
        segments. $\mu \equiv 1$.
\item $\Omega \subseteq \R^3$ is a prismatic domain with a triangle as basis.
        $\Gamma$ equals either one half of one of the rectangular sides or one
        rectangular side or two of the three rectangular sides. There is a
        plane which intersects $\Omega$ such that the coefficient function
        $\mu$ is constant above and below the plane.
\item $\Omega$ is a bounded domain with Lipschitz boundary.
        Additionally, for each  $x \in \overline{\Gamma} \cap
        (\partial\Omega\setminus \Gamma)$ the mapping $\phi_{\mathrm x}$
        defined in Assumption~\ref{a-groegerregulaer} is a
        $C^1$-diffeomorphism from $\Upsilon_{\mathrm x}$ onto its image. $\mu
        \in BUC(\Omega)$.
\end{enumerate}
\end{proposition}
%
\noindent
The assertions i) and ii) are shown in \cite{e/r/s}, while iii) is proved in \cite{e/k/r/s} and iv)
 is a result of Dauge \cite{dauge}. Recently, v) was
obtained in \cite{hall} and vi) will be published in a forthcoming paper.
\qed
%
\begin{corollary} \label{c-lokal}
The assertion remains true, if there is a finite open covering $\Upsilon_1,
\ldots, \Upsilon_l$ of $\overline \Omega$, such that each of the pairs
$\Omega_j := \Upsilon_j \cap \Omega$, $\Gamma_j := \Gamma \cap \Upsilon_j$
fulfills one of the points i) -- vi).
\end{corollary}
%
\begin{proof}
The corollary can be proved by means of Lemma~\ref{l-project} and
Lemma~\ref{l-restr/ext}.
\end{proof}
%
\begin{remark} \label{r-globalsetting}
Proposition~\ref{p-geo3d} together with Corollary~\ref{c-lokal} provides a
huge zoo of geometries and boundary constellations, for which $-\nabla \cdot
\mu \nabla$ provides the required isomorphism. We intend to complete this in
the future. 
\end{remark}
%
%
\subsection{Nonlinearities and reaction terms}
The most common case is that where $\mathcal F$ is the exponential or the
Fermi-Dirac distribution function $\mathcal F_{1/2}$ given by
\[ \mathcal F_{1/2}(t) := \frac{2}{\sqrt{\pi}} \; \int_0^\infty
        \frac{\sqrt{s}}{1 + \e^{s-t}} \; \dd s
\]
and $\mathcal G$ also is a Nemytzkii operator of the same type. In phase
separation problems, a rigorous formulation as a minimal problem for the free
energy reveals that $\mathcal{G} = \mathcal{F}^\prime$ is appropriate.
This topic has been thoroughly investigated in \cite{quastel:92},
\cite{quastel:99}, \cite{lebowitz:97}, and \cite{lebowitz:98}, see also
\cite{skrypnik:02} and \cite{griepentrog04}. It is noteworthy that in this
case $\frac{\mathcal G}{\mathcal F'} \equiv 1$ (we conjecture that this is not
accidental) and the evolution equation \eqref{e-quasiformal} leads not to a
quasilinear equation \eqref{5.2} but to one which is only semilinear. We
consider this as a hint for the adequateness of our treatment of the parabolic
equations.

As a second example we present a nonlocal operator arising in the diffusion
of bacteria; see \cite{chipot1}, \cite{chipot2} and references therein.
%
\begin{example}
Let $\eta$ be a continuously differentiable function on $\R$ which is bounded
from above and below by positive constants. Assume $\varphi \in L^2(\Omega)$
and define
\[ \mathcal G(u) := \eta \biggl( \int_\Omega u \varphi \; \dd \mathrm x
	\biggr), \quad u \in H^{1,q}.
\]
\end{example}
%
Now we give two examples for mappings $\mathcal R$.
%
\begin{example} \label{ex-zerlegung}
Assume that $\left[ T_0, T \right[ = \cup_{l=1}^j \left[ t_l, t_{l+1} \right[$
is a (disjoint) decomposition of $\left[ T_0, T \right[$ and let for $l
\in \{1, \dots, j \}$
\[ Z_l : \R \times \R^{d} \to \R
\]
be a function which satisfies the following condition: For any compact set $\mathcal K
\subseteq \R$ there is a constant $L_\mathcal K$ such that for any $a, \tilde a \in \mathcal K,
\; b, \tilde b \in \R^d$ the inequality
\[ |Z_l(a, b) - Z_l(\tilde a, \tilde b)| \le L_K |a - \tilde a|_\R \; \bigl(
        |b|_{\R^{d}}^2 + |\tilde b|_{\R^{d}}^2 \bigr) + L_K |b - \tilde b
        |_{\R^{d}} \; \bigl( |b|_{\R^{d}} + |\tilde b|_{\R^{d}} \bigr)
\]
holds. We define a mapping $Z : \left[ T_0, T \right[ \times \R \times \R^{d}
\to \R$ by setting
\[ Z(t, a, b) := Z_l(a, b), \quad \text{if} \quad t \in [t_l,t_{l+1}[.
\]
The function $Z$ defines a mapping $\mathcal R : \left[ T_0, T \right[ \times
{H}^{1,q} \to L^{q/2}$ in the following way: If $\psi$ is the restriction of an
$\R$-valued, continuously differentiable function on $\R^d$ to $\Omega$, then
we put
\[ \mathcal R(t, \psi)(x) = Z(t, \psi(x), (\nabla \psi)(x)) \quad \mbox{for } x
	\in \Omega
\]
and afterwards extend $\mathcal R$ by continuity to the whole set $\left[ T_0, T
\right[ \times H^{1,q}$.
\end{example}
%
%
\begin{example} \label{exam5.2}
Assume $\iota : \R \to \left] 0, \infty \right[$ to be a continuously
differentiable function. Furthermore, let $\mathcal T : H^{1,q} \to H^{1,q}$ be
the mapping which assigns to $v \in H^{1,q}$ the solution $\varphi$ of the
elliptic problem (including boundary conditions)
\begin{equation} \label{e-diri}
  -\nabla \cdot \iota (v) \nabla \varphi = 0.
\end{equation}
If one defines
\[ 
\mathcal R(v) = \iota (v) |\nabla (\mathcal T(v))|^2,
\]
then, under reasonable suppositions on the data of \eqref{e-diri}, the mapping
$\mathcal R$ satisfies Assumption~{\bf Ra)}.
\end{example}
%
This second example comes from a model which describes electrical heat
conduction; see \cite{chipot} and the references therein.
\subsection{An unorthodox example: two crossing beams} \label{subsec-balks}
Finally, we want to present in some detail the example of two beams, mentioned
in the introduction, which is not a domain with Lipschitz boundary, and,
hence, not covered by former theories. Consider in $\R^3$ the set
\[ B_{\Join} := \left] -10, 10 \right[ \times \left] -1, 1 \right[ \times
        \left] -2, 0 \right[ \quad \cup \quad \left] -1, 1 \right[ \times
        \left] -10, 10 \right[ \times \left] 0, 2 \right[ \quad \cup \quad
\left] -1, 1 \right[ \times \left] -1, 1 \right[ \times \{ 0\},
\]
together with a $3\times 3$ matrix $\mu_1$, considered as the coeffcient
matrix on the first beam, and another $3\times 3$ matrix $\mu_2$, considered
as the coefficient function on the other beam. Both matrices are assumed to be
real, symmetric and positive definite. If one defines the coefficient function
$\mu$ as $\mu_1$ on the first beam, and as $\mu_2$ on the other, then,
due to Proposition~\ref{p-geo3d}~iii),
\[ -\nabla \cdot \mu \nabla : H^{1,q}_0 \to H^{-1,q}
\]
provides a topological isomorphism for some $q > 3$, if one can show that
$B_{\Join}$ is a Lipschitz domain. In fact, we will show more, namely:
%
\begin{lemma} \label{l-balks} 
$B_{\Join}$ fulfills Assumption~\ref{a-groegerregulaer}.
\end{lemma}
%
\begin{proof}
For all points $\mathrm x \in \partial \Omega$ the existence of a corresponding
neighborhood $\Upsilon_\mathrm x$ and a mapping $\Phi_\mathrm x$ can be
deduced easily, except for the points $\mathrm x$ from the set 
\[ \mathrm{Sing} := \{ (-1,-1,0), (-1,1,0), (1,-1,0), (1,1,0) \}.
\]
In fact, for all points $\mathrm x \in B_{\Join} \setminus \mathrm{Sing}$ there
is a neighborhood $\Upsilon_\mathrm x$, such that either $B_{\Join} \cap
\Upsilon_\mathrm x$ or $\Upsilon_\mathrm x \setminus B_{\Join}$ is convex and,
hence, a domain with Lipschitz boundary. Thus, these points can be treated as
in Remark~\ref{r-groegerreg}.

Exemplarily, we aim at a suitable transformation in a neighborhood of the
point $(1,-1,0)$; the construction for the other three points is -- mutatis
mutandis -- the same. For doing so, we first shift $B_{\Join}$ by the vector
$(-1,1,0)$, so that the transformed point of interest becomes the origin. Now
we apply the transformation $\phi_\blacktriangle$ on $\R^3$ that is given in
Figure~\ref{fig-Balken1}.
\begin{figure}[htbp]
  \includegraphics[scale=0.6]{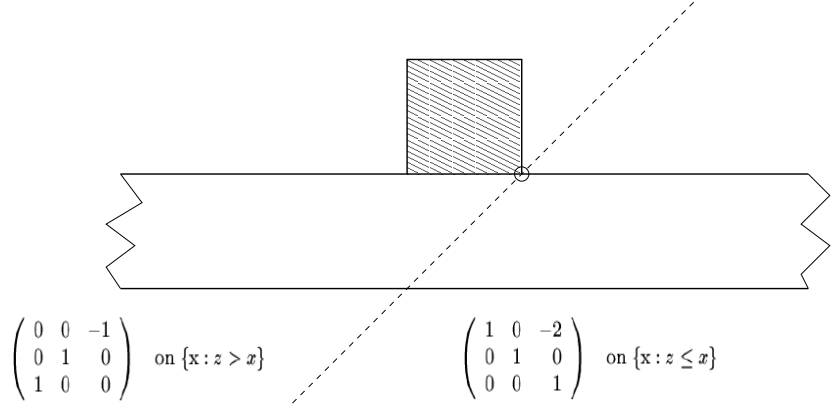}
  \caption{\label{fig-Balken1} Cut through $B_{\Join} + (-1,1,0)$ at a plane $y
      = \delta$ (for $\delta > 0$ small) and the transformation
	$\phi_{\blacktriangle}$}
\end{figure}
The following is straighforward to verify:
\begin{itemize}
\item Both transformations coincide on the plane $\{\mathrm x \with z = x \}$
      and thus together define a globally bi-Lipschitz mapping
	$\phi_\blacktriangle: \R^3 \to \R^3$, which, additionally, is
	volume-preserving.
\item The intersection of $\phi_\blacktriangle \bigl( B_{\Join} + (-1,1,0)
      \bigr)$ with a sufficiently small, paraxial cube $\epsilon K$ around
      $\mathrm 0$ equals the set 
      \[ \{ \mathrm x \with -\epsilon < x < 0,\ -\epsilon < y <
             \epsilon,\ -\epsilon < z < 0 \} \cup \{ \mathrm x \with 0 \le x <
             \epsilon,\ 0 < y < \epsilon,\ -\epsilon < z < 0 \}.
      \]
\end{itemize}
(To prove the latter, note that the $y$-component is left invariant under
$\phi_\blacktriangle$ and that $\phi_\blacktriangle$ acts in the plane $y =
0$ as follows: the vector $(0,1) $ is mapped onto $(-1,0)$ and the vector
$(-1,0)$ onto $(0,-1)$. Finally, the vector $(1,0)$ is left invariant.)
Next we introduce the mapping $\phi_\triangle$ which is defined as the linear mapping
$\left( \begin{array}{rrr}
	2  & 1 &0\\
	-1 & 0 &0\\
	0  &0  &1
\end{array} \right)$ on the set $\{\mathrm x \with -x < y \}$ and as the
identity on the set $\{\mathrm x \with -x \ge y \}$, see
Figure~\ref{fig-Balken2}.

\begin{figure}[htbp]
  \includegraphics[scale=0.6]{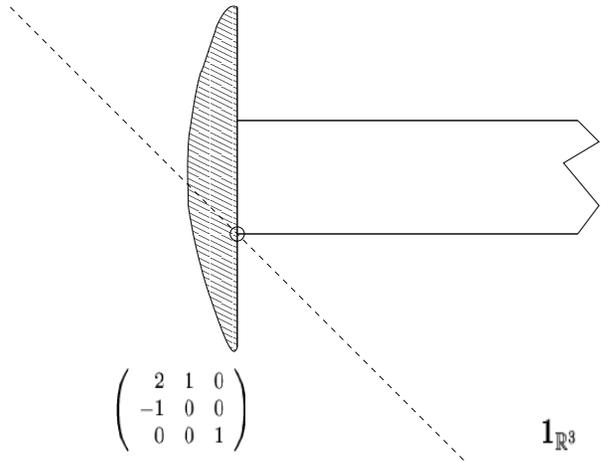}
  \caption{\label{fig-Balken2} Cut through $\phi_\blacktriangle \bigl(
        B_{\Join} + (-1,1,0) \bigr)$ at a plane $z = -\delta$ in a
        neighborhood of $\mathrm 0$ ($\delta > 0$ sufficiently small)}
\end{figure}

One directly verifies that
$\left( \begin{array}{rrr}
	2  & 1 &0\\
	-1 & 0 &0\\
	0  &0  &1
\end{array} \right)$ acts as the identity on the set $\{\mathrm x \with -x = y
\}$; thus $\phi_\triangle$ in fact is a bi-Lipschitz, volume-preserving
mapping from $\R^3$ onto itself. After this transformation the resulting
object, intersected with a sufficiently small paraxial cube $\epsilon K$,
equals the convex set
\[
\{\mathrm x \with -\epsilon <x<\epsilon, 0 < y <\epsilon, -\epsilon < z <0 \}.
\] 
Here again  Remark~\ref{r-groegerreg} applies, what finishes the proof.
\end{proof}
%
%
%
%
%
\section{Concluding Remarks} \label{sec-Remarks}
%
%
\begin{remark} \label{r-conc8}
The reader may have asked himself why we restricted the considerations to
real, symmetric coefficient functions $\mu$. The answer is twofold: first,
we need at all costs Gaussian estimates for our techniques and it is known
that these are not available for complex coefficients in general, see \cite{ACT96}
 and also \cite{Dav97}.
Additionally, Proposition \ref{p-ausch/tcha} also rests on this supposition.
On the other hand, in the applications we have primarly in mind this condition
is satisfied.
\end{remark}
%
%
%
%
\begin{remark} \label{r-conc6}
Under the additional Assumption \ref{a-q}, Theorem~\ref{t-central} implies
maximal parabolic regularity for $-\nabla \cdot \mu \nabla$ on
$H^{-1,q}_\Gamma$ for every $q \in [2,\infty[$, as in the 2-$d$ case.

Besides, the question arises whether the limitation for the exponents, caused
by the localization procedure, is principal in nature or may be overcome when
applying alternative ideas and techniques (cf. Theorem \ref{t-mainsectdown}).
We do not know the answer at present.
\end{remark}
%
%
\begin{remark} \label{r-conc1}
We considered here only the case of one single parabolic equation, but
everything can be carried over in a straightforward way to the case of
diagonal systems; 'diagonal' in this case means that the function $\mathcal G$
is allowed to depend on the vector $u = (u_1, \dots, u_n)$ of solutions and
the right hand side also. In the same spirit one can treat triagonal systems.
\end{remark}
%
%
\begin{remark} \label{r-conc2}
Inspecting Proposition~\ref{p-pruess}, one easily observes that in fact an
additional $t$-dependence of the function $\mathcal G$ would be admissible. We
did not carry this out here for the sake of technical simplicity.
\end{remark}
%
%
\begin{remark} \label{r-conc3}
In \eqref{e-boundcond77} we restricted our setting to the case where the Dirichlet
boundary condition is homogeneous. It is straightforward to generalize this to
the case of inhomogeneous Dirichlet conditions by splitting off the
inhomogeneity, see \cite[Ch.~II.2]{ggz} or \cite[Ch.~1.2]{cia}, see also
\cite{hie/reh} where this has been carried out in detail in the case of
parabolic systems.
\end{remark}
%
%
\begin{remark} \label{r-conc4}
If one knows a priori that the right hand side of \eqref{e-quasiformal}
depends H\"older continuously on the time variable $t$, then one can use other
local existence and uniqueness results for abstract parabolic equations, see
e.g. \cite{luna} for details. In this case the solution $u$ is even strongly
differentiable in the space $X$ (with continuous derivative), what may lead to
a better justification of time discretization then, compare \cite{ashy} and
references therein.
\end{remark} 
%
%
\begin{remark} \label{r-conc5}
Let us explicitely mention that Assumption~\ref{a-q} is not always fulfilled
in the 3-$d$ case. First, there is the classical counterexample of Meyers, see
\cite{mey}, a simpler (and somewhat more striking) one is constructed in
\cite{e/k/r/s}, see also \cite{e/r/s}. The point, however, is that not the
mixed boundary conditions are the obstruction but a somewhat 'irregular'
behavior of the coefficient function $\mu$ in the inner of the domain. If one
is confronted with this, spaces with weight may be the way out.
\end{remark}
%
%
\begin{remark} \label{r-conc7}
In two and three space dimensions one can give the following simplifying
characterization for a set $\Omega \cup \Gamma$ to be regular in the sense of
Gr\"oger, i.e. to satisfy
Assumption~\ref{a-groegerregulaer}~a), see \cite{HalMeyRehb}:

If $\Omega \subseteq \R^2$ is a bounded Lipschitz domain and $\Gamma \subseteq
\partial \Omega$ is relatively open, then $\Omega \cup \Gamma $ is regular in
the sense of Gr\"oger iff $\partial \Omega \setminus  \Gamma$ is the finite
union of (non-degenerate) closed arc pieces.

In $\R^3$ the following
characterization can be proved, heavily resting on a deep result of
Tukia \cite{tukia}:

If $\Omega \subset \R^3$ is a Lipschitz domain and $\Gamma \subset \partial
\Omega$ is relatively open, then $\Omega \cup \Gamma$ is regular in the sense
of Gr\"oger iff the following two conditions are satisfied:
\renewcommand{\labelenumi}{\roman{enumi})}
\begin{enumerate}
\item $\partial \Omega \setminus \Gamma$ is the closure of its interior
	(within $\partial \Omega$).
\item for any ${\mathrm x} \in \overline \Gamma \cap (\partial \Omega
	\setminus \Gamma)$ there is an open neighborhood $\mathcal U \ni
	{\mathrm x}$ and a bi-Lipschitz mapping $\kappa: \mathcal U \cap
	\overline \Gamma \cap (\partial \Omega \setminus \Gamma) \to \left]
	-1, 1 \right[$.
\end{enumerate}
\end{remark}
%

\end{document}